\documentclass[12pt]{article}

\usepackage{amsmath, amsthm, amssymb}
\usepackage{enumerate}
\usepackage{pdflscape}
\usepackage{caption}
\usepackage{bm}

\usepackage{ifpdf}
\ifpdf
\usepackage[pdftex]{graphicx}
\else
\usepackage[dvips]{graphicx}
\fi
\usepackage{tikz}
 	 \usetikzlibrary{arrows,backgrounds}
\usepackage[all]{xy}

\usepackage{multicol}

\usepackage{tocvsec2}

\usepackage{bbm}

\input xy
\xyoption{all}

\usepackage[pdftex,plainpages=false,hypertexnames=false,pdfpagelabels]{hyperref}
\newcommand{\arxiv}[1]{\href{http://arxiv.org/abs/#1}{\tt arXiv:\nolinkurl{#1}}}
\newcommand{\arXiv}[1]{\href{http://arxiv.org/abs/#1}{\tt arXiv:\nolinkurl{#1}}}
\newcommand{\doi}[1]{\href{http://dx.doi.org/#1}{{\tt DOI:#1}}}

\newcommand{\googlebooks}[1]{(preview at \href{http://books.google.com/books?id=#1}{google books})}

\usepackage{xcolor}
\definecolor{dark-red}{rgb}{0.7,0.25,0.25}
\definecolor{dark-blue}{rgb}{0.15,0.15,0.55}
\definecolor{medium-blue}{rgb}{0,0,.8}
\definecolor{DarkGreen}{RGB}{0,150,0}
\definecolor{rho}{named}{red}
\hypersetup{
   colorlinks, linkcolor={purple},
   citecolor={medium-blue}, urlcolor={medium-blue}
}

\usepackage{longtable}
\usepackage{fullpage}

\theoremstyle{plain}
\newtheorem{thm}{Theorem}[section]
\newtheorem*{thm*}{Theorem}
\newtheorem{thmalpha}{Theorem}

\newtheorem{cor}[thm]{Corollary}
\newtheorem{coralpha}[thmalpha]{Corollary}
\newtheorem*{cor*}{Corollary}

\newtheorem*{conj*}{Conjecture}
\newtheorem{lem}[thm]{Lemma}

\newtheorem{prop}[thm]{Proposition}

\newtheorem*{quest*}{Question}
\newtheorem*{claim*}{Claim}

\theoremstyle{definition}
\newtheorem{defn}[thm]{Definition}

\newtheorem{alg}[thm]{Algorithm}

\newtheorem{nota}[thm]{Notation}

\newtheorem{exs}[thm]{Examples}
\newtheorem{ex}[thm]{Example}
\newtheorem{sub-ex}[thm]{Sub-Example}
\newtheorem{rem}[thm]{Remark}
\newtheorem*{rem*}{Remark}

\newtheorem{rems}[thm]{Remarks}

\DeclareMathOperator{\Ad}{Ad}
\DeclareMathOperator{\Aut}{Aut}
\DeclareMathOperator{\coev}{coev}
\DeclareMathOperator{\Dom}{Dom}
\DeclareMathOperator{\End}{End}
\DeclareMathOperator{\ev}{ev}
\DeclareMathOperator{\Hom}{Hom}

\DeclareMathOperator{\op}{op}

\DeclareMathOperator{\supp}{supp}
\DeclareMathOperator{\id}{id}
\DeclareMathOperator{\Isom}{Isom}
\DeclareMathOperator{\ind}{ind}

\DeclareMathOperator{\Irr}{Irr}
\DeclareMathOperator{\Spec}{Spec}
\DeclareMathOperator{\Stab}{Stab}
\DeclareMathOperator{\Tr}{Tr}
\DeclareMathOperator{\tr}{tr}
\DeclareMathOperator{\Gr}{Gr}

\newcommand{\comment}[1]{}

\newcommand{\be}{\begin{enumerate}[label=(\arabic*)]}
\newcommand{\ee}{\end{enumerate}}

\newcommand{\C}{\mathbb{C}}

\newcommand{\set}[2]{\left\{#1 \middle| #2\right\}}

\renewcommand{\alg}[1]{{\bm{\langle} #1\bm{\rangle}}}

\def\semicolon{;}
\def\applytolist#1{
    \expandafter\def\csname multi#1\endcsname##1{
        \def\multiack{##1}\ifx\multiack\semicolon
            \def\next{\relax}
        \else
            \csname #1\endcsname{##1}
            \def\next{\csname multi#1\endcsname}
        \fi
        \next}
    \csname multi#1\endcsname}

\def\calc#1{\expandafter\def\csname c#1\endcsname{{\mathcal #1}}}
\applytolist{calc}QWERTYUIOPLKJHGFDSAZXCVBNM;
\def\bbc#1{\expandafter\def\csname bb#1\endcsname{{\mathbb #1}}}
\applytolist{bbc}QWERTYUIOPLKJHGFDSAZXCVBNM;
\def\bfc#1{\expandafter\def\csname bf#1\endcsname{{\mathbf #1}}}
\applytolist{bfc}QWERTYUIOPLKJHGFDSAZXCVBNM;
\def\sfc#1{\expandafter\def\csname s#1\endcsname{{\sf #1}}}
\applytolist{sfc}QWERTYUIOPLKJHGFDSAZXCVBNM;
\def\fc#1{\expandafter\def\csname f#1\endcsname{{\mathfrak #1}}}
\applytolist{fc}QWERTYUIOPLKJHGFDSAZXCVBNM;

\newcommand{\Rep}{{\sf Rep}}

\newcommand{\Mod}{{\sf Mod}}
\newcommand{\Proj}{{\sf Proj}}

\newcommand{\Bim}{{\sf Bim}}
\newcommand{\bfBim}{{\sf Bim_{bf}}}
\newcommand{\spBim}{{\sf Bim^{sp}}}
\newcommand{\spbfBim}{{\sf Bim_{bf}^{sp}}}
\renewcommand{\Vec}{{\sf Vec}}
\newcommand{\fdVec}{{\sf Vec_{fd}}}
\newcommand{\Hilb}{{\sf Hilb}}
\newcommand{\fdHilb}{{\sf Hilb_{fd}}}
\newcommand{\ConAlg}{{\sf ConAlg}}
\newcommand{\DisInc}{{\sf DisInc}}

\newcommand{\jw}[1]{f^{(#1)}}

\newcommand{\noshow}[1]{}
\newcommand{\MR}[1]{}

\usetikzlibrary{shapes}
\usetikzlibrary{backgrounds}
\usetikzlibrary{decorations,decorations.pathreplacing,decorations.markings}
\usetikzlibrary{fit,calc,through}
\usetikzlibrary{external}
\usetikzlibrary{arrows}
\tikzset{vertex/.style = {shape=circle,draw,fill=black,inner sep=0pt,minimum size=5pt}}
\tikzset{edge/.style = {->,> = latex', bend right}}
\tikzset{
	super thick/.style={line width=3pt}
}
\tikzset{
    quadruple/.style args={[#1] in [#2] in [#3] in [#4]}{
        #1,preaction={preaction={preaction={draw,#4},draw,#3}, draw,#2}
    }
}
\tikzstyle{shaded}=[fill=red!10!blue!20!gray!30!white]
\tikzstyle{unshaded}=[fill=white]
\tikzstyle{empty box}=[circle, draw, thick, fill=white, opaque, inner sep=2mm]
\tikzstyle{annular}=[scale=.7, inner sep=1mm, baseline]
\tikzstyle{rectangular}=[scale=.75, inner sep=1mm, baseline=-.1cm]
\tikzstyle{mid>}=[decoration={markings, mark=at position 0.5 with {\arrow{>}}}, postaction={decorate}]
\tikzstyle{mid<}=[decoration={markings, mark=at position 0.5 with {\arrow{<}}}, postaction={decorate}]
\tikzstyle{over}=[double, draw=white, super thick, double=]

\newcommand{\roundNbox}[6]{
	\draw[rounded corners=5pt, very thick, #1] ($#2+(-#3,-#3)+(-#4,0)$) rectangle ($#2+(#3,#3)+(#5,0)$);
	\coordinate (ZZa) at ($#2+(-#4,0)$);
	\coordinate (ZZb) at ($#2+(#5,0)$);
	\node at ($1/2*(ZZa)+1/2*(ZZb)$) {#6};
}

\begin{document}
\title{Realizations of algebra objects and discrete subfactors}
\author{Corey Jones and David Penneys}
\date{\today}
\maketitle
\begin{abstract}
We give a characterization of extremal irreducible discrete subfactors $(N\subseteq M, E)$ where $N$ is type ${\rm II}_1$ in terms of connected W*-algebra objects in rigid C*-tensor categories.
We prove an equivalence of categories where the morphisms for discrete inclusions are normal $N-N$ bilinear ucp maps which preserve the state $\tau \circ E$, and the morphisms for W*-algebra objects are categorical ucp morphisms.

As an application, we get a well-behaved notion of the standard invariant of an extremal irreducible discrete subfactor, together with a subfactor reconstruction theorem.
Thus our equivalence provides many new examples of discrete inclusions $(N\subseteq M, E)$, in particular, examples where $M$ is type ${\rm III}$ coming from non Kac-type discrete quantum groups and associated module W*-categories.
Finally, we obtain a Galois correspondence between intermediate subfactors of an extremal irreducible discrete inclusion and intermediate W*-algebra objects.
%
\end{abstract}

\tableofcontents

\section{Introduction}

The modern theory of ${\rm II}_1$ subfactors was initiated by Vaughan Jones in his landmark article \cite{MR0696688}.  
The 
\emph{Jones index} of a subfactor must take values in
$\set{4\cos^2(\pi/n)}{n\geq 3}\cup[0,\infty]$,
and the \emph{standard invariant} of a finite index subfactor is the lattice of higher relative commutants.
The standard invariant has been axiomatized in many equivalent ways; in finite depth, we have Ocneanu's paragroups \cite{MR996454} and Popa's canonical commuting squares \cite{MR1055708}.
Popa's $\lambda$-lattices \cite{MR1334479} give an axiomatization of the standard invariant in the general case, and Jones' equivalent planar algebras \cite{math.QA/9909027} provide a powerful graphical calculus for the construction \cite{MR2979509} and classification \cite{MR3166042,1509.00038} of standard invariants.

Using Longo's Q-systems \cite{MR1027496}, the article \cite{MR1966524} re-interprets the standard invariant as a unitary Frobenius algebra object in a rigid C*-tensor category \cite{MR1444286,MR2091457}.
Using this re-interpretation, we may view an extremal finite index subfactor $N\subseteq M$ (not just its standard invariant!) as a triple $(\cC, A, \bfH)$, where $\cC$ is a rigid C*-tensor category, $A\in\cC$ is a unitary Frobenius algebra generating $\cC$, and $\bfH$ is an embedding $\cC \hookrightarrow \spbfBim(N)$, the spherical/extremal bi-finite $N-N$ bimodules.
The subfactor $N\subseteq M$ is recovered by taking the $N$-bounded vectors $\bfH(A)^\circ$, i.e., $\bfH(A) \cong L^2(M)$.

Recently, there has been tremendous progress on understanding infinite index subfactors $N\subseteq M$ which are irreducible, i.e., $N'\cap M = \bbC$, and \emph{discrete/quasi-regular} \cite{1511.07329,MR1622812}.
In this article, we will restrict our attention to irreducible inclusions $(N\subseteq M, E)$ where $(N,\tau)$ is a type ${\rm II}_1$ factor with its canonical trace, $M$ is an arbitrary factor, and $E: M\to N$ is a faithful normal conditional expectation.
Following \cite{MR1622812}, such an inclusion is called \emph{discrete}\footnote{
A prior notion of a discrete inclusion of semi-finite factors was introduced in \cite{MR1055223} as an inclusion $(N,\Tr_N)\subseteq (M,\Tr_M)$ with a trace preserving faithful normal conditional expectation $E: M\to N$.
However, this definition does not ensure properties similar to those of crossed products of discrete groups, so a refined definition was given by \cite{MR1622812} which applies to inclusions of arbitrary factors.
}
if $L^2(M, \phi)$ can be decomposed into a direct sum of bi-finite $N-N$ bimodules, using the faithful normal state $\phi = \tau\circ E$.
Discreteness in our setting is equivalent to \emph{quasi-regularity}, a notion originating in \cite{MR1729488,MR2215135} which plays an important role in \cite{MR3406647,1511.07329} (see Proposition \ref{prop:QuasiregularIffDiscrete}).
Following \cite{MR3040370}, we call such an inclusion \emph{extremal} if every $N-N$ sub-bimodule of $L^2(M,\phi)$ has the same left and right von Neumann dimension.

In this article, we characterize extremal irreducible discrete inclusions $(N\subseteq M, E)$ where $N$ is assumed to be a type ${\rm II}_1$ factor and $M$ is an arbitrary factor in terms of categorical data similar to the triples $(\cC, A, \bfH)$ for extremal finite index subfactors.
As an immediate application, we obtain:
\begin{itemize}
\item
a well-behaved definition of the \emph{standard invariant of a discrete inclusion} as an algebra object $A$ associated to a rigid C*-tensor category $\cC$ in the spirit of Longo and Muger \cite{MR1027496,MR1966524}, along with 
\item
a subfactor reconstruction theorem in the spirit of Popa \cite{MR1334479}.
\end{itemize}
However, in the infinite index case, we must use some `infinite dimensional' algebra object which necessarily cannot lie in $\cC$, as objects in $\cC$ have finite quantum dimension.
We solve this problem using our theory of operator algebras internal to rigid C*-tensor categories introduced in \cite{MR3687214}.

An operator algebra is a $*$-subalgebra of $B(H)$.
From a categorical perspective, an operator algebra is a vector space which acts on a Hilbert space.
Starting with the rigid C*-tensor category $\fdHilb$ of finite dimensional Hilbert spaces, we forget the adjoint to obtain the rigid tensor category $\fdVec$ of finite dimensional vector spaces, which still caries an involutive structure given by complex conjugation.
Our operator algebra is an object in the category $\Vec$ of vector spaces, the $\ind$ completion of $\fdVec$.
Our operator algebra acts on an object in the category $\Hilb$ of Hilbert spaces, the unitary $\ind^*$ completion of $\fdHilb$.
$$
\xymatrix@C=0pt@R=5pt{
\fdVec\ar[dd]_{\ind} && \fdHilb\ar[ll]_{\text{Forget}} \ar[dd]^{\ind^*}
\\
\\
\Vec & \ni A \curvearrowright H\in& \Hilb
}
\qquad
\qquad
\xymatrix@C=0pt@R=5pt{
\cC^{\natural}\ar[dd]_{\ind} && \cC\ar[ll]_{\text{Forget}} \ar[dd]^{\ind^*}
\\
\\
\Vec(\cC) & \ni \bfA \curvearrowright \bfH\in& \Hilb(\cC)
}
$$

Now a rigid C*-tensor category $\cC$ has a canonical bi-involutive structure \cite{MR3663592} with two commuting involutions called the adjoint $* : \cC(a,b) \to \cC(b,a)$ and the conjugate $\overline{\,\cdot\,} : \cC(a , b) \to \cC(\overline{a}, \overline{b})$.
As before, we can forget the adjoint to obtain the involutive rigid tensor category $\cC^{\natural}$.
We can then take the $\ind$ completion to obtain the involutive tensor category $\Vec(\cC)$ of linear functors $(\cC^{\natural})^{\text{op}}\to \Vec$.
Back in $\cC$, we may take the Neshveyev-Yamashita unitary $\ind^*$ completion \cite{MR3509018} to obtain the bi-involutive W*-tensor category $\Hilb(\cC)$ of linear dagger functors $\cC^{\op}\to \Hilb$ with bounded natural transformations.

A $*$-algebra $\bfA\in \Vec(\cC)$ is an involutive lax monoidal functor $(\cC^{\natural})^{\text{op}} \to \Vec$, which is compatible with the conjugation functors $\overline{\,\cdot\,}$ on both $\cC^{\natural}$ and $\Vec$.
In \cite{MR3687214}, we gave an equivalence of categories between $*$-algebra objects in $\Vec(\cC)$ and cyclic $\cC$-module dagger categories $(\cM,m)$ with basepoint, whose objects are exactly the $c\otimes m$ for $c\in \cC$.
Using this equivalence, we call a $*$-algebra object $\bfA\in \Vec(\cC)$ a C*/W*-\emph{algebra object} if and only if its associated cyclic $\cC$-module dagger category of right free $\bfA$-modules is a C*/W*-category respectively.
We also defined the notion of a normal ucp morphism between W*-algebra objects in terms of ucp-multipliers between cyclic $\cC$-module W*-categories.

A W*-algebra object $\bfA\in \Vec(\cC)$ is called \emph{connected} if $\bfA(1_\cC) = \bbC$.
In this case, $\bfA(c)$ is a finite dimensional vector space for all $c\in \cC$ (see Proposition \ref{prop:DimensionBound}).

\begin{thmalpha}
\label{thm:Main}
Suppose $\bfH: \cC\to \spbfBim(N)$ is a fully faithful bi-involutive representation.
There is an equivalence of categories
\[
\left\{\,
\parbox{5.5cm}{\rm Connected W*-algebra objects $\bfA\in \Vec(\cC)$ with ucp morphisms}\,\left\}
\,\,\,\,\cong\,\,
\left\{\,\parbox{8cm}{\rm Extremal irreducible discrete inclusions $(N\subseteq M,E)$ supported on $\bfH(\cC)$ with normal $N-N$ bilinear ucp maps preserving $\tau\circ E$}\,\right\}.
\right.\right.
\]
\end{thmalpha}
\noindent
It is straightforward to show that under the correspondence between W*-algebra objects and cyclic $\cC$-module W*-categories, 
intermediate W*-algebra objects correspond to (non-full) $\cC$-module W*-subcategories with the same basepoint.

To prove Theorem \ref{thm:Main}, to each triple $(\cC, \bfA, \bfH)$, we associate a canonical extremal irreducible discrete inclusion $(N\subseteq M, E)$.
One way to view this over factor $M$ is as a \emph{realization} or \emph{coend}  $|\bfA|_\bfH$ of the functor $\bfH^\circ\otimes \bfA : \cC \boxtimes \cC^{\op} \to \Vec$, where $\bfH^\circ : \cC \to \Vec$ takes the vector space of $N$-bounded vectors $\bfH(c)^\circ$ for $c\in \cC$.
We use the term `realization' to remind the reader of the `geometric realization' of a simplicial set in algebraic topology, which is also a coend (see Remark \ref{rem:AlgebraicTopology}).

A second way to view the over factor $M$ is as a \emph{crossed product} $N \rtimes_\bfH \bfA$, where $\bfA$ is an object similar to a discrete group.
When $\bfH$ comes from an outer action of a discrete group $\Gamma \to \Aut(N)$, the group algebra $\bfA=\bbC[\Gamma]\in \Vec(\Gamma)$ is a connected W*-algebra object, and $N\rtimes_\bfH \bfA $ is the usual crossed product.

Using Connes' cocycle derivative, we show that $|\bfA|_\bfH=N\rtimes_\bfH \bfA$ must either be type ${\rm II}_1$ or type ${\rm III}$
(see Corollary \ref{cor:IrreducibleInclusionIIorIII}).
We introduce a modular operator for connected W*-algebra objects in $\Vec(\cC)$, along with the analog of Connes' modular spectrum $S(\bfA)$, which is easily computable in examples.
Moreover, Connes' modular spectrum $S(|\bfA|_\bfH)$ is equal to $S(\bfA)$ (see Proposition \ref{prop:ModularInvariant}).
Thus we can construct explicit examples of extremal irreducible discrete inclusions $(N\subseteq M, E)$ where $N$ is type ${\rm II}_1$ and $M$ is type ${\rm III}$, in spite of the fact that $\bfH: \cC \to \spbfBim(N)$ takes values in extremal $N-N$ bimodules.

For example, a discrete quantum group $\bbG = (\cC, \bfF)$ can be viewed as a dagger tensor functor $\bfF: \cC \to \fdHilb$, which gives us a $\cC$-module W*-category.
Choosing the basepoint $\bbC \in \fdHilb$, we get a canonical connected W*-algebra object $\bfA\in \Vec(\cC)$.

\begin{coralpha}
The realization $|\bfA|_\bfH$ of the connected \emph{W*}-algebra object $\bfA$ coming from a discrete quantum group $\bbG = (\cC,\bfF)$ together with a fully faithful representation $\bfH : \cC \to \spbfBim(N)$ is type ${\rm III}$ if and only if $\bbG$ is not Kac-type.
(See Section \ref{sec:QuantumGroups} for more details.)
\end{coralpha}

This corollary is related to \cite[Thm.~3.5]{MR2113893} in the case of compact quantum groups.
De Commer and Yamashita's classification of $\cT\cL\cJ(\delta)$-module W*-categories \cite{MR3420332} gives many more such examples (see Section \ref{sec:TLJmodules} for more details).

Going back the other way in Theorem \ref{thm:Main}, we use the same technique as \cite{MR1622812,1511.07329}.
Starting with an extremal irreducible inclusion $(N\subseteq M, E)$ supported on $\bfH(\cC)$, we get a connected W*-algebra object $\bfA\in \Vec(\cC)$ by taking normal $N-N$ bilinear maps
$$
\bfA(c) := \Hom_{N-N}(\bfH(c), L^2(M,\phi)). 
$$
The key ingredient in showing this defines a W*-algebra object is a generalization of \cite[Lem.~2.5]{1511.07329} to the case where $M$ is not assumed to be type ${\rm II}_1$.
We thank Stefaan Vaes for his help with this argument, which appears in Section \ref{sec:PSVlemma}.
This allows us to prove a Frobenius reciprocity result in the presence of a bifinite $N-N$ bimodule (Theorem \ref{thm:BifiniteFrobeniusReciprocity}), which gives a natural isomorphism
$$
\bfA(c) = \Hom_{N-N}(\bfH(c), L^2(M,\phi)) \cong \Hom_{N-M}(\bfH(c)\boxtimes_N  L^2(M,\phi), L^2(M,\phi)).
$$
Thus $\bfA$ is the W*-algebra object associated to the cyclic $\cC$-module W*-category generated by $L^2(M,\phi)\in \Bim(N,M)$, the $N-M$ bimodules.

\subsection*{Applications}

As discussed above, as an immediate application of our Theorem \ref{thm:Main}, we may define the \emph{standard invariant} of an extremal irreducible discrete inclusion $(N\subseteq M, E)$, as:
\begin{itemize}
\item
the rigid C*-tensor category $\cC$ of bifinite $N-N$ bimodules generated by $L^2(M,\phi)$, (the full 
subcategory of bifinite $N-N$ bimodules which appear as summands of $\boxtimes_N^k L^2(M,\phi)$ for some $k\geq 0$), together with
\item
the connected W*-algebra object $\bfA\in \Vec(\cC)$ given by 
$\bfA(K) = \Hom_{N-N}(K, L^2(M,\phi)) \cong \Hom_{N-N}(K\boxtimes_N L^2(M,\phi), L^2(M,\phi))$,
which corresponds to the cyclic $\cC$-module W*-subcategory of $N-M$ bimodules $\Bim(N,M)$ generated by $L^2(M,\phi)$.
\end{itemize}
We may also define an \emph{abstract standard invariant} as a pair $(\cC,\bfA)$ where $\cC$ is a rigid C*-tensor category, and $\bfA\in \Vec(\cC)$ is a connected W*-algebra object such that $\bfA$ generates $\cC$ (see Section \ref{sec:StandardInvariants} for the precise definition).

Moreover, we get a subfactor reconstruction theorem, in the spirit of Popa's theorems for finite index subfactors \cite{MR1055708,MR1334479,MR1339767} and the diagrammatic subfactor reconstructions of Guionnet-Jones-Shlyakhtenko-Walker \cite{MR2732052,MR2645882}.

\begin{coralpha}
Given an abstract standard invariant $(\cC, \bfA)$ with $\cC$ a rigid \emph{C*}-tensor category and $\bfA\in \Vec(\cC)$ a connected \emph{W*}-algebra object, there is an extremal irreducible discrete inclusion $(N\subseteq M, E)$ whose standard invariant is equivalent to $(\cC, \bfA)$.

Conversely, starting with an extremal irreducible discrete inclusion $(N\subseteq M, E)$, we get a standard invariant $(\cC, \bfA)$ and a fully faithful representation $\bfH : \cC \to \spbfBim(N)$ which allows us to construct an extremal irreducible discrete inclusion $(N\subseteq P, E_N^P)$.
There is a $*$-algebra isomorphism $\delta:P \to M$ which fixes $N$ such that $E_N^P =E\circ \delta$.
\end{coralpha}

We refer the reader to Section \ref{sec:StandardInvariants} for more details.
As a special case, we get yet another axiomatization of the standard invariant of an irreducible finite index ${\rm II}_1$ subfactor in terms of \emph{compact} connected W*-algebra objects in $\cC^{\natural}$ instead of unitary Frobenius algebras in $\cC$.
Our followup article \cite{1707.02155} provides a simple translation between these types of algebra objects in the finite index case.

As another application of our equivalence of categories, we discuss analytic properties for extremal irreducible discrete inclusions in Section \ref{sec:AnalyticProperties} in terms of analytic properties for connected W*-algebra objects in $\Vec(\cC)$ as discussed in \cite[\S5.5]{MR3687214}.
In particular, we have:
\begin{coralpha}
If $\cC$ is amenable or has the Haagerup property, any extremal irreducible discrete inclusion $(N\subseteq M, E)$ supported on $\bfH(\cC)$ has the corresponding property for $M$ relative to $N$.
\end{coralpha}
\noindent
We give explicit examples coming from planar algebras in Section \ref{sec:PlanarAlgebras}.

Finally, our equivalence also gives a Galois correspondence in the spirit of \cite{MR1622812,MR2561199}.
\begin{coralpha}
There is a bijective correspondence between intermediate subfactors $N\subseteq P \subseteq M$ of an extremal irreducible discrete inclusion supported on $\bfH(\cC)$ and intermediate connected \emph{W*}-algebra objects $\mathbf{1} \subseteq \bfB \subseteq \bfA$ in $\Vec(\cC)$.
\end{coralpha}
\noindent
Indeed, our correspondence closely resembles \cite[Lem.~3.8 and Thm.~3.9]{MR1622812} in the setting of irreducible discrete inclusions of infinite factors.
See Section \ref{sec:Galois} for more details.

\paragraph{Acknowledgements}

The authors would like to thank
Marcel Bischoff,
Ben Hayes,
Andr\'{e} Heniques,
Henri Moscovici,
Brent Nelson,
Dimitri Shlyakhtenko,
Stefaan Vaes, and
Makoto Yamashita
for helpful conversations.
Corey Jones was supported by Discovery Projects `Subfactors and symmetries' DP140100732 and `Low dimensional categories' DP160103479 from the Australian Research Council.
David Penneys was supported by NSF grants DMS-1500387/1655912 and 1654159.

Part of this work was completed at the recent program on von Neumann algebras at the Hausdorff Institute of Mathematics.
The authors would like to thank the institute and the organizers for their hospitality.

The authors would like to thank the Isaac Newton Institute for Mathematical Sciences, Cambridge, for support and hospitality during the programme Operator Algebras: Subfactors and their Applications where work on this paper was undertaken.
This work was supported by EPSRC grant no EP/K032208/1.

\section{Background}

We refer the reader to \cite[\S 2.1-2.2]{MR3663592} and \cite[\S 2.1-2.3]{MR3687214} for background on rigid C*-tensor categories.
Of particular importance is their bi-involutive structure, consisting of two commuting involutions (the adjoint $*$ and conjugate $\overline{\,\cdot\,}$) together with coherence data.
As usual, we will suppress all associators, unitors, and the involutive structure morphisms.

We fix a rigid C*-tensor category $\cC$.
As in \cite[\S 2.4]{MR3687214}, $\Vec(\cC)$ denotes the involutive tensor category of linear functors $\cC^{\op} \to \Vec$ with linear transformations.
Note that $\Vec(\cC) = \ind(\cC^{\natural})$, where $\cC^{\natural}$ denotes the involutive tensor category obtained from $\cC$ by forgetting the adjoint.
As in \cite[\S 2.6]{MR3687214}, $\Hilb(\cC)$ denotes the bi-involutive tensor W*-category of linear dagger functors $\cC^{\op} \to \Hilb$ with bounded natural transformations.
Note that $\Hilb(\cC)$ is the Neshveyev-Yamashita unitary $\ind$-category of $\cC$ from \cite{MR3509018}.

\subsection{Graphical calculus}

The Yoneda embedding gives fully faithful inclusions $\cC^{\natural} \hookrightarrow \Vec(\cC)$ as involutive tensor categories and $\cC\hookrightarrow \Hilb(\cC)$ as bi-inovlutive tensor categories.
This allows us to use the graphical calculus for the tensor categories $\Vec(\cC)$ and $\Hilb(\cC)$ which extends the usual graphical calculus for $\cC^{\natural}$ and $\cC$ as described in \cite[\S2.5-2.6]{MR3687214}.
For $1,c\in \cC$, we denote their images in $\Vec(\cC)$ and $\Hilb(\cC)$ under the Yoneda embedding by $\mathbf{1}, \mathbf{c}\in \Vec(\cC)$.

For $\bfV\in \Vec(\cC)$, the Yoneda lemma gives a natural identification $\bfV(a) \cong \Hom_{\Vec(\cC)}(\mathbf{a}, \bfV)$ for $a\in\cC$.
Given $f\in \bfV(a)$, we denote the corresponding natural transformation $\mathbf{a} \Rightarrow \bfV$ by
\begin{equation}
\label{eq:Yoneda}
\bfV(a)\ni f 
\longleftrightarrow
\begin{tikzpicture}[baseline=-.1cm]
    \draw (0,.6) -- (0,-.6);
    \roundNbox{unshaded}{(0,0)}{.3}{0}{0}{$f$}
    \node at (.2,.5) {\scriptsize{$\bfV$}};
    \node at (.2,-.5) {\scriptsize{$\mathbf{a}$}};
\end{tikzpicture}
\in \Hom_{\Vec(\cC)}(\mathbf{a}, \bfV).
\end{equation}
Given $\psi \in \cC(b,a)$, we denote the map $\bfV(\psi):\bfV(a) \to \bfV(b)$ by adding a coupon labelled $\psi$ below the $f$ coupon in \eqref{eq:Yoneda}.
Similarly, a natural transformation $\theta:\bfV \Rightarrow \bfW$ is denoted by adding a coupon labelled $\theta$ above the $f$ coupon in \eqref{eq:Yoneda}.
We refer the reader to \cite[\S2.5-2.6]{MR3687214} for the graphical tensor product on $\Vec(\cC)$ and $\Hilb(\cC)$, and the graphical adjoint and inner products for $\Hilb(\cC)$.

In this article, we will need to simultaneously use the graphical calculus for $\Vec(\cC),\Hilb(\cC)$ and $\Vec(\cC^{\op}),\Hilb(\cC^{\op})$.
Here, $\cC^{\op}$ is the rigid C*-tensor category obtained from $\cC$ by taking the opposite morphisms.
As before, the Yoneda embedding gives us a graphical calculus for $\Vec(\cC^{\op})$ and $\Hilb(\cC^{\op})$, but we must now read our diagrams from \emph{top to bottom}.
For example, if $\bfH\in \Hilb(\cC^{\op})$, we draw the following diagram for the natural transformation $\mathbf{a} \Rightarrow \bfH$ corresponding to $\xi\in \bfH(a)$:
$$
\bfH(a)\ni \xi 
\longleftrightarrow
\begin{tikzpicture}[baseline=-.1cm]
    \draw (0,.6) -- (0,-.6);
    \roundNbox{unshaded}{(0,0)}{.3}{0}{0}{$\xi$}
    \node at (.2,.5) {\scriptsize{$\mathbf{a}$}};
    \node at (.2,-.5) {\scriptsize{$\bfH$}};
\end{tikzpicture}
\in \Hom_{\Hilb(\cC^{\op})}(\mathbf{a}, \bfH).
$$
We use similar notation for $\bfH(\psi)$ for $\psi\in \cC(b,a)$ and natural transformations $\theta: \bfH \Rightarrow \bfK$; we add the $\psi$ coupon above and the $\theta$ coupon below.

\begin{rem}
We deduce two important relations for morphisms in $\Hilb(\cC^{\op})$.
First, suppose $\eta \in \bfH(a)$, $\xi \in \bfH(b)$, and $\psi\in \cC(b,c)$, so that 
$$
\eta^* \circ \bfH(\psi)[\xi] 
=
\begin{tikzpicture}[baseline=-.1cm]
    \draw (0,1.6) -- (0,-1.6);
    \roundNbox{unshaded}{(0,1)}{.3}{0}{0}{$\psi$}
    \roundNbox{unshaded}{(0,0)}{.3}{0}{0}{$\xi$}
    \roundNbox{unshaded}{(0,-1)}{.3}{0}{0}{$\eta^*$}
    \node at (-.15,1.5) {\scriptsize{$\mathbf{c}$}};
    \node at (-.15,.5) {\scriptsize{$\mathbf{b}$}};
    \node at (-.15,-.5) {\scriptsize{$\bfH$}};
    \node at (-.15,-1.5) {\scriptsize{$\mathbf{a}$}};
\end{tikzpicture}
\in
\Hom_{\Hilb(\cC^{\op})}(\mathbf{c},\mathbf{a})
\cong
\cC^{\op}(c,a).
$$
Taking the opposite morphism in $\cC(a,c)$, we have
\begin{equation}
\label{eq:HilbC-op Relation1}
(\eta^* \circ \bfH(\psi)[\xi] )^{\op}
= 
(\eta^*\circ \xi \circ \psi^{\op})^{\op} 
= 
\psi\circ (\eta^*\circ \xi)^{\op}
\in \cC(a,c).
\end{equation}
Now suppose that $\phi \in \cC(a,b)$, $\psi\in \cC(b,c)$, $\eta\in\bfH(d)$, and $\xi \in \bfH(e)$, so that
$$
(\psi^{\op}\otimes \eta^*)
\circ
(\phi^{\op}\otimes \xi)
=
\begin{tikzpicture}[baseline=-.1cm]
    \draw (0,1.1) -- (0,-1.1);
    \draw (-.8,1.1) -- (-.8,-1.1);
    \roundNbox{unshaded}{(0,.5)}{.3}{0}{0}{$\xi$}
    \roundNbox{unshaded}{(0,-.5)}{.3}{0}{0}{$\eta^*$}
    \roundNbox{unshaded}{(-.8,.5)}{.3}{0}{0}{$\phi$}
    \roundNbox{unshaded}{(-.8,-.5)}{.3}{0}{0}{$\psi$}
    \node at (-.95,-1) {\scriptsize{$\mathbf{a}$}};
    \node at (-.15,-1) {\scriptsize{$\mathbf{d}$}};
    \node at (-.95,0) {\scriptsize{$\mathbf{b}$}};
    \node at (-.2,0) {\scriptsize{$\bfH$}};
    \node at (-.95,1) {\scriptsize{$\mathbf{c}$}};
    \node at (-.15,1) {\scriptsize{$\mathbf{e}$}};
\end{tikzpicture}
\in 
\Hom_{\Hilb(\cC^{\op})}(\mathbf{c}\otimes \mathbf{e},\mathbf{a}\otimes \mathbf{d})
\cong
\cC^{\op}(c\otimes e,a\otimes d).
$$
Taking the opposite morphism in $\cC(a\otimes d, c\otimes e)$, we have
\begin{equation}
\label{eq:HilbC-op Relation2}
((\psi^{\op}\otimes \eta^*)
\circ
(\phi^{\op}\otimes \xi))^{\op}
=
(\phi\circ \psi)\otimes (\eta^*\circ \xi)^{\op}
\in \cC(a\otimes d, c\otimes e).
\end{equation}
\end{rem}

\subsection{Algebra objects in rigid C*-tensor categories}

We will be concerned with algebra objects $(\bfA, m, i) \in \Vec(\cC)$, which consist of an object $\bfA\in \Vec(\cC)$ together with an associative multiplication $m: \bfA \otimes \bfA \to \bfA$ and unit morphism $i: \mathbf{1} \to \bfA$.
By \cite[Prop.~3.3]{MR3687214}, algebra objects in $\Vec(\cC)$ are equivalent to linear lax monoidal functors $\cC^{\op} \to \Vec$.
By \cite[Thm.~3.8]{MR3687214}, we have an equivalence of categories between algebra objects in $\Vec(\cC)$ and \emph{cyclic} $\cC$-module categories $(\cM,m)$ with basepoint, whose objects are exactly the $c\otimes m$ for $c\in\cC$.
We get a cyclic module category $(\cM,m)$ from an algebra $\bfA$ by taking the category of free right $\bfA$-modules in $\Vec(\cC)$, and we recover an algebra from a cyclic $\cC$-module category by taking the internal hom: $\bfA(c) := \cM(c \otimes m, m)$.

We defined a $*$-structure for an algebra object $\bfA\in \Vec(\cC)$ in \cite[\S3.3]{MR3687214}, along with several equivalent characterizations.
Of particular importance is the equivalence of categories between $*$-algebra objects in $\Vec(\cC)$ and cyclic $\cC$-module dagger categories $(\cM,m)$.
It is through this equivalence that one may elegantly define the notions of C*/W*-algebra objects in $\Vec(\cC)$.
We now rapidly recall these definitions.

\begin{defn}
A $*$-algebra object is an algebra object $\bfA\in \Vec(\cC)$ together with a conjugate linear natural transformation $j: \bfA \Rightarrow \bfA$ which satisfies involutive, unital, and monoidal axioms.
This means that for all $c\in \cC$, we have a conjugate linear map $j_c : \bfA(c) \to \bfA(\overline{c})$ which is conjugate natural, and which is compatible with the conjugation on $\Vec(\cC)$, together with the multiplication and unit of $\bfA$.
We refer the reader to \cite[\S3.3]{MR3687214} for more details.
\end{defn}

We recall that being a C* or W*-algebra is a property, and not extra structure, of a complex $*$-algebra.
A $*$-algebra is a C*-algebra if and only if it admits a faithful representation on a Hilbert space whose image is a C*-algebra.
A $*$-algebra is a W*-algebra if and only if it is a C*-algebra with a predual.
We can alternatively characterize these properties in terms of the spectral radius \cite[Rem.~2.2]{MR3687214}.
In the same vein, being a C* or W*-category is a property of a dagger category, and not extra structure.

\begin{defn}
A $*$-algebra object $\bfA\in \Vec(\cC)$ is a C*/W*-algebra object if and only if its corresponding cyclic $\cC$-module dagger category $(\cM,m)$ is a C*/W*-category respectively.
\end{defn}

In this article, we study the simplest W*-algebra objects: the connected W*-algebra objects.

\begin{defn}
An algebra object $\bfA\in \Vec(\cC)$ is \emph{connected} if $\dim( \bfA(1_{\cC}))=1$.
\end{defn}

By \cite[Cor.~4.6]{MR3687214}, a connected C*-algebra object is \emph{locally finite}, which means that $\dim(\bfA(a))<\infty$ for all $a\in\cC$.
We will give a constructive proof of this fact in Proposition \ref{prop:DimensionBound}, together with an explicit upper bound on the dimension of $\bfM(a)$ using modular theory for connected algebras.
Just as a finite dimensional $*$-algebra is C* if and only if it is W*, a connected $*$-algebra object is C* if and only if it is W*.

The \emph{base algebra} of $\bfA$ is the honest C*-algebra $\bfA(1_\cC)$.
We have canonical $\bfA(1_\cC)$-valued inner products \cite[Def.~4.11]{MR3687214} on each $\bfA(a)$ given by
\begin{equation}
\label{eq:A(1)-InnerProduct}
{}_a\langle f, g\rangle
=
\begin{tikzpicture}[baseline = -.1cm]
    \draw (-.5,-.3) arc (-180:0:.5cm);
    \draw (-.5,.3) arc (180:0:.5cm);
    \filldraw (0,.8) circle (.05cm);
    \draw (0,.8) -- (0,1.2);
    \roundNbox{unshaded}{(-.5,0)}{.3}{0}{0}{$f$}
    \roundNbox{unshaded}{(.5,0)}{.3}{.2}{.2}{$j_a(g)$}
    \node at (-.65,-.5) {\scriptsize{$\mathbf{a}$}};
    \node at (-.65,.5) {\scriptsize{$\bfA$}};
    \node at (.65,-.5) {\scriptsize{$\overline{\mathbf{a}}$}};
    \node at (.65,.5) {\scriptsize{$\bfA$}};
    \node at (.15,1) {\scriptsize{$\bfA$}};
\end{tikzpicture}
\qquad\qquad
\langle f| g\rangle_{a}
=
\begin{tikzpicture}[baseline = -.1cm, xscale=-1]
    \draw (-.5,-.3) arc (-180:0:.5cm);
    \draw (-.5,.3) arc (180:0:.5cm);
    \filldraw (0,.8) circle (.05cm);
    \draw (0,.8) -- (0,1.2);
    \roundNbox{unshaded}{(-.5,0)}{.3}{0}{0}{$g$}
    \roundNbox{unshaded}{(.5,0)}{.3}{.2}{.2}{$j_a(f)$}
    \node at (-.65,-.5) {\scriptsize{$\mathbf{a}$}};
    \node at (-.65,.5) {\scriptsize{$\bfA$}};
    \node at (.65,-.5) {\scriptsize{$\overline{\mathbf{a}}$}};
    \node at (.65,.5) {\scriptsize{$\bfA$}};
    \node at (.15,1) {\scriptsize{$\bfA$}};
\end{tikzpicture}
\end{equation}
A \emph{state} on a C*-algebra object $\bfA \in \Vec(\cC)$ is a state on $\bfA(1_\cC)$.
For a connected W*-algebra object $\bfA\in \Vec(\cC)$, $\bfA(1_\cC) \cong \bbC$, and thus there is a canonical state which maps $i_\bfA \in \bfA(1_\cC)$ to $1_\bbC$.
Thus the inner products \eqref{eq:A(1)-InnerProduct} are genuine inner products on $\bfA(a)$, and each $\bfA(a)$ is a finite dimensional Hilbert space.
In general, these inner products will \emph{not} be equal, and their difference will involve a modular operator.
We perform this analysis in the next section.
In the event that the inner products \eqref{eq:A(1)-InnerProduct} are equal for all $a\in \cC$, we call $\bfA$ \emph{tracial}.

\begin{defn}
\label{defn:RightL2}
Given a connected W*-algebra object $\bfA\in \Vec(\cC)$, we define
$L^2(\bfA)\in \Hilb(\cC)$ to be the Hilbert space object obtained from $\bfA$ by endowing each $\bfA(a)$ with the right inner product $\langle \cdot|\cdot\rangle_a$ in \eqref{eq:A(1)-InnerProduct}.
\end{defn}

It is important to remember that there are really two choices for $L^2(\bfA)$, and they are only the same when $\bfA$ is tracial.
As in \cite[\S4.5]{MR3687214}, we have a bounded unital $*$-algebra representation $\pi: \bfA \Rightarrow \bfB(L^2(\bfA))$, where $\bfB(L^2(\bfA))$ is the W*-algebra object corresponding to the cyclic $\cC$-module W*-subcategory of $\Hilb(\cC)$ generated by $L^2(\bfA)$.
Using this $*$-algebra representation, we see that the inner products \eqref{eq:A(1)-InnerProduct} are given in $\Hilb(\cC)$ by the equations
\begin{equation}
\label{eq:InnerProductsInHilbC}
\begin{tikzpicture}[baseline=-.1cm]
	\draw (-.15,-.5) -- (-.15,.5);
	\draw (.15,-1.2) -- (.15,1.2);
	\roundNbox{unshaded}{(0,.5)}{.3}{.3}{.3}{$\pi_a(f)$}
	\roundNbox{unshaded}{(0,-.5)}{.3}{.3}{.3}{$\pi_a(g)^*$}
	\node at (-.3,0) {\scriptsize{$\mathbf{a}$}};
	\node at (.6,0) {\scriptsize{$L^2(\bfA)$}};
	\node at (.6,1) {\scriptsize{$L^2(\bfA)$}};
	\node at (.6,-1) {\scriptsize{$L^2(\bfA)$}};
\end{tikzpicture}
=
{}_a\langle f,g\rangle \id_{L^2(\bfA)}
\qquad\qquad
\begin{tikzpicture}[baseline=-.1cm]
	\draw (-.3,.8) arc (0:180:.3cm) -- (-.9,-.8) arc (-180:0:.3cm);
	\draw (.3,-1.2) -- (.3,1.2);
	\roundNbox{unshaded}{(0,.5)}{.3}{.3}{.3}{$\pi_a(g)$}
	\roundNbox{unshaded}{(0,-.5)}{.3}{.3}{.3}{$\pi_a(f)^*$}
	\node at (-1.1,0) {\scriptsize{$\overline{\mathbf{a}}$}};
	\node at (.75,0) {\scriptsize{$L^2(\bfA)$}};
	\node at (.75,1) {\scriptsize{$L^2(\bfA)$}};
	\node at (.75,-1) {\scriptsize{$L^2(\bfA)$}};
\end{tikzpicture}
=
\langle f|g\rangle_a \id_{L^2(\bfA)}.
\end{equation}

Recall that an object $\bfF\in \Vec(\cC)$ (respectively $\Hilb(\cC)$) is called \emph{compact} if $\bfF \in \cC^{\natural}\subset \Vec(\cC)$ (respectively in $\cC\subset \Hilb(\cC)$).

\begin{prop}
\label{prop:Compact}
Every compact connected \emph{W*}-algebra object $\bfA\in \cC^{\natural}\subset \Vec(\cC)$ is tracial.
\end{prop}
\begin{proof}
Note that $L^2(\bfA) \in \cC \subset \Hilb(\cC)$ is compact, so by \eqref{eq:InnerProductsInHilbC}, for all $f,g\in\bfA(a)$, 
\begin{equation*}
\langle g|f\rangle_a d_{L^2(\bfA)}
=
\begin{tikzpicture}[baseline=-.1cm]
	\draw (-.3,.8) arc (0:180:.2cm) -- (-.7,-.8) arc (-180:0:.2cm);
	\draw (0,-.5) -- (0,.5);
	\draw (.3,.8) .. controls ++(90:.5cm) and ++(90:.5cm) .. (-1.1,.8)  -- (-1.1,-.8) .. controls ++(270:.5cm) and ++(270:.5cm) .. (.3,-.8);
	\roundNbox{unshaded}{(0,.5)}{.3}{.3}{.3}{$\pi_a(g)^*$}
	\roundNbox{unshaded}{(0,-.5)}{.3}{.3}{.3}{$\pi_a(f)$}
	\node at (-.9,0) {\scriptsize{$\overline{\mathbf{a}}$}};
	\node at (.45,0) {\scriptsize{$L^2(\bfA)$}};
	\node at (.75,1) {\scriptsize{$L^2(\bfA)$}};
	\node at (.75,-1) {\scriptsize{$L^2(\bfA)$}};
\end{tikzpicture}
=
\begin{tikzpicture}[baseline=-.1cm]
	\draw (-.15,-.5) -- (-.15,.5);
	\draw (.15,-.5) -- (.15,.5);
	\draw (0,.8) arc (0:180:.4cm) -- (-.8,-.8) arc (-180:0:.4cm);
	\roundNbox{unshaded}{(0,.5)}{.3}{.3}{.3}{$\pi_a(f)$}
	\roundNbox{unshaded}{(0,-.5)}{.3}{.3}{.3}{$\pi_a(g)^*$}
	\node at (-.3,0) {\scriptsize{$\mathbf{a}$}};
	\node at (.6,0) {\scriptsize{$L^2(\bfA)$}};
	\node at (.45,1) {\scriptsize{$L^2(\bfA)$}};
	\node at (.45,-1) {\scriptsize{$L^2(\bfA)$}};
\end{tikzpicture}
=
{}_a\langle f, g\rangle d_{L^2(\bfA)}.
\qedhere
\end{equation*}
\end{proof}

\subsection{Basic Modular theory for connected W*-algebra objects}
\label{sec:ModularTheory}

We now discuss the basic modular theory to relate the left and right inner products on the $\bfA(a)$, along with the analog of Connes modular spectrum.

Recall that for a von Neumann algebra $M$ together with a faithful normal semi-finite weight $\phi$, the map $m\Omega \mapsto m^*\Omega$ is closable with closure $S_\phi = J_\phi \Delta^{1/2}_\phi$ where $\Delta_\phi = S_\phi^*S_\phi$.
Recall that Connes' modular spectrum is given by
$$
S(M) := \bigcap \set{\Spec(\Delta_\phi)}{\phi \text{ is a faithful normal semi-finite weight on }M}.
$$
The centralizer of a weight $\phi$ on $M$ is the set 
$$
M^\phi = \set{x\in M}{\sigma_t^\phi(x) =x \text{ for all }t\in\bbR}.
$$
If we have a faithful normal semi-finite weight $\phi$ on $M$ such that $M^\phi$ is factor, then $S(M) = \Spec(\Delta_\phi)$ by \cite[Cor.~3.2.7(a)]{MR0341115}.
A factor $M$ is called type 
$$
\begin{cases}
{\rm III}_0
&\text{if } S(M) = \{0,1\}
\\
{\rm III}_\lambda
&\text{if } S(M) = \{0\}\cup \lambda^\bbZ \text{ for }0<\lambda < 1
\\
{\rm III}_1
&\text{if } S(M) = [0,\infty).
\end{cases}
$$

Now suppose $\bfA$ is a connected W*-algebra object and recall $L^2(\bfA)$ has the right inner product.
For $a\in \cC$, we define a conjugate linear map $S_a : L^2(\bfA)(a) \to L^2(\bfA)(\overline{a})$ by $S_a(f) = j_a(f)$.
Since $L^2(\bfA)(a)$ is finite dimensional, $S_a$ is a bounded conjugate-linear operator, and clearly $S_a^{-1} = S_{\overline{a}}$.
We define $\Delta_a = S_a^* S_a$ which is a bounded positive operator on the finite dimensional Hilbert space $L^2(\bfA)(a)$.
Since $S_a$ and $S_a^*$ are conjugate natural, $\Delta: L^2(\bfA) \Rightarrow L^2(\bfA)$ is a natural transformation, but it might \emph{not} be a bounded natural transformation, so we cannot view it in $\End_{\Hilb(\cC)}(L^2(\bfA))$.
We calculate that for all $f,g\in \bfA(a)$,
\begin{equation}
\label{eq:ModularOperator}
{}_a\langle f, g\rangle
=
\begin{tikzpicture}[baseline = -.1cm]
    \draw (-.5,-.3) arc (-180:0:.5cm);
    \draw (-.5,.3) arc (180:0:.5cm);
    \filldraw (0,.8) circle (.05cm);
    \draw (0,.8) -- (0,1.2);
    \roundNbox{unshaded}{(-.5,0)}{.3}{0}{0}{$f$}
    \roundNbox{unshaded}{(.5,0)}{.3}{.2}{.2}{$j_a(g)$}
    \node at (-.65,-.5) {\scriptsize{$\mathbf{a}$}};
    \node at (-.65,.5) {\scriptsize{$\bfA$}};
    \node at (.65,-.5) {\scriptsize{$\overline{\mathbf{a}}$}};
    \node at (.65,.5) {\scriptsize{$\bfA$}};
    \node at (.15,1) {\scriptsize{$\bfA$}};
\end{tikzpicture}
=
\begin{tikzpicture}[baseline = -.1cm]
    \draw (-.5,-.3) arc (-180:0:.5cm);
    \draw (-.5,.3) arc (180:0:.5cm);
    \filldraw (0,.8) circle (.05cm);
    \draw (0,.8) -- (0,1.2);
    \roundNbox{unshaded}{(-.5,0)}{.3}{1}{.2}{$j_{\overline{a}}(j_{a}(f))$}
    \roundNbox{unshaded}{(.5,0)}{.3}{.1}{.3}{$j_a(g)$}
    \node at (-.65,-.5) {\scriptsize{$\overline{\overline{\mathbf{a}}}$}};
    \node at (-.65,.5) {\scriptsize{$\bfA$}};
    \node at (.65,-.5) {\scriptsize{$\overline{\mathbf{a}}$}};
    \node at (.65,.5) {\scriptsize{$\bfA$}};
    \node at (.15,1) {\scriptsize{$\bfA$}};
\end{tikzpicture}
=
\langle S_a f| S_a g\rangle_{\overline{a}}
=
\langle g| \Delta_a f\rangle_{a}.
\end{equation}
This proves the following result.

\begin{cor}
The connected \emph{W*}-algebra object $\bfA$ is tracial if and only if $\Delta$ is the identity natural tansformation.
\end{cor}

Using the modular operator $\Delta_a$, we prove a constructive bound on $\dim(\bfA(a))$ for $a\in\cC$ in terms of $d_a^2$.
Our proof is a translation of the argument in \cite[p.~39]{MR1622812} to the language of connected W*-algebra objects.

\begin{prop}
\label{prop:DimensionBound}
If $\bfA\in \Vec(\cC)$ is a connected \emph{W*}-algebra object, then $\dim(\bfA(a))\leq d_a^2$ for all $a\in\cC$.
If $\bfA$ is tracial, then $\dim(\bfA(a)) \leq d_a$.
\end{prop}
\begin{proof}
We compute that for all $a\in \cC$, $\Tr(\Delta_a)\leq d_a$ and $\Tr(\Delta_a^{-1})\leq d_a$, where $\Tr$ is the non-normalized trace on $B(L^2(\bfA)(a))$.
This implies both results by analyzing the spectral decomposition of $\Delta_a$.
Indeed, these inequalities together imply $d_a^{-1} \leq \Delta_a \leq d_a$, which implies $\dim(L^2(\bfA)(a))\leq d_a^2$.
When $\Delta_a = \id$, then using only the inequality $\Tr(\Delta_a)\leq d_a$ gives the result.

Choose an orthonormal basis $\{e_i\}$ for $\bfA(a)$ with respect to the \emph{left} inner product.
This means by \eqref{eq:InnerProductsInHilbC}, 
$$
\begin{tikzpicture}[baseline=-.1cm]
	\draw (-.15,-.5) -- (-.15,.5);
	\draw (.15,-1.2) -- (.15,1.2);
	\roundNbox{unshaded}{(0,.5)}{.3}{.35}{.35}{$\pi_a(e_i)$}
	\roundNbox{unshaded}{(0,-.5)}{.3}{.35}{.35}{$\pi_a(e_j)^*$}
	\node at (-.3,0) {\scriptsize{$\mathbf{a}$}};
	\node at (.6,0) {\scriptsize{$L^2(\bfA)$}};
	\node at (.6,1) {\scriptsize{$L^2(\bfA)$}};
	\node at (.6,-1) {\scriptsize{$L^2(\bfA)$}};
\end{tikzpicture}
=
{}_a\langle e_i, e_j\rangle\id_{L^2(\bfA)}
=
\delta_{i=j} \id_{L^2(\bfA)}.
$$
We now compute using \eqref{eq:InnerProductsInHilbC} and \eqref{eq:ModularOperator} (and suppressing $\id_{L^2(\bfA)}$) that
$$
\Tr(\Delta_a^{-1}) 
=
\sum {}_a \langle \Delta_a^{-1} e_i, e_i\rangle
=
\sum \langle e_i | e_i\rangle_a
=
\sum
\begin{tikzpicture}[baseline=-.1cm]
	\draw (-.3,.8) arc (0:180:.3cm) -- (-.9,-.8) arc (-180:0:.3cm);
	\draw (.3,-1.2) -- (.3,1.2);
	\roundNbox{unshaded}{(0,.5)}{.3}{.3}{.3}{$\pi_a(e_i)$}
	\roundNbox{unshaded}{(0,-.5)}{.3}{.3}{.3}{$\pi_a(e_i)^*$}
	\node at (-1.1,0) {\scriptsize{$\overline{\mathbf{a}}$}};
	\node at (.75,0) {\scriptsize{$L^2(\bfA)$}};
	\node at (.75,1) {\scriptsize{$L^2(\bfA)$}};
	\node at (.75,-1) {\scriptsize{$L^2(\bfA)$}};
\end{tikzpicture}
\leq
d_a.
$$
Now since $S_a^{-1} = S_{\overline{a}}$, we know $\Delta_a^{-1} = (S_a^*S_a)^{-1} = S_{\overline{a}} S_{\overline{a}}^*$.
Applying the above calculation to $\overline{a}$, we calculate
\begin{equation*}
\Tr(\Delta_a) = \Tr(S_a^*S_a) = \Tr(S_aS_a^*) = \Tr(\Delta_{\overline{a}}^{-1}) \leq d_a.
\qedhere
\end{equation*}
\end{proof}

\begin{rem}
Translating from the language of \cite[p.~43]{MR1622812}, they were not aware of examples of non-tracial connected W*-algebra objects where $\Delta \neq \id$.
We give examples in Section \ref{sec:QuantumGroups} coming from non Kac-type discrete quantum groups.
This non-tracial case may appear in \cite{MR2561199}, although we could not find explicit examples mentioned therein.
\end{rem}

\begin{defn}
The \emph{modular spectrum} $S(\bfA)$ of $\bfA$ is the closure of
$\bigcup_{a\in \cC} \Spec(\Delta_a)$
in
$[0,\infty)$.
\end{defn}

The following lemma is a helpful tool to compute $S(\bfA)$.

\begin{lem}  
Let $\Lambda\subseteq \cC$ be a collection of objects such that for every $a\in \Irr(\cC)$, there is some $b\in \Lambda$ such that $a\prec b$, i.e., $\cC(a,b) \neq 0$. 
Then $S(\bfA)$ is the closure of $\bigcup_{b\in \Lambda} \Spec(\Delta_b)$ in $[0,\infty)$.
\end{lem}
\begin{proof}  
Clearly $\bigcup_{a\in \Lambda} \Spec(\Delta_a) \subset \bigcup_{a\in \cC} \Spec(\Delta_a)$, so we need only prove the reverse direction.
By \cite[Def.~2.31]{MR3687214}, we have a natural isomorphism $L^{2}(\bfA)(b)\cong \bigoplus_{a\in \Irr(\cC)} \bfA(a)\boxtimes \cC(b, a)$.
Now suppose $a\in \cC$, and decompose $a$ as a direct sum of simple objects $a\cong \bigoplus c$ where the $c\in\Irr(\cC)$ can occur with multiplicity.
For each $c\in \Irr(\cC)$, pick $b_c\in\Lambda$ such that $c\prec b_c$.
Since $\Delta$ is natural, 
\begin{equation*}
\Spec(\Delta_a) 
\subseteq 
\bigcup_{c\prec a} \Spec(\Delta_c) 
\subseteq 
\bigcup_{c\prec a} \Spec(\Delta_{b_c}) 
\subseteq 
\bigcup_{b\in\Lambda} \Spec(\Delta_{b}) .
\qedhere
\end{equation*}
\end{proof}

We continue our exploration of the modular theory for connected W*-algebra objects in Section \ref{sec:MoreModularTheory}.
We postpone a systematic treatment of the general Tomita-Takesaki theory for arbitrary W*-algebra objects to a future article.

\subsection{Bimodules over \texorpdfstring{${\rm II}_1$}{II_1} factors}
\label{sec:II_1 bimodules}

Consider the category $\Bim(N)$ of bimodules over a ${\rm II}_1$ factor $(N,\tau)$, which is a bi-involutive W*-tensor category whose tensor product is given by the Connes fusion relative tensor product, denoted by $\boxtimes_N$.
We begin with some background on the rigid C*-tensor category $\bfBim(N)$, which is the full subcategory of $\Bim(N)$ consisting of bifinite $N-N$ bimodules.

First, we recall some information on bounded vectors and bases for a bifinite $N-N$ bimodule $H$, which can be found in \cite{MR1049618,MR1424954,MR3040370,ClaireSorinII_1}.
An $H_N$-\emph{basis} consists of a finite subset $\{\beta\} \subset H^\circ$, the (left and right) $N$-bounded vectors, such that $\sum_\beta L_\beta L_\beta^* = \id_H$.
Here, for $\xi \in H^\circ$, $L_\xi : L^2(N) \to H$ denotes the bounded left multiplication operator $n\Omega \mapsto \xi n$.
We have the right $N$-valued inner product on $H^\circ$ given by $\langle \eta|\xi\rangle_N = L_\eta^* L_\xi$, which is linear on the right.
Similarly, an ${}_NH$-\emph{basis} consists of a finite subset $\{\alpha\}\subset H^\circ$ such that $\sum_\alpha R_\alpha R_\alpha^* = \id_H$.
The left $N$-valued inner product on $H^\circ$ is given by ${}_N\langle \eta, \xi\rangle = J R_\eta^* R_\xi J$.

These bases are called \emph{orthogonal} if moreover the $L_\beta L_\beta^*$ (respectively the $R_\alpha R_\alpha^*$) are mutually orthogonal projections.
An orthogonal $H_N$ (respectively ${}_NH$) basis always exists.

Recall that $\overline{H}$ is the conjugate Hilbert space with the $N-N$ bimodule structure $a\cdot \overline{\eta} \cdot b := \overline{b^*\eta a^*}$.
The underlying set of $\overline{H}$ is taken to be $H$, but with the conjugate Hilbert space structure, so $\overline{\overline{H}} = H$ on the nose as $N-N$ bimodules.
We denote vectors in $\overline{H}$ with a bar over them merely to remember we are using the conjugate Hilbert space structure.
It is straightforward to verify that $\{\beta\}$ is an $H_N$-basis if and only if $\{\overline{\beta}\}$ is an ${}_N \overline{H}$-basis, and similarly swapping the role of left and right.
One uses that 
\begin{equation}
\label{eq:SwapLeftAndRightInnerProduct}
\langle \eta| \xi \rangle_N = {}_N\langle \overline{\eta}, \overline{\xi}\rangle \qquad\qquad \text{ for all }\eta, \xi \in H^\circ.
\end{equation}

The rigid C*-tensor category $\bfBim(N)$ has \emph{two} natural choices for evaluation and coevaluation morphisms.
One comes from choosing standard solutions to the conjugate equations \cite{MR1444286,MR2091457,MR3342166}, and we call these the \emph{standard} evaluation and coevaluation morphisms.
The second comes from using bases of bounded vectors and $N$-valued inner products coming from the trace on $N$, and we call these the \emph{tracial} evaluation and coevaluation morphisms.

\begin{defn}
For a given bifinite $N-N$ bimodule $H$, the \emph{tracial} evaluation and coevaluation are given by
\begin{align}
\ev_H: \overline{H} \boxtimes_N H &\to L^2(N) 
&
\overline{H}^\circ \boxtimes_N H^\circ
\ni
\overline{\eta} \boxtimes \xi 
&\mapsto 
\langle \eta| \xi\rangle_N\Omega
\in N\Omega 
\notag
\\
\coev_H: L^2(N)  &\to H\boxtimes_N \overline{H}
&
N\Omega
\ni
n\Omega 
&\mapsto 
n\sum_\beta \beta\boxtimes \overline{\beta} 
\in
\overline{H}^\circ \boxtimes_N H^\circ.
\label{eq:EvAndCoev}
\end{align}
It is straightforward to show using \eqref{eq:SwapLeftAndRightInnerProduct} and \eqref{eq:EvAndCoev} that $\ev_H^* = \coev_{\overline{H}}$ and $\coev_H^* = \ev_{\overline{H}}$, and thus
\begin{equation}
\label{eq:EvAndCoevBar}
\coev_{\overline{H}}(n\Omega) = n\sum_\alpha \overline{\alpha}\boxtimes \alpha
\qquad\text{ and }\qquad 
\ev_{\overline{H}}(\eta \boxtimes \overline{\xi}) = {}_N\langle \eta, \xi\rangle\Omega .
\end{equation}
The left and right von Neumann dimensions of $H$ are given by
\begin{equation}
\label{eq:vonNeumannDimension}
\dim_{N-}(H) = \coev_H^* \circ \coev_H
\qquad
\text{and}
\qquad
\dim_{-N}(H) = \ev_H \circ \ev_H^*.
\end{equation}
\end{defn}

\begin{defn}[\cite{MR1444286,MR2091457,MR3342166}]
\label{defn:StatisticalDimension}
The \emph{standard} evaluation and coevaluation are determined by the \emph{balancing condition}, which says that for all $f\in \End_{N-N}(H)$,
\begin{equation}
\label{eq:BalancingCondition}
\coev_{H}^* \circ (f\otimes \id_{\overline{H}}) \circ \coev_H
=
\ev_H \circ (\id_{\overline{H}}\otimes f)\circ \ev_H^*.
\end{equation}
This common value called the \emph{statistical dimension} $d_H$ of $H$, and $d_H^2 = \dim_{N-}(H)\dim_{-N}(H)$.
\end{defn}

\begin{rem}
The tracial and standard evaluation and coevalutation on $\bfBim(N)$ induce distinct dual functors such that the double dual functors are both the identity functor.
However, choosing the identity natural isomorphism induces \emph{distinct} pivotal structures for the tracial and standard dual functors!
(For an easier example, one can observe the same phenomenon by choosing distinct dual functors for $\fdHilb(\bbZ/3\bbZ)$ using distinct cube roots of unity.)
Since the bi-involutive structure only sees the adjoint and complex conjugation, and not the dual functors, $\bfBim(N)$ still has a canonical bi-involutive structure.
\end{rem}

\begin{defn}[{\cite[Def.~4.1 and Cor.~4.4]{MR3040370}}]
An $N-N$ bimodule $H\in \Bim(N)$ is called \emph{extremal} or \emph{spherical} if for every $N-N$ sub-bimodule $K\subseteq H$,
$\dim_{N-}(K) = \dim_{-N}(K)$.
It is easy to see that sphericality is preserved under taking subobjects, direct sums, conjugates, and fusion, and thus the category $\spBim(N)$ of all spherical $N-N$ bimodules is a bi-involutive W*-tensor subcategory of $\Bim(N)$.
We denote by $\spbfBim(N)$ the spherical bi-finite $N-N$ bimodules.
\end{defn}

\begin{rem}
For an extremal/spherical bifinite $N-N$ bimodule $K$, the tracial evaluation and coevaluation are standard solutions as in Definition \ref{defn:StatisticalDimension}.
Thus we may write $\ev_H$ without any confusion.
\end{rem}

An $N-N$ bimodule $H$ is called \emph{invertible} if $H\boxtimes_N \overline{H} \cong L^2(N)$ and $\overline{H}\boxtimes_N H \cong L^2(N)$.
Such bimodules are easily seen to be bifinite with statistical dimension equal to 1.
It is easy to see that an invertible extremal bimodule is of the form $L^2(N)_\theta$ where $\theta\in \Aut(N)$ and the action is given by $x\cdot n\Omega \cdot y = xn\theta(y)\Omega$.
If $H$ is invertible but not extremal, then $\dim_{N-}(H)$ and $\dim_{-N}(H)$ both belong to the \emph{fundamental group} of $N$, given by
$\cF(N)=\set{\lambda >0}{N\cong N^\lambda}$.

\subsection{Representing rigid C*-tensor categories as bifinite bimodules}
\label{sec:RepresentationsOfCategories}

In this article, we assume the existence of a fully faithful bi-involutive representation $\bfH:\cC \to \spbfBim(N)$.
We know such representations exist for $N=L\bbF_\infty$ by \cite{MR2051399,MR3405915} when $\cC$ is essentially countable (we review the representation from \cite{MR3405915} in Section \ref{sec:PlanarAlgebras}).
This means we may assume that the tracial evaluation and coevaluation agree with the standard evaluation and coevaluation on $\bfH(\cC)\subset \spbfBim(N)$, and
\begin{equation}
\label{eq:NonStandard}
\bfH(\ev_c) = \ev_{\bfH(c)}
\qquad
\text{and}
\qquad
\bfH(\coev_c) = \coev_{\bfH(c)}
\qquad
\text{for all }c\in\cC.
\end{equation}

\begin{nota}
\label{nota:RepresentationAndConnectedAlgebra}
For the remainder of this section, we fix the following notation:
\begin{itemize}
\item
$(N,\tau)$ is a ${\rm II}_1$ factor with its canonical trace.
\item
$\bfH:\cC \to \spbfBim(N)$ is a fully faithful bi-involutive representation.
We will denote the tensorator by $\mu^\bfH$, i.e., $\mu_{a,b}^\bfH$ is a unitary natural isomorphism $\bfH(a) \boxtimes_N \bfH(b) \to \bfH(a\otimes b)$ which satisfies certain coherence axioms.
We suppress the involutive structure natural unitary isomorphisms $\chi_c : \bfH(\overline{c}) \to \overline{\bfH(c)}$.
\item
$\bfH^\circ: \cC \to \Vec$ is the lax monoidal functor obtained from $\bfH$ by taking the $N$-bounded vectors, i.e., $\bfH^\circ(c):=\bfH(c)^\circ$.
Here, the unit is given by $i_{\bfH^\circ}=\Omega\in \bfH(1_\cC)^\circ = N\Omega$ and the laxitor $\mu_{a,b}^{\bfH^\circ}:\bfH^\circ(a)\otimes \bfH^\circ(b) \to \bfH^\circ(a\otimes b)$ is given by $\eta \otimes \xi \mapsto \mu^\bfH_{a,b}(\eta\boxtimes \xi)$.
We will see in Proposition \ref{prop:BoundedVectorsGivesAnAlgebra} below that $\bfH^\circ\in \Vec(\cC^{\op})$ is a W*-algebra object.
\item
$\bfH$ will also denote the composite of $\bfH: \cC \to \spbfBim(N)$ with the forgetful functor $\spbfBim(N) \to \Hilb$ by a slight abuse of notation. 
\end{itemize}
\end{nota}

\begin{prop}
\label{prop:BoundedVectorsGivesAnAlgebra}
The lax monoidal functor $\bfH^\circ: \cC \to \Vec$ is a W*-algebra object in $\Vec(\cC^{\op})$, 
where for $\eta \in \bfH(a)^\circ$, $j_a^{\bfH^\circ}(\eta) = \overline{\eta} \in \overline{\bfH(a)}^\circ \cong \bfH(\overline{a})^\circ$.
\end{prop}
\begin{proof}
Consider the rigid C*-tensor category $\bfH(\cC)\subset \spbfBim(N)$.
Forgetting the left $N$-action, we consider $\bfH(\cC)\subset \Mod(N)$, the category of right $N$-modules.
Since these right $N$-modules are still $N-N$ bimodules, we still have a left $\cC$-module category.
Let $(\Mod(N), L^2(N))$ be the cyclic left $\cC$-module W*-subcategory of $\Mod(N)$ generated by $L^2(N)$ via $\bfH$, i.e., the objects are given by $\bfH(a)$ for $a\in \cC$.
Now we view $\Mod(N)^{\op}$ as a left $\cC^{\op}$-module W*-category by \cite[Rem.~3.14]{MR3687214}.
Note that for all $a\in \cC$, we have
$$
\bfH(a)^\circ 
\cong
\Hom_{\Mod(N)}(L^2(N), \bfH(a))
\cong
\Hom_{\Mod(N)^{\op}}(\bfH(a)\boxtimes_N L^2(N), L^2(N)).
$$
Thus $\bfH^\circ\in \Vec(\cC^{\op})$ is exactly the W*-algebra object corresponding to $(\Mod(N)^{\op}, L^2(N))$.
That the unit, laxitor, and $*$-structure are as claimed is a straightforward exercise using \cite[Const.~3.11 and Thm.~3.20]{MR3687214}.
\end{proof}

Using this proposition, we compute the following lemma which will be useful later on.

\begin{lem}
\label{lem:HPreservesEvaluation}
For all $\eta, \xi \in \bfH(a)$, 
$\langle \eta | \xi \rangle_N^{\bfH(a)} = \bfH^\circ(\ev_a)[\mu^{\bfH^\circ}_{\overline{a},a}(j_{a}^{\bfH^\circ}(\eta) \otimes \xi)]$.
\end{lem}
\begin{proof}
By the definition of $\mu^{\bfH^\circ}$ and $j^{\bfH^\circ}$ via Proposition \ref{prop:BoundedVectorsGivesAnAlgebra}, the fact that $\bfH$ preserves standard solutions, and \eqref{eq:EvAndCoev}, we have
\begin{equation*}
\bfH^\circ(\ev_a)[\mu^{\bfH^\circ}_{\overline{a},a}(j_{a}^{\bfH^\circ}(\eta) \otimes \xi)]
=
\bfH^\circ(\ev_a)\mu^{\bfH}_{\overline{a},a}(\overline{\eta} \boxtimes \xi)
=
\ev_{\bfH(a)}(\overline{\eta} \boxtimes \xi)
=\langle \eta|\xi\rangle_N^{\bfH(a)}.
\qedhere
\end{equation*}
\end{proof}

\begin{prop}
Let $\tau$ be the canonical state on $\bfH^\circ$ corresponding to the trace on $\bfH(1_\cC)^\circ = N\Omega$.
Then $\tau$ is tracial on $\bfH^\circ$, and we have an isomorphism $\bfH \cong L^2(\bfH^\circ)_\tau$, giving a normal $*$-algebra natural transformation $\pi : \bfH^\circ \Rightarrow \bfB(\bfH)$.
\end{prop}
\begin{proof}
First, $\tau$ is tracial, since for all $a\in \Irr(\cC)$ and $\eta, \xi\in \bfH(a)^\circ$, by  \eqref{eq:SwapLeftAndRightInnerProduct}, 
\eqref{eq:NonStandard}, and Lemma \ref{lem:HPreservesEvaluation},
\begin{equation}
\label{eq:BiinvolutiveRepresentationTracial}
\langle \eta| \xi\rangle_{a}
=
\tau(\langle \eta |\xi\rangle_{N}^{\bfH(a)}) 
=
\langle \xi, \eta\rangle_{\bfH(a)}
=
\tau({}_N^{\bfH(a)}\langle \xi, \eta\rangle)
=
{}_{a}\langle \eta, \xi\rangle.
\end{equation}
This equality directly gives us the desired isomorphism $\bfH \cong L^2(\bfH^\circ)_\tau$ since $\bfH(a)^\circ$ is dense in $\bfH(a)$.
\end{proof}

\section{Discrete/quasi-regular subfactors}
\label{sec:DiscreteAndQuasi-regular}

Discrete inclusions of semi-finite factors were first introduced in \cite{MR1055223} as an inclusion $(N,\Tr_N)\subseteq (M,\Tr_M)$ with a trace preserving faithful normal conditional expectation $E: M\to N$.
This definition was refined to apply to inclusions of arbitrary factors in \cite{MR1622812} to ensure properties of crossed products by discrete groups.
In this article, we will focus solely on the case when $N$ is a type ${\rm II}_1$ factor.

\begin{nota}
\label{nota:Discrete}
For $A$ a von Neumann algebra with faithful normal state $\phi$, we denote by $\Omega$ the image of $1\in A$ in $L^2(A,\phi)$.
We denote by $S_\phi, J_\phi, \Delta_\phi, \sigma^\phi_t$ the familiar ingredients of the Tomita-Takesaki theory.
We use the convention that $\sigma^\phi_t(x) = \Delta^{it}x\Delta^{-it}$.
We denote the right action of $A$ on $L^2(A,\phi)$ by $\xi\vartriangleleft a = J_\phi a^* J_\phi \xi$.

For this section, $(N,\tr_N)$ is a type ${\rm II}_1$ factor with its faithful tracial state, $M$ is a factor containing $N$, and $E: M\to N$ is a faithful normal conditional expectation.
We let $\phi=\tr_N\circ E$ be the canonical faithful normal state on $M$, and notice that $N$ is contained in $M^\phi$, the centralizer of $\phi$.
This means that $\sigma^\phi_t(n) = n$ for all $t\in \bbR$, and every $n\in N$ is entire.
\end{nota}

The following lemma is straightforward.

\begin{lem}
\label{lem:ConjugateBimodule}
Suppose $K\subset L^2(M,\phi)$ is an $N-N$ sub-bimodule.
The map $J_\phi \xi \mapsto \overline{\xi}$ gives an $N-N$ bilinear unitary isomorphism between $J_\phi K$ and the conjugate $N-N$ bimodule $\overline{K}$.
\end{lem}
\begin{proof}
For all $\xi \in K$ and $n_1,n_2\in N$, we have
$$
n_1 \vartriangleright J_\phi\xi \vartriangleleft n_2
=
J_\phi n_2^* J_\phi n_1 J_\phi \xi
=
J_\phi( n_2^* \vartriangleright \xi \vartriangleleft n_1^*).
$$
We leave the rest of the details to the reader.
\end{proof}

We may define a bounded right $N$-action on $L^2(M,\phi)$ by $(m\Omega)\cdot n = mn\Omega$.
Indeed, if $n\in N$ and $m\in M$, since $E(m^*m)\geq 0$, we have
\begin{align*}
\|mn\Omega\|_2^2 &=\phi (n^*m^*mn) = \tr_N(n^*E(m^*m)n) = \tr_N(E(m^*m)^{1/2}nn^* E(m^*m)^{1/2})
\\&\leq \|nn^*\|_\infty \tr_N(E(m^*m)) = \|n\|^2_\infty \phi (m^*m) = \|n\|^2_\infty\cdot \|m\|_2^2.
\end{align*}
By \cite{MR0303307}, this right action is exactly the right action coming from the Tomita-Takesaki theory:
\begin{equation}
\label{eq:RightAction}
(m\Omega) \vartriangleleft n
=
(J_\phi n^*J_\phi)m\Omega
=
(J_\phi \sigma_{i/2}^\phi(n)^* J_\phi) m\Omega
=
mn\Omega
=
(m\Omega)\cdot n.
\end{equation}

\subsection{Definitions of discrete and quasi-regular}
\label{sec:Definitions}

\begin{defn}[\cite{MR1622812}]
An pair $(N\subseteq M,E)$ as in Notation \ref{nota:Discrete} is called a \emph{discrete inclusion} if the following equivalent conditions hold
\begin{enumerate}[(1)]
\item
Denoting the basic construction algebra by $\langle M,N\rangle=\{M,e_N\}''=J_\phi N'J_\phi \subset B(L^2(M,\phi))$ \cite{MR829381}, the relative commutant $N'\cap \langle M, N\rangle$ is a (possibly infinite) direct sum of finite dimensional matrix algebras, and $pN \subseteq p\langle M, N\rangle p$ has finite index for every finite rank projection $p\in N'\cap \langle M, N\rangle$.
\item
As an $N-N$ bimodule, $L^2(M,\phi)\cong\bigoplus_{i\geq 0} n_i H_i$ where the $H_i$ are pairwise non-isomorphic irreducible bimodules in $\bfBim(N)$, $H_0 \cong L^2(N)$, and $n_i \in \bbN$ for all $i$.
\end{enumerate}
An inclusion $N\subseteq M$ is \emph{irreducible} if $N'\cap M = \bbC$.
We will see later in Corollary \ref{cor:IrreducibleHasOneL2N} that irreducibility implies $n_0=1$ in (2) above.
A discrete inclusion $(N\subseteq M, E)$ is called \emph{extremal} if $L^2(M,\phi)$ is an extremal bimodule.
\end{defn}

\begin{rem}
When $M$ is type ${\rm II}_1$ and $[M:N] <\infty$, irreducible implies extremal.
This is no longer the case when $[M:N]=\infty$ \cite{MR1622812} or when $M$ is type ${\rm III}$ (see Example \ref{ex:TypeIII}).
\end{rem}

For an irreducible subfactor with faithful normal conditional expectation $(N\subseteq M, E)$ with $(N, \tr_N)$ type ${\rm II}_1$, we will see that discreteness is equivalent to \emph{quasi-regularity}, a notion originating in \cite{MR1729488,MR2215135} which plays an important role in \cite{MR3406647,1511.07329}.

\begin{defn}
Suppose we have a pair $(N\subseteq M,E)$ as in Notation \ref{nota:Discrete}.
The \emph{quasi-normalizer} $\cQ\cN_M(N)$ of $N\subset M$ is the set
$$
\set{x\in M}{\text{there exist }x_1,\dots, x_m, y_1,\dots, y_n\text{ such that }xN\subset \sum_{i=1}^m Nx_i \text{ and }Nx \subset \sum_{j=1}^n y_j N}.
$$
We say $(N\subseteq M,E)$ is \emph{quasi-regular} if $\cQ\cN_M(N)'' =M$.
\end{defn}

\begin{rem}
For notational convenience, we will write $Q^\circ = \cQ\cN_M(N)$ and $Q = (Q^\circ)'' = \cQ\cN_M(N)''$.
It is easy to see that $Q^\circ$ is a unital $*$-subalgebra of $M$ containing $N$.
Thus $N\subseteq Q\subseteq M$, and $Q$ is a factor if $N\subseteq M$ is irreducible.
\end{rem}

It is trivial to see that for an irreducible inclusion $(N\subseteq M, E)$ with $M$ type ${\rm II}_1$ and $E$ the unique trace preserving conditional expectation, discreteness is equivalent to quasi-regularity, since $Q=M$ if and only if $L^2(Q) = L^2(M)$ \cite[Prop.~2.1.8]{MR0696688}.
In our situation, $M$ is not assumed to be type ${\rm II}_1$, and there are indeed examples where $M$ is type ${\rm III}$ (see Example \ref{ex:TypeIII} and Sections \ref{sec:QuantumGroups} and \ref{sec:TLJmodules}).
By Corollary \ref{cor:IrreducibleInclusionIIorIII} below, we will see that if $M$ is not type ${\rm III}$, then $M$ is type ${\rm II}_1$, and $E$ must be the canonical trace-preserving conditional expectation.

\subsection{Extending a technical lemma of Popa-Shlyakhtenko-Vaes}
\label{sec:PSVlemma}

One of the key ingredients to our study of discrete subfactors requires extending a lemma of Popa-Shlyakhtenko-Vaes \cite[Lem.~2.5]{1511.07329} to the setting of irreducible discrete inclusions $(N\subseteq M, E)$ where $M$ is not assumed to be type ${\rm II}_1$, and $\phi$ is not assumed to be a trace.
This lemma assures us that we can pick Pimsner-Popa bases \cite{MR860811} for bifinite $N-N$ sub-bimodules $K\subset L^2(M,\phi)$ from $K\cap M\Omega$, and thus $K^\circ = K\cap M\Omega$.
We would like to thank Stefaan Vaes for helping us with this argument.
We begin with the following definition.

\begin{defn}
Suppose $N$ is a ${\rm II}_1$ factor, and we have an inclusion $K\subset H$ of $N-N$ bimodules.
The \emph{isotypic component of $K$ in $H$} is the smallest $N-N$ subbimodule of $H$ which contains all $N-N$ subbimodules of $H$ isomorphic to a summand of $K$.
(The isotypic component of $K$ is usually only defined when $K$ is irreducible, but the above definition makes sense for arbitrary $K$.)
\end{defn}

The lemma and corollary below are known to experts.
We provide proofs for completeness and the convenience of the reader.

\begin{lem}
\label{lem:UniqueNormalizingState}
Suppose $A\subseteq B$ is an irreducible subfactor, $\phi$ is a faithful normal state on $B$, and $\omega$ is a faithful semi-finite normal weight on $B$.
If the centralizers of $\phi$ and $\omega$ both contain $A$, then $\omega = \lambda \phi$ for some $\lambda >0$.
In particular, $\omega$ is finite.
\end{lem}
\begin{proof}
Recall that the Connes cocycle derivative $u_t=(D\phi: D\omega)_t\in B$ satisfies $\sigma^\phi_t(x) = u_t \sigma^\omega_t(x) u_t^*$ for all $x\in B$ and $t\in \bbR$ \cite[Thm.~XIII.3.3]{MR1943007}.
Now since $A$ is in the centralizer of both $\phi$ and $\omega$, for all $a\in A$, $a=u_tau_t^*$.
Hence for all $t\in \bbR$, $u_t\in U(A'\cap B) = U(1)$, and thus must be of the form $u_t = \lambda^{it}$ for some fixed $\lambda>0$.
Hence $\sigma_t^\phi = \sigma_t^\omega$ for all $t\in \bbR$, and by \cite[Cor.~XIII.3.6]{MR1943007}, we have $\omega = \phi_\lambda =\lambda \phi$.
\end{proof}

\begin{cor}
\label{cor:IrreducibleInclusionIIorIII}
Suppose $A\subseteq B$ is an irreducible subfactor with $A$ type ${\rm II}_1$ and $\phi$ is a faithful normal state on $B$ which contains $A$ in its centralizer $B^{\phi}$.
Then either $B$ is type ${\rm II}_1$ and $\phi$ is the canonical trace on $B$, or $B$ is type ${\rm III}$.
\end{cor}
\begin{proof}
It is clear that $B$ cannot be type ${\rm I}_\infty$, since $A'\cap B(H)$ is always type ${\rm II}$.
Suppose $B$ is type ${\rm II}$.
Let $\omega = \Tr_B$ be a faithful semi-finite normal trace on $B$, and note that $A$ is contained in the centralizer of $\omega$ since the modular automorphism group is trivial.
By Lemma \ref{lem:UniqueNormalizingState}, $\Tr_M = \lambda \phi$ for some $\lambda >0$, so $\Tr_M$ is finite, $B$ is type ${\rm II}_1$, and $\phi$ is the canonical trace on $B$.  
\end{proof}

\begin{rem}
Suppose $B$ is type ${\rm III}$ as in the statement of Corollary \ref{cor:IrreducibleInclusionIIorIII}.
Since $A\subseteq B^{\phi}\subseteq B$ and $A\subseteq B$ is irreducible, $B^{\phi}$ is a factor.  
By \cite[Cor.~3.2.7(a)]{MR0341115}, Connes' modular spectrum $S(B)=\operatorname{Spec}(\Delta_{\phi})$.
This means $B$ is type ${\rm III}_\lambda$ for $\lambda>0$.
\end{rem}

We now prove the analog of \cite[Lem.~2.5, parts 1-4]{1511.07329}.
We will prove the analog of \cite[Lem.~2.5, part 5]{1511.07329} in the next section.

\begin{prop}
\label{prop:MultiplicationMapBounded}
Suppose we have an irreducible inclusion $(N\subseteq M, E)$, 
and $K\subset L^2(M,\phi)$ is a bifinite $N-N$ bimodule.
\begin{enumerate}[(1)]
\item
There is a finite Pimsner-Popa $M_N$-basis  \cite{MR860811} $\{b_1,\dots, b_n\}\subset K\cap M\Omega$ such that
$$
\qquad\qquad
\text{
$P_K(x)
=
\sum_{i=1}^n b_iE_N(b_i^*x)$
for all
$x\in M$.}
$$
\item
The space of $N$-bounded vectors $K^\circ \subset K$ is equal to $K\cap M\Omega$.
\item
The densely defined linear map $K^\circ \otimes M\Omega \to L^2(M,\phi)$ given by $k\Omega\otimes m\Omega \mapsto k m\Omega$
extends to a bounded operator $\mu_K:K\boxtimes_N L^2(M,\phi)\to L^2(M,\phi)$.
\end{enumerate}
There are similar statements for ${}_N M$-bases and the multiplication map  $L^2(M,\phi)\boxtimes_N K\to L^2(M,\phi)$.
\end{prop}
\begin{proof}
The proof begins identically to \cite{1511.07329}.
We may assume $K$ is irreducible.
We choose an $N-N$ bilinear unitary isomorphism $V: p(\bbC^n \otimes L^2(N)) \to K$ for some projection $p\in M_n(\bbC)\otimes N$, where the left $N$-action on $p(\bbC^n \otimes L^2(N))$ is given by a finite index inclusion $\psi: N \to p(M_n(\bbC) \otimes N)p$.
For $i=1,\dots, n$, we let $b_i = V(p(e_i\otimes \Omega))\in K^\circ$, and note that $\{b_i\}_{i=1}^n\subset K^\circ$ is a Pimsner-Popa $K_N$-basis, which satisfies $\sum_{i=1}^n b_i \langle b_i |\xi\rangle_N = \xi$ for all $\xi\in K^\circ$.

Now we must depart from the proof in \cite{1511.07329}, since $\phi$ is not assumed to be a trace on $M$.
Look at the normal positive linear functional on $M$ given by
$$
\omega(x) = \sum_{i=1}^n \langle x b_i, b_i\rangle_{L^2(M,\phi)}.
$$
Notice that $N$ is contained in the centralizer of $\omega$, since $\omega$ is independent of the choice of basis $\{b_i\}_{i=1}^n$.
Indeed, if $P_K\in B(L^2(M,\phi))$ is the projection with range $K$, then since $N$ is contained in the centralizer of $\phi$, we have $P_KMP_K \subset (N^{\op})'\cap B(K)$.
It is well known that $\omega$ is independent of the choice of basis on $(N^{\op})'\cap B(K)$.
Thus for all $u\in U(N)$, we have
\begin{align*}
\omega(u^*xu)
&=
\sum_{i=1}^n \langle u^* xub_i,b_i\rangle_{L^2(M,\phi)}
=
\sum_{i=1}^n \langle x ub_i,ub_i\rangle_{L^2(M,\phi)}
=
\sum_{i=1}^n \langle x P_Kub_i,P_Kub_i\rangle_{L^2(M,\phi)}
\\&=
\sum_{i=1}^n \langle P_Kx P_Kub_i,ub_i\rangle_{K}
=
\sum_{i=1}^n \langle P_Kx P_Kb_i,b_i\rangle_{K}
=
\sum_{i=1}^n \langle xb_i,b_i\rangle_{L^2(M,\phi)}
=
\omega(x).
\end{align*}
Moreover, the support projection $\supp(\omega)$ of $\omega$ lies in $N'\cap M=\bbC$.
Since $\omega \neq 0$, $\supp(\omega)=1$, and $\omega$ is faithful.
By Lemma \ref{lem:UniqueNormalizingState}, we have $\omega(1)^{-1}\omega = \phi$.

We calculate that the $b_i$ are all right $M$-bounded vectors:
$$
\|xb_i\|^2_{L^2(M,\phi)}
=
\langle xb_i, xb_i \rangle
\leq
\sum_{i=1}^n \langle xb_i, xb_i \rangle
=
\omega(x^*x)
=
\omega(1) \phi(x^*x)
=
\omega(1) \|x\Omega\|_{L^2(M,\phi)}^2.
$$
Now the space of right $M$-bounded vectors is exactly $M'\Omega = J_\phi M \Omega$, since right multiplication by a right $M$-bounded vector from $L^2(M,\phi)$ commutes with left multiplication by $M$.
Thus for each $i=1,\dots, n$, $J_\phi b_i =a_i\Omega \in M\Omega$ for some $a_i\Omega \in J_\phi K\cap M\Omega$.

Notice now that $J_\phi K$ is also a bifinite $N-N$ bimodule by Lemma \ref{lem:ConjugateBimodule}.
We claim that $\{a_i\Omega\}_{i=1}^n$ is an ${}_N(J_\phi K)$-basis of elements from $J_\phi K\cap M\Omega$.
Indeed, for all $\xi\in (J_\phi K)^\circ$, we have $J_\phi \xi \in K^\circ$, and by \eqref{eq:SwapLeftAndRightInnerProduct},
\begin{align*}
\sum_{i=1}^n {}_N\langle \xi, a_i\rangle a_i
&=
J_\phi\sum_{i=1}^n J_\phi ({}_N\langle \xi, a_i\Omega\rangle a_i\Omega)
=
J_\phi\sum_{i=1}^n (J_\phi a_i\Omega) {}_N\langle a_i\Omega, \xi\rangle
\\&=
J_\phi\sum_{i=1}^n (J_\phi a_i\Omega) \langle J_\phi a_i\Omega| J_\phi\xi\rangle_N
=
J_\phi\sum_{i=1}^n b_i \langle b_i| J_\phi\xi\rangle_N
=
J_\phi J_\phi\xi
=
\xi.
\end{align*}

Since we could repeat the entire argument above starting with $J_\phi K$ instead of $K$, we may conclude that there is a ${}_N K$-basis $\{a_i\Omega\}_{i=1}^m$ consisting of elements of $K\cap M\Omega$.
This immediately implies that $K^\circ =K\cap M\Omega$, since both are algebraic $N-N$ bimodules.
The direction $K\cap M\Omega \subseteq K^\circ$ is trivial, and for all $\xi\in K^\circ$, $\xi = \sum_{i=1}^m {}_N\langle \xi, a_i\Omega\rangle a_i\Omega \in K\cap M\Omega$.
We conclude that our original $K_N$-basis $\{b_i\}_{i=1}^n\subset K^\circ$ is a subset of $K\cap M\Omega$.

The rest of the proof proceeds exactly as in \cite[Lem.~2.5, parts 1-4]{1511.07329}.
\end{proof}

\subsection{A modified Frobenius reciprocity application}

We may now use Proposition \ref{prop:MultiplicationMapBounded} to prove a modified version of Frobenius reciprocity 
for an irreducible inclusion $(N\subseteq M, E)$ in the presence of a bifinite $N-N$ bimodule $K$.
Consider the two intertwiner spaces
$\Hom_{N-M}(K\boxtimes_N L^2(M,\phi), L^2(M,\phi))$
and
$\Hom_{N-N}(K, L^2(M,\phi))$.
Since we have a faithful normal conditional expectation $E: M\to N$,
starting with $f\in \Hom_{N-M}(K\boxtimes L^2(M,\phi), L^2(M,\phi))$, we define
$\Phi(f) \in \Hom_{N-N}(K, L^2(M,\phi))$ as the composite map
\begin{equation}
\label{eq:FrobeniusEasy}
\begin{split}
K \cong K\boxtimes_N L^2(N)
&\xrightarrow{\id_K\boxtimes e_N^*}
K\boxtimes_N L^2(M,\phi)
\cong
K\boxtimes_N L^2(M,\phi)\boxtimes_{M} L^2(M,\phi)
\\&
\xrightarrow{f\boxtimes \id_{L^2(M,\phi)}}
L^2(M,\phi)\boxtimes_{M} L^2(M,\phi)
\cong
L^2(M,\phi).
\end{split}
\end{equation}
Using the bounded multiplication map from Proposition \ref{prop:MultiplicationMapBounded}, we can build a map in the other direction.
Starting with $g\in \Hom_{N-N}(K, L^2(M,\phi))$,
we can think of $g$ as a map $K\to g(K)$, where $g(K)\subset L^2(M,\phi)$ is a bifinite $N-N$ sub-bimodule.
We then define $\Psi(g)\in \Hom_{N-M}(K\boxtimes_N L^2(M,\phi), L^2(M,\phi))$ as the composite map
\begin{equation}
\label{eq:FrobeniusHard}
K \boxtimes_N L^2(M,\phi)
\xrightarrow{g\boxtimes \id_{L^2(M,\phi)}}
g(K)\boxtimes_N L^2(M,\phi)
\xrightarrow{\mu_{g(K)}}
L^2(M,\phi).
\end{equation}

\begin{rem}
\label{rem:MultiplyWithEntire}
Note that for entire $m\in M$ and arbitrary $\xi\in K$, the multiplication map $\mu_K:K\boxtimes_N L^2(M,\phi)\to L^2(M,\phi)$ is given by
$\xi\otimes m\Omega
\mapsto
\xi \vartriangleleft \sigma_{i/2}^\phi(m)
=
J_\phi \sigma_{i/2}^\phi(m)^*J_\phi\xi$.
Recall that if $M_{\text{ent}}\subset M$ denotes the set of entire elements, then $M_{\text{ent}}\Omega$ is dense in $L^2(M, \phi)$.
(Indeed, the maximal Tomita algebra $\mathfrak{A}_0\subset M_{\text{ent}}$, and $\mathfrak{A}_0\Omega$ is dense in $L^2(M, \phi)$
\cite[pp.~99-102]{MR1943007}).)
\end{rem}

\begin{thm}[Bifinite Frobenius Reciprocity]
\label{thm:BifiniteFrobeniusReciprocity}
Suppose $(N\subseteq M, E)$ is an irreducible inclusion and $K$ is a bifinite $N-N$ bimodule.
The maps $\Phi$ and $\Psi$ defined by \eqref{eq:FrobeniusEasy} and \eqref{eq:FrobeniusHard} witness a natural isomorphism
$$
\Hom_{N-M}(K\boxtimes_N L^2(M,\phi), L^2(M,\phi))
\cong
\Hom_{N-N}(K, L^2(M,\phi)).
$$
\end{thm}
\begin{proof}
Naturality is straightforward and left to the reader.
We verify that the maps compose each way to the identity.
Suppose $f\in \Hom_{N-M}(K\boxtimes_N L^2(M,\phi), L^2(M,\phi))$.
Since $M_{\text{ent}}\Omega$ is dense in $L^2(M,\phi)$ by Remark \ref{rem:MultiplyWithEntire}, we calculate that for $\xi\in K$ and $m\in M_{\text{ent}}$,
\begin{align*}
\Psi(\Phi(f))(\xi \boxtimes m\Omega)
&=
[\mu_{\Phi(f)(K)}\circ (\Phi(f) \boxtimes \id_{L^2(M,\phi)})]
(\xi \boxtimes m\Omega)
&&\text{\eqref{eq:FrobeniusHard}}
\\&=
\mu_{\Phi(f)(K)}(\Phi(f)(\xi)\boxtimes m\Omega)
\\&=
\mu_{\Phi(f)(K)}(f(\xi\otimes \Omega) \boxtimes m\Omega)
&&\text{\eqref{eq:FrobeniusEasy}}
\\&=
f(\xi\boxtimes \Omega)\vartriangleleft \sigma_{i/2}^\phi(m)
&&\text{(Rem.~\ref{rem:MultiplyWithEntire})}
\\&=
f(\xi\boxtimes (\Omega\vartriangleleft \sigma_{i/2}^\phi(m)))
\\&=
f(\xi\boxtimes J_\phi\sigma_{i/2}^\phi(m)^*J_\phi\Omega)
\\&=
f(\xi\boxtimes m\Omega).
&&\text{($m$ entire)}
\end{align*}
Now starting with $g\in \Hom_{N-N}(K, L^2(M,\phi))$, we easily calculate that for all $\xi\in K$,
\begin{equation*}
\Phi(\Psi(g))(\xi)
=
\Psi(g)(\xi\boxtimes \Omega)
=
\mu_{g(K)}(g(\xi)\boxtimes \Omega)
=
g(\xi).
\qedhere
\end{equation*}
\end{proof}

\begin{cor}
\label{cor:IrreducibleHasOneL2N}
For an irreducible inclusion $(N\subseteq M, E)$,
$$
\Hom_{N-N}(L^2(N), L^2(M,\phi)) \cong \Hom_{N-M}(L^2(M,\phi), L^2(M,\phi)) =N'\cap M = \bbC. 
$$
\end{cor}

\begin{defn}
Given an irreducible inclusion $(N\subseteq M, E)$, we define the \emph{underlying connected} W*-\emph{algebra object} to be the algebra object $\alg{M} \in \Vec(\bfBim(N))$ corresponding to the cyclic $\cC$-module W*-category $\Bim(N,M)$ of $N-M$ bimodules generated by the basepoint $L^2(M,\phi)$.
By Theorem \ref{thm:BifiniteFrobeniusReciprocity}, for $K\in \bfBim(N)$,
$$
\alg{M}(K)
=
\Hom_{N-M}(K\boxtimes_N L^2(M,\phi), L^2(M,\phi))
\cong
\Hom_{N-N}(K, L^2(M,\phi)),
$$
which is is finite dimensional by Proposition \ref{prop:MultiplicationMapBounded}.
(Note also that a connected W*-algebra is locally finite.)
\end{defn}

\begin{ex}
The underlying connected algebra object of a non-trivial irreducible inclusion can be the trivial W*-algebra object $1\in \cC^{\natural} \subset \Vec(\cC)$.
For example, we can take $N= R\rtimes \Stab(1) \subset R\rtimes S_\infty = M$, which is \emph{not} discrete, with the canonical trace preserving conditional expectation as in \cite[Cor.~5.11]{MR3040370}.
Here, the only bifinite $N-N$ bimodule summand of $L^2(M)$ is $L^2(N)$.
\end{ex}

We now prove a bound on the multiplicity of a bifinite $N-N$ sub-bimodule $K\subset L^2(M,\phi)$.
When $M$ is type ${\rm II}_1$, \cite[Lem.~2.5, part 5]{1511.07329} tells us that the multiplicity of $K$ is bounded above by both $\dim_{N-}(K)$ and $\dim_{-N}(K)$.
However, when $M$ is type ${\rm III}$, this bound no longer holds (see Corollary \ref{cor:NonExtremal} and Example \ref{ex:TypeIII}).
Instead, we have that the multiplicity is bounded by the product $\dim_{N-}(K)\dim_{-N}(K)=d_K^2$, the square of the statistical dimension \cite{MR1444286} (see Definition \ref{defn:StatisticalDimension}).
Thus this bound agrees with that in \cite[p.~39]{MR1622812}.

\begin{cor}
\label{cor:FiniteIsotypicComponent}
Suppose $(N\subseteq M, E)$ is an irreducible inclusion and $K\subset L^2(M,\phi)$ is a bifinite $N-N$ sub-bimodule.
Then the isotypic component of $K$ in $L^2(M,\phi)$ is bifinite, with the multiplicity of $K$ in $L^2(M,\phi)$ bounded above by $d_K^2=\dim_{N-}(K)\dim_{-N}(K)$.
\end{cor}
\begin{proof}
We consider the underlying connected W*-algebra object $\alg{M}$.
By Proposition \ref{prop:DimensionBound} and Theorem \ref{thm:BifiniteFrobeniusReciprocity}, we have 
\begin{align*}
\dim(\Hom_{N-N}(K, L^2(M,\phi)))
&=
\dim(\Hom_{N-M}(K\boxtimes_N L^2(M,\phi), L^2(M,\phi)))
\\&=
\dim(\alg{M}(K))
\leq
d_K^2.
\qedhere
\end{align*}
\end{proof}

\begin{rem}
\label{rem:IdentifyQN}
As noted at the end of \cite[Lem.~2.5]{1511.07329},
Proposition \ref{prop:MultiplicationMapBounded} and Corollary \ref{cor:FiniteIsotypicComponent} show that taking $\{K_i\}$ to be a maximal family of inequivalent irreducible bifinite $N-N$ sub-bimodules of $L^2(M,\phi)$, the \emph{evaluation homomorphism}
\begin{equation}
\label{eq:EvaluationHomomorphism}
\bigoplus_{i} K_i^\circ \otimes \Hom_{N-N}(K_i , L^2(M,\phi)) \to \cQ\cN_N(M) =Q
\end{equation}
from the algebraic direct sum given by $\xi\otimes f \mapsto f(\xi)$ is an isomorphism of vector spaces.
Note that the bounded map $f$ sends the bounded vector $\xi \in K_i$ to a bounded vector $f(\xi)$ in a bifinite $N-N$ sub-bimodule of $L^2(M,\phi)$, and thus $f(\xi) \in M\Omega$.
This means we have
\begin{align*}
Q^\circ 
&= \set{x\in M}{x\Omega \in K^\circ \text{ for some bifinite $N-N$ sub-bimodule }K \subset L^2(M,\phi)}
\\
L^2(Q,\phi) 
&= 
\bigvee \set{K\subset L^2(M,\phi)}{K\text{ is a bifinite $N-N$ sub-bimodule}}.
\end{align*}
Later on, we will see that the left hand side of \eqref{eq:EvaluationHomomorphism} is an \emph{algebraic realization}, and has a natural $*$-algebra structure.
Moreover, \eqref{eq:EvaluationHomomorphism} is naturally a $*$-algebra isomorphism.
\end{rem}

\subsection{Equivalences}

In this section, we will show that for an irreducible inclusion $(N\subseteq M, E)$, discreteness is equivalent to quasi-regularity.
Moreover, if $N = M^\phi$, the centralizer of $\phi$, then these conditions are equivalent to $M$ being generated by eigenoperators for $\sigma^\phi$ \cite{1604.03900}.
First, we show that these properties imply that $\Delta_\phi$ is almost periodic \cite{MR0358374}.

\begin{cor}
\label{cor:DeltaBounded}
Suppose $(N\subseteq M, E)$ is an irreducible inclusion and $K\subset L^2(M,\phi)$ is an irreducible bifinite $N-N$ sub-bimodule.
Let $\widetilde{K}$ be the isotypic component of $K$ in $L^2(M,\phi)$.
Then  $\Delta_\phi|_{\widetilde{K}\cap M\Omega}$ extends to a positive bounded operator in the finite dimensional von Neumann algebra $\End_{N-N}(\widetilde{K})$, and is thus diagonalizable.
Hence $\sigma_t^\phi$ preserves $\widetilde{K}\cap M\Omega$ for all $t\in \bbR$.
\end{cor}
\begin{proof}
First, it is easy to calculate that $\Delta^{1/2}_\phi =J_\phi S_\phi$ is $N-N$ bilinear on $M\Omega\subset \Dom(\Delta^{1/2}_\phi) = \Dom(S_\phi)\subset L^2(M,\phi)$.
By Corollary \ref{cor:FiniteIsotypicComponent}, $\widetilde{K}\subseteq L^2(M,\phi)$ is bifinite, and thus by Proposition \ref{prop:MultiplicationMapBounded}, $\widetilde{K}^\circ = \widetilde{K}\cap M\Omega$ is an algebraic $N-N$ bimodule, and is preserved by $\Delta^{1/2}_\phi$.
Now $\Delta^{1/2}_\phi$ is affiliated to the finite dimensional von Neumann algebra $\End_{N-N}(\widetilde{K})$, and thus extends to a bounded positive operator on $\widetilde{K}$, and is thus diagonalizable.
Hence we may write $\widetilde{K}=\bigoplus_{j=1}^n K_j$ with each $K_j \cong K$ as an $N-N$ bimodule, and $\Delta_\phi$ acts as a scalar $\lambda_j>0$ on $K_j$.
Then by Proposition \ref{prop:MultiplicationMapBounded} and \cite{MR1424954},
$$
\widetilde{K}\cap M\Omega = \widetilde{K}^\circ = \bigoplus_{j=1}^n K_j^\circ = \bigoplus_{j=1}^n K_j\cap M\Omega,
$$ 
and we may write each $x\Omega \in \widetilde{K}\cap M\Omega$ as $\sum_{j=1}^n x_j \Omega$ with each $x_j \in K_j\cap M\Omega$.
Then 
$$
\sigma_t^\phi(x)\Omega 
= 
\Delta^{it} x\Omega 
= 
\sum_{j=1}^n \Delta^{it} x_j \Omega
=
\sum_{j=1}^n \lambda_j^{it} x_j\Omega
\in 
\widetilde{K}\cap M\Omega,
$$
and thus $\sigma_t^\phi$ preserves $\widetilde{K}\cap M\Omega$.
\end{proof}

An immediate consequence of this corollary is the following.

\begin{cor}
\label{cor:QuasiPeriodic}
An irreducible discrete inclusion $(N\subseteq M, E)$ is \emph{almost-periodic}, i.e., the modular operator $\Delta_\phi$ on $L^2(M,\phi)$ is diagonalizable \cite{MR0358374}.
\end{cor}

\begin{prop}
\label{prop:DenseSubalgebraOfDiscrete}
Suppose $(N\subseteq M, E)$ is an irreducible inclusion.
The quasi-normalizer $Q^\circ=\cQ\cN_M(N)$ enjoys the following properties:
\begin{enumerate}[(1)]
\item
$S_\phi Q^\circ\Omega =Q^\circ\Omega$.
\item
$J_\phi Q^\circ\Omega =Q^\circ\Omega$.
\item
$\sigma_t^\phi(Q^\circ) = Q^\circ$ and $\sigma_t^\phi(Q) = Q$ for all $t\in \bbR$. 
\end{enumerate}
\end{prop}
\begin{proof}
The first property follows from the fact that $Q^\circ$ is a $*$-algebra.
The second follows from Remark \ref{rem:IdentifyQN} together with Lemma \ref{lem:ConjugateBimodule}.
For the third, we note that Remark \ref{rem:IdentifyQN} and Corollary \ref{cor:DeltaBounded} show that $\sigma_t^\phi(Q^\circ) = Q^\circ$, and $\sigma_t^\phi(Q) = Q$ follows by ultraweak continuity of $\sigma_t^\phi$.
\end{proof}

\begin{prop}
\label{prop:QuasiregularIffDiscrete}
Suppose $(N\subseteq M, E)$ is an irreducible inclusion.
Then the inclusion is quasi-regular if and only if it is discrete.
\end{prop}
\begin{proof}
Suppose $(N\subseteq M, E)$ is quasi-regular.
By Remark \ref{rem:IdentifyQN}, $L^2(M,\phi)$ decomposes as a direct sum of bifinite $N-N$ bimodules.
By Corollary \ref{cor:FiniteIsotypicComponent}, each of those bimodules occurs with finite multiplicity in $L^2(M,\phi)$.
Thus, the decomposition is as desired, and $(N\subseteq M, E)$ is discrete.

Suppose now that $(N\subseteq M, E)$ is discrete.
By (3) of Proposition \ref{prop:DenseSubalgebraOfDiscrete}, we may apply Takesaki's theorem \cite{MR0303307} to get a normal faithful conditional expectation $E_Q : M \to Q$ which satisfies $\phi|_Q \circ E = \phi$.
Following \cite{MR829381}, we get a Jones projection $e_Q \in B(L^2(M,\phi))$ with range $L^2(Q, \phi)$.
But by Remark \ref{rem:IdentifyQN}, $L^2(Q,\phi) = L^2(M,\phi)$, so $e_Q = 1$.
Thus by \cite[Lem.~3.1]{MR829381}, $Q=M$. 
\end{proof}

We expect the following proposition and corollary are known to experts.
We include proofs for completeness and convenience of the reader.

\begin{prop}
\label{prop:Eigenoperators}
Suppose $(N\subseteq M, E)$ is an irreducible inclusion. 
If $N=M^\phi$, then discreteness/quasi-regularity is equivalent to $M$ being generated by eigenoperators for $\sigma^\phi$.
\end{prop}
\begin{proof}
If $(N\subseteq M, E)$ is discrete, then $Q=M$ is generated as a von Neumann algebra by 
$$
\set{x\in M}{x\Omega \in K\cap M\Omega \text{ for some irreducible $N-N$ subbimodule } K\subset L^2(M,\phi)}.
$$
By Corollary \ref{cor:DeltaBounded}, this set is algebraically spanned by eigenoperators for $\sigma^\phi$.

Now assume there is a generating subset $G\subset M$ of eigenoperators for $\sigma^\phi$.
For $\lambda>0$, let
$$
K_\lambda^\circ = 
\set{x\in M}{ \sigma_t^\phi(x) = \lambda^{it}x \text{ for all }t\in\bbR},
$$
i.e., $K_\lambda^\circ\subset M$ is the eigenspace associated to $\lambda >0$.
Note that each $K^\circ_\lambda$ is an algebraic $N-N$ bimodule, and we let $K_\lambda$ be its closure in $L^2(M,\phi)$.

Suppose $\lambda >0$ such that $K_\lambda^\circ \neq 0$.
Notice that $x\in K_\lambda^\circ\subset M$ implies $x$ is left $N$-bounded, so $L_x$ is a bounded operator.
We claim that $L_x^* = L_{x^*}$, so that $x,y\in K_\lambda^\circ$ implies 
$\langle x|y\rangle_N = L_{x}^* L_y = L_{x^*} L_y = x^*y$.
Indeed, for all $x,y\in K_\lambda^\circ$ and $n\in N$, $x^*y \in M^\phi = N$, and thus
$$
\langle L_x n\Omega  , y\Omega\rangle_{L^2(M,\phi)} = \phi(y^*xn) = \tau(E(y^*xn)) = \tau((y^*x)n) = \langle n\Omega, x^*y\Omega \rangle_{L^2(N)}.
$$
We now claim that the multiplication maps $K_\lambda \boxtimes_N K_{\lambda^{-1}} \to L^2(N)$ and $K_{\lambda^{-1}}\boxtimes_N K_\lambda \to L^2(N)$ are unitary isomorphisms.
Thus each $K_\lambda$ is an invertible $N-N$ bimodule with inverse $K_{\lambda^{-1}}$, and both $K_\lambda, K_{\lambda^{-1}}$ are bifinite $N-N$ bimoudules.
Indeed, for all $x_1,x_2\in K_\lambda^\circ$ and $y_1,y_2\in K_{\lambda^{-1}}^\circ$, 
\begin{align*}
\langle x_1 \boxtimes y_1 , x_2\boxtimes y_2 \rangle_{K_\lambda \boxtimes_N K_{\lambda^{-1}}}
&=
\langle \langle x_2|x_1\rangle_N y_1, y_2\rangle_{K_{\lambda^{-1}}}
=
\langle (x_2^*x_1) y_1, y_2\rangle_{K_{\lambda^{-1}}}
\\&=
\phi(y_2^* x_2^*x_1y_1)
=
\tau(y_2^*x_2^*x_1y_1)
=
\langle x_1 y_1 , x_2y_2\rangle_{L^2(N)}.
\end{align*}
This means that the multiplication map is a non-zero isometry whose image is a non-zero $N-N$ sub-bimodule of $L^2(N)$, which is irreducible, and thus the map is unitary.
A similar argument works for the fusion in the reverse order.

We conclude that if $x\in M$ is an eigenoperator, then $x\Omega$ lives in an invertible and thus bifinite $N-N$ sub-bimodule of $L^2(M,\phi)$.
By Remark \ref{rem:IdentifyQN}, the unital $*$-subalgebra generated by $G$ is a subset of $Q^\circ$, and thus $Q = M$.
\end{proof}

\begin{cor}
\label{cor:NonExtremal}
Suppose we have an irreducible discrete inclusion $(N\subseteq M, E)$ where $N=M^\phi$.
Then 
$$
L^2(M,\phi) \cong \bigoplus_{\text{eigenvalues $\lambda$ of $\Delta_\phi$}} K_\lambda
$$
where each $K_\lambda$ is invertible, $\dim_{N-}(K_\lambda) = \lambda^{-1}$ and $\dim_{-N}(K_\lambda) = \lambda$, and $K_0 = L^2(N)$.
Thus if $M$ is type ${\rm III}$, each $\lambda^\bbZ$ is a subgroup of the fundamental group of $N$.
\end{cor}
\begin{proof}
Let $K_\lambda^\circ, K_\lambda$ be as in the proof of Proposition \ref{prop:Eigenoperators}.
We calculate that for all $x\in K_\lambda^\circ$, $R_x^* = \lambda R_{x^*}$.
Indeed, for all $y\in K_\lambda^\circ$ and $n\in N$, we have $y\in M_{\text{ent}}$, $xy^*\in M^\phi=N$, and 
$$
\langle R_x n\Omega, y\Omega \rangle_{K_\lambda} 
=
\phi(y^*nx)
=
\phi(nx\sigma_{-i}^\phi(y^*))
=
\lambda^{-1} \phi(nxy^*)
=
\lambda^{-1} \phi(xy^*n)
=
\langle n\Omega, \lambda y^{-1}x^* \rangle_{L^2(N)}.
$$
Thus for all $x,y\in K_\lambda^\circ$, ${}_N\langle x,y\rangle = \lambda^{-1} x y^*$.
Pick an orthogonal right $N$-basis $\{\beta\}\subset K_\lambda^\circ$ so that $\sum_\beta L_\beta L_\beta^* = \sum \beta \beta^* = \id_H$.
Using \eqref{eq:EvAndCoev} and \eqref{eq:EvAndCoevBar}, we calculate that  
$$
(\coev_{K_\lambda}^*\circ  \coev_{K_\lambda}) \Omega
=
\coev_{K_\lambda}^*\left(\sum_\beta \beta \boxtimes_N J_\phi \beta\right) 
=
\sum_\beta {}_N\langle\beta , J_\phi \beta\rangle 
=
\lambda^{-1}
\sum_\beta \beta \beta^*
=
\lambda^{-1}.
$$
We conclude that $\ev_{K_\lambda} \circ \ev_{K_\lambda}^* = \lambda$.
Thus $K_\lambda$ is only extremal for $\lambda = 1$.
Thus the decomposition of $L^2(M,\phi)$ is as claimed by \eqref{eq:vonNeumannDimension} and Proposition \ref{prop:Eigenoperators}.
The final claim follows immediately.
\end{proof}

\begin{ex}
\label{ex:TypeIII}
We thank Brent Nelson for pointing out to us examples of irreducible inclusions where $M$ is type ${\rm III}$ and generated by eigenoperators, and $N=M^\phi$ is type ${\rm II}_1$ \cite{1604.03900}.
In particular, we can take $M$ to be one of Shlyakhtenko's non-trivial almost periodic free Araki-Woods factors and $\phi$ to be the free quasi-free state where $M^\phi = L\bbF_\infty$ \cite{MR1444786,MR2697238}.
\end{ex}

\section{Realizations}
\label{sec:Realizations}

We now define algebra and Hilbert spaces realizations of pairs of functors into $\Vec$ or $\Hilb$ respectively.
For this section, we fix a rigid C*-tensor category $\cC$, along with a set of representatives for the simple objects $\Irr(\cC)$.
Recall that the inner product on $\cC(a,b)$ is given by
$$
\langle \phi, \psi \rangle_{\cC(a,b)} = 
\begin{tikzpicture}[baseline=-.1cm]
	\draw (0,.8) arc (180:0:.3cm) -- (.6,-.8) arc (0:-180:.3cm) -- (0,.8);
	\roundNbox{unshaded}{(0,.5)}{.3}{0}{0}{$\psi^*$}
	\roundNbox{unshaded}{(0,-.5)}{.3}{0}{0}{$\phi$}
	\node at (-.2,0) {\scriptsize{$b$}};
	\node at (.8,0) {\scriptsize{$\overline{a}$}};
	\node at (-.2,1) {\scriptsize{$a$}};
	\node at (-.2,-1) {\scriptsize{$a$}};
\end{tikzpicture}\,.
$$
For every $a\in \Irr(\cC)$ and $b\in \cC$, we let $\Isom(a,b)$ be a maximal set of isometries in $\cC(a,b)$ with orthogonal ranges.
This means for all $a,b\in \cC$, we have the \emph{fusion relation}
\begin{equation}
\label{eq:FusionRelation}
\begin{tikzpicture}[baseline=-.1cm]
	\draw (-.3,-.6) -- (-.3,.6);
	\draw (.3,-.6) -- (.3,.6);
	\node at (-.3,-.8) {\scriptsize{$a$}};
	\node at (.3,-.8) {\scriptsize{$b$}};
\end{tikzpicture}
=
\sum_{c\in\Irr(\cC)}
\sum_{\alpha \in \Isom(c, a\otimes b)}
\begin{tikzpicture}[baseline=-.1cm]
	\draw (-.3,-.6) arc (180:0:.3cm);
	\draw (-.3,.6) arc (-180:0:.3cm);
	\draw (0,-.3) -- (0,.3);
	\filldraw[fill=white] (0,.3) circle (.05cm) node [above] {\scriptsize{$\alpha$}};
	\filldraw[fill=white] (0,-.3) circle (.05cm) node [below] {\scriptsize{$\alpha^*$}};
	\node at (-.3,-.8) {\scriptsize{$a$}};
	\node at (-.3,.8) {\scriptsize{$a$}};
	\node at (.3,-.8) {\scriptsize{$b$}};
	\node at (.3,.8) {\scriptsize{$b$}};
	\node at (.2,0) {\scriptsize{$c$}};
\end{tikzpicture}
\end{equation}
and for all $a,b,c,d\in \cC$, we have the \emph{$I=H$ relation}
\begin{equation}
\label{eq:I=H}
\sum_{e\in\Irr(\cC)}
\sum_{\substack{
\alpha \in \Isom(a, e\otimes d)
\\
\beta \in \Isom(e, b\otimes c)
}}
\begin{tikzpicture}[baseline=.4cm]
	\draw (-.6,.9) arc (-180:0:.3cm);
	\draw (-.3,.6) arc (-180:0:.45cm) -- (.6,.9);
	\draw (.15,.15) -- (.15,-.3);
	\filldraw[fill=white] (.15,.15) circle (.05cm) node [above] {\scriptsize{$\alpha$}};
	\filldraw[fill=white] (-.3,.6) circle (.05cm) node [above] {\scriptsize{$\beta$}};
	\node at (.15,-.5) {\scriptsize{$a$}};
	\node at (-.6,1.1) {\scriptsize{$b$}};
	\node at (0,1.1) {\scriptsize{$c$}};
	\node at (.6,1.1) {\scriptsize{$d$}};
	\node at (-.3,.15) {\scriptsize{$e$}};
\end{tikzpicture}
\otimes
\begin{tikzpicture}[baseline=.4cm, xscale=-1]
	\draw (-.6,.9) arc (-180:0:.3cm);
	\draw (-.3,.6) arc (-180:0:.45cm) -- (.6,.9);
	\draw (.15,.15) -- (.15,-.3);
	\filldraw[fill=white] (.15,.15) circle (.05cm) node [above] {\scriptsize{$\overline{\alpha}$}};
	\filldraw[fill=white] (-.3,.6) circle (.05cm) node [above] {\scriptsize{$\overline{\beta}$}};
	\node at (.15,-.5) {\scriptsize{$\overline{a}$}};
	\node at (-.6,1.1) {\scriptsize{$\overline{b}$}};
	\node at (0,1.1) {\scriptsize{$\overline{c}$}};
	\node at (.6,1.1) {\scriptsize{$\overline{d}$}};
	\node at (-.3,.15) {\scriptsize{$\overline{e}$}};
\end{tikzpicture}
=
\sum_{f\in\Irr(\cC)}
\sum_{\substack{
\gamma \in \Isom(a, b\otimes f)
\\
\delta \in \Isom(f, c\otimes d)
}}
\begin{tikzpicture}[baseline=.4cm, xscale=-1]
	\draw (-.6,.9) arc (-180:0:.3cm);
	\draw (-.3,.6) arc (-180:0:.45cm) -- (.6,.9);
	\draw (.15,.15) -- (.15,-.3);
	\filldraw[fill=white] (.15,.15) circle (.05cm) node [above] {\scriptsize{$\gamma$}};
	\filldraw[fill=white] (-.3,.6) circle (.05cm) node [above] {\scriptsize{$\delta$}};
	\node at (.15,-.5) {\scriptsize{$a$}};
	\node at (-.6,1.1) {\scriptsize{$d$}};
	\node at (0,1.1) {\scriptsize{$c$}};
	\node at (.6,1.1) {\scriptsize{$b$}};
	\node at (-.3,.15) {\scriptsize{$f$}};
\end{tikzpicture}
\otimes
\begin{tikzpicture}[baseline=.4cm]
	\draw (-.6,.9) arc (-180:0:.3cm);
	\draw (-.3,.6) arc (-180:0:.45cm) -- (.6,.9);
	\draw (.15,.15) -- (.15,-.3);
	\filldraw[fill=white] (.15,.15) circle (.05cm) node [above] {\scriptsize{$\overline{\gamma}$}};
	\filldraw[fill=white] (-.3,.6) circle (.05cm) node [above] {\scriptsize{$\overline{\delta}$}};
	\node at (.15,-.5) {\scriptsize{$\overline{a}$}};
	\node at (-.6,1.1) {\scriptsize{$\overline{d}$}};
	\node at (0,1.1) {\scriptsize{$\overline{c}$}};
	\node at (.6,1.1) {\scriptsize{$\overline{b}$}};
	\node at (-.3,.15) {\scriptsize{$\overline{f}$}};
\end{tikzpicture}
\end{equation}
Moreover, for all $a\in\Irr(\cC)$, the sum $\sum_{\alpha\in \Isom(a,b)} \alpha \otimes \alpha^* \in \cC(a,b)\otimes \cC(b,a)$ is independent of the choice of $\Isom(a,b)$ since it summing over a basis and its dual basis with respect to the renormalized inner product on $\cC(a,b)$ given by $d_a^{-1}\langle \cdot, \cdot\rangle_{\cC(a,b)}$.

\subsection{Algebraic realizations}

Recall that algebra objects in $\Vec(\cC)$ are equivalent to lax monoidal linear functors $\cC^{\op}\to \Vec$.
Given two lax monoidal linear functors
$(\bfF, \mu^\bfF, i_\bfF):\cC \rightarrow \Vec$
and
$(\bfG, \mu^\bfG, i_\bfG) : \cC^{\op} \rightarrow \Vec$,
we get a lax monoidal linear functor $\bfF \otimes \bfG : \cC \boxtimes \cC^{\op} \to \Vec$ by $(a,b) \mapsto \bfF(a)\otimes \bfG(b)$.
We now construct an associative complex algebra called the \textit{realization} of $\bfF \otimes \bfG$.

\begin{defn}
\label{defn:AlgebraRealization}
We define the underling vector space of the realization by
$$
|\bfF\otimes\bfG| = \bigoplus_{a \in \Irr(\cC)} \bfF(a)\otimes \bfG(a).
$$
For
$f^a\otimes g^a \in \bfF(a)\otimes \bfG(a)$
and
$f^b \otimes g^b\in \bfF(b)\otimes \bfG(b)$,
we define their product by the formula
$$
(f^a\otimes g^a )\cdot (f^b \otimes g^b)
:=
\sum_{c\in \Irr(\cC)} \sum_{\alpha\in \Isom(c, a\otimes b)}
\bfF(\alpha^{*})[\mu^{\bfF}_{a,b}(f^a\otimes f^b)]
\otimes
\bfG(\alpha)[\mu ^{\bfG}_{a,b}(g^a\otimes g^b)].
$$
One shows this multiplication is associative using the I=H relation \eqref{eq:I=H}, together with the associativity of the lax monoidal functors $\bfF$ and $\bfG$.
The unit of the algebra is given by
$i_{\bfF}\otimes i_{\bfG} \in \bfF(1_\cC) \otimes \bfG(1_\cC)$.

If $\bfF,\bfG$ are also involutive, they correspond to $*$-algebras in $\cC$ and $\cC^{\text{op}}$ respectively.
We can then make $|\bfF\otimes\bfG|$ into a $*$-algebra as follows.
Let $j^{\bfF}$ and $j^{\bfG}$ be the corresponding $*$-structures of $\bfF$ and $\bfG$.
For each $a\in \Irr(\cC)$, let $a'\in \Irr(\cC)$ such that $a'\cong \overline{a}$, and pick a unitary isomorphism $\gamma_a: a'\to \overline{a}$.
For $f\otimes g \in \bfF(a)\otimes \bfG(a)$, we define
$$
( f \otimes g )^* := \bfF(\gamma_a^*)[j_a^\bfF(f)] \otimes \bfG(\gamma_a)[j_a^\bfG(g)].
$$
Notice that this definition is independent of the choice of the $\gamma_a$, since $\dim(\cC(a',\overline{a}))=1$ for every $a\in \Irr(\cC)$, and $\gamma_a$ appears with its adjoint $\gamma_a^*$.
Again, it is easy to check that this formula defines an involution, since $j^\bfF_{\overline{a}}\circ j^\bfF_{a}=\id_{a}$ by the properties of the $*$-structure.
\end{defn}

\begin{rem}
\label{rem:AlgebraicTopology}
The term `realization' comes from the geometric realization of a semi-simplicial set in algebraic topology.
Recall that a semi-simplicial set is a functor $X_\bullet: {\sf ss\Delta}^{\text{op}} \to {\sf Set}$.
Here ${\sf ss\Delta}$ is the semi-simplicial category, whose objects are $[n]$ for $n\geq 0$, and whose morphisms are generated by $d_i : [n]\to [n+1]$ for $0\leq i\leq n+1$ satisfying $d_j d_i = d_i d_{j-1}$ for $i<j$.
Writing $X_n$ for $X_\bullet([n])$, the realization is a quotient
$$
\coprod_{n\geq 0} X_n \times \Delta^n / \sim
$$
where $\Delta^n$ is the standard $n$-simplex.
We may think of $\Delta^\bullet$ as a functor ${\sf ss\Delta}\to {\sf Top}$ where the image of the morphism $d_i: [n]\to [n+1]$ is the face map which includes $\Delta^n$ into the $i$-th face of $\Delta^{n+1}$.
Thus the topological realization is formed by taking a coproduct over objects $[n]$ of the contravariant functor $X_\bullet$ applied to $[n]$ times the covariant functor $\Delta^\bullet$ applied to $[n]$, which resembles our definition above.

Indeed, both constructions described above are examples of \emph{coends}.
\end{rem}

We have a nice graphical calculus for the realization $*$-algebra $|\bfF\otimes \bfG|$, which is best defined by taking a coend over all objects of $\cC$, not just $\Irr(\cC)$.
We form a larger (non-asociative) algebra by
$$
\|\bfF\otimes \bfG\| := \bigoplus_{a \in\cC} \bfF(a) \otimes \bfG(a)
$$
together with the same unit.
However, multiplication and the $*$-structure are now given by
$$
( f^a \otimes g^a ) \cdot ( f^b \otimes g^b )
:=
\mu^{\bfF}_{a,b}(f^a \otimes f^b) \otimes \mu^{\bfG}_{a,b}(g^a \otimes g^b)
\qquad\qquad
( f \otimes g )^* := j_a^\bfF(f) \otimes j_a^\bfG(g)
$$
where we do not sum over isometry spaces.
When $\cC$ is non-strict, this algebra is non-associative, since generally
$$
\bfF((a\otimes b)\otimes c) \otimes \bfG((a\otimes b)\otimes c)
\neq
\bfF(a\otimes (b\otimes c)) \otimes \bfG(a\otimes (b\otimes c)).
$$
In the same vein, the unit does not really behave like a unit.
However, when $\cC$ is strict, this algebra is associative and unital.

We now look at the ideal $I$ generated by
$$
\{ \bfF(\psi)(f) \otimes g - f \otimes \bfG(\psi)(g)\,  |\,  f \in \bfF(a),  g \in \bfG(b),\text{ and }\psi \in \cC(a, b) \}.
$$

\begin{lem}
\label{lem:I maps to zero}
Suppose $c\in \Irr(\cC)$ and $a,b\in\cC$.
For all $\psi\in \cC(a,b)$, we have
$$
\sum_{\alpha \in \Isom(c,a)}\alpha^* \otimes (\psi \circ \alpha)
=
\sum_{\beta\in \Isom(c,b)}  (\beta^*\circ \psi)\otimes \beta 
$$
as elements of $\cC(a,c)\otimes \cC(c,b)$.
\end{lem}
\begin{proof}
We expand $\alpha \circ \psi$ in terms of $\Isom(c,b)$ to get
\begin{align*}
\sum_{\alpha \in \Isom(c,a)}\alpha^* \otimes (\psi\circ \alpha)
&=
d_c^{-1}
\sum_{
\substack{\alpha \in \Isom(c,a)
\\
\beta\in \Isom(c,b)
}}
\alpha^* \otimes \langle \psi\circ \alpha, \beta \rangle_{\cC(c,b)} \beta
\\&=
d_c^{-1}
\sum_{
\substack{\alpha \in \Isom(c,a)
\\
\beta\in \Isom(c,b)
}}
(\langle \psi^*\circ \beta, \alpha \rangle_{\cC(c,b)}\alpha)^* \otimes \beta
\\&=
\sum_{\beta \in \Isom(c,b)}(\psi^*\circ \beta)^* \otimes \beta
=
\sum_{\beta \in \Isom(c,b)}(\beta^*\circ \psi) \otimes \beta.
\qedhere
\end{align*}
\end{proof}

\begin{prop}
The map $\|\bfF \otimes \bfG\| \rightarrow |\bfF \otimes \bfG|$ given by
$$
\bfF(a)\otimes \bfG(a)\ni f \otimes g
\longmapsto
\sum_{c \in \Irr(\cC)} \sum_{\alpha\in \Isom(c,a)}
\bfF(\alpha^{*})(f) \otimes \bfG(\alpha)(g)
$$
descends to a unital isomorphism of $*$-algebras
$\|\bfF \otimes \bfG\|/I \cong |\bfF \otimes \bfG|$.
\end{prop}
\begin{proof}
The defined map is clearly surjective.
Thus we must show it descends to an injective $*$-homomorphism.

First, we note that the map is independent of the choices of isometry spaces by the discussion at the beginning of Section \ref{sec:Realizations}.
This means the map preserves the (non-associative) $*$-algebra structure, as the natural maps
$$
\Isom(c,a) \otimes \Isom(b,d) \otimes \Isom(e,c\otimes d) \to \Isom(e,a\otimes b)
$$
for $a,b\in \cC$ and $c,d,e\in\Irr(\cC)$ are isometries with respect to the normalized inner products.

To show that we actually get a homomorphism, we must show that elements of $I$ map to zero, i.e., elements of the form
$\bfF(\psi)(f) \otimes g - f \otimes \bfG(\psi)(g)$
map to zero for $f \in \bfF(a)$,  $g \in \bfG(b)$, and $\psi \in \cC(a,b)$.
This follows directly from Lemma \ref{lem:I maps to zero}.
Thus we get a well-defined surjective $*$-homomorphism $q:\|\bfF\otimes \bfG\|/I \rightarrow |\bfF \otimes \bfG|$.

To prove injectivity, we note there is a natural linear map $i:|\bfF \otimes \bfG| \rightarrow \|\bfF \otimes \bfG\|/I$ given by inclusion and applying the quotient map.
It is easy to see that $q\circ i = \id_{|\bfF\otimes \bfG|}$.
These maps are in fact inverses on the quotient; one checks this by checking the composite in the other order on simple tensors in $\bfF(a)\otimes \bfG(a)$ for $a\in \Irr(\cC)$.
\end{proof}

Using this $*$-isomorphism, we get the following graphical calculus for $|\bfF\otimes \bfG| \cong \|\bfF \otimes \bfG\|/I$.
We use the Yoneda embedding to identify
$\bfF(a)=\Hom_{\Vec(\cC^{\op})}(a, \bfF)$
and
$\bfG(a)=\Hom_{\Vec(\cC)}(a, \bfG)$.
We use string diagrams to draw elements of each of these spaces, but we use \emph{different conventions simultaneously} for each.
We draw elements of $\bfG(a)$ as string diagrams from $\mathbf{a}$ to $\bfG$ reading bottom to top.
We draw elements of $\bfF(a)$ as string diagrams from $\mathbf{a}$ to $\bfF$ reading top to bottom.
Given an element of $\bfF(a)\otimes \bfG(a)$, we connect both string diagrams by their $\mathbf{a}$ strings, but we draw a \emph{barrier membrane} between them, depicted below by the horizontal red line.
For instance, if $f\in \bfF(a)$ and $g\in \bfG(a)$, then
\begin{equation}
\label{eq:PermeableMembrane}
f\otimes g
=
\begin{tikzpicture}[baseline=-.1cm]
    \draw[red] (-.5,0) -- (.5,0);
    \draw (0,-1.2) -- (0,1.2);
    \roundNbox{unshaded}{(0,.6)}{.3}{0}{0}{$g$}
    \roundNbox{unshaded}{(0,-.6)}{.3}{0}{0}{$f$}
    \node at (.2,1.05) {\scriptsize{$\bfG$}};
    \node at (.15,.15) {\scriptsize{$\mathbf{a}$}};
    \node at (.15,-.15) {\scriptsize{$\mathbf{a}$}};
    \node at (.2,-1.05) {\scriptsize{$\bfF$}};
\end{tikzpicture}
\,.
\end{equation}
Now the barrier membrane is \emph{permeable} to morphisms from $\cC$.
For $f \in \bfF(a)$, $g\in \bfG(b)$, and $\psi\in \cC(a,b)$, we have
$$
\qquad
\qquad
\bfF(\psi)(f)\otimes g=
\begin{tikzpicture}[yscale=-1,baseline=-.7cm]
    \draw[red] (-.5,0) -- (.5,0);
    \draw (0,-1.2) -- (0,2.4);
    \roundNbox{unshaded}{(0,1.8)}{.3}{0}{0}{$f$}
    \roundNbox{unshaded}{(0,.6)}{.3}{0}{0}{$\psi$}
    \roundNbox{unshaded}{(0,-.6)}{.3}{0}{0}{$g$}
    \node at (.2,2.25) {\scriptsize{$\bfF$}};
    \node at (.2,1.2) {\scriptsize{$\mathbf{a}$}};
    \node at (.15,.15) {\scriptsize{$\mathbf{b}$}};
    \node at (.15,-.15) {\scriptsize{$\mathbf{b}$}};
    \node at (.2,-1.05) {\scriptsize{$\bfG$}};
\end{tikzpicture}
=
\begin{tikzpicture}[baseline=.5cm]
    \draw[red] (-.5,0) -- (.5,0);
    \draw (0,-1.2) -- (0,2.4);
    \roundNbox{unshaded}{(0,1.8)}{.3}{0}{0}{$g$}
    \roundNbox{unshaded}{(0,.6)}{.3}{0}{0}{$\psi$}
    \roundNbox{unshaded}{(0,-.6)}{.3}{0}{0}{$f$}
    \node at (.2,2.25) {\scriptsize{$\bfG$}};
    \node at (.2,1.2) {\scriptsize{$\mathbf{b}$}};
    \node at (.15,.15) {\scriptsize{$\mathbf{a}$}};
    \node at (.15,-.15) {\scriptsize{$\mathbf{a}$}};
    \node at (.2,-1.05) {\scriptsize{$\bfF$}};
\end{tikzpicture}
=
f\otimes \bfG(\psi)(g).
$$
Notice that pulling the $\psi$ through the membrane is compatible with swapping the covariant functor $\bfF: \cC \to \Vec$ with the contravariant functor $\bfG : \cC \to \Vec$.

Using this graphical notation, it is easy to express the multiplication and $*$-structure.
A lax monoidal functor $\bfG : \cC^{\op} \rightarrow \Vec$ is equivalent to an algebra in $\Vec(C)$, and thus a lax monoidal functor $\bfF : \cC \rightarrow \Vec$ is equivalent to an algebra in $\Vec(\cC^{\op})$.
If $f^a \otimes g^a \in \bfF(a)\otimes \bfG(a)$ and $f^b \otimes g^b \in \bfF(b) \otimes \bfG(b)$, then
the multiplication and star structure are given by
\begin{equation}
\label{eq:GraphicalMultiplicationAndStar}
(f^a \otimes g^a)\cdot(f^b \otimes g^b)=
\begin{tikzpicture}[baseline=-.1cm]
    \draw[red] (-.5,0) -- (1.5,0);
    \draw (0,-.6) -- (0,.6);
    \draw (1,-.6) -- (1,.6);
    \draw (0,.9) arc (180:0:.5cm);
    \draw (0,-.9) arc (-180:0:.5cm);
    \filldraw (.5,1.4) circle (.05cm);
    \filldraw (.5,-1.4) circle (.05cm);
    \draw (.5,1.4) -- (.5,1.8);
    \draw (.5,-1.4) -- (.5,-1.8);
    \roundNbox{unshaded}{(0,.6)}{.3}{0}{0}{$g^a$}
    \roundNbox{unshaded}{(0,-.6)}{.3}{0}{0}{$f^a$}
    \roundNbox{unshaded}{(1,.6)}{.3}{0}{0}{$g^b$}
    \roundNbox{unshaded}{(1,-.6)}{.3}{0}{0}{$f^b$}
    \node at (.7,1.6) {\scriptsize{$\bfG$}};
    \node at (-.2,1.1) {\scriptsize{$\bfG$}};
    \node at (1.2,1.1) {\scriptsize{$\bfG$}};
    \node at (-.15,.15) {\scriptsize{$\mathbf{a}$}};
    \node at (-.15,-.15) {\scriptsize{$\mathbf{a}$}};
    \node at (1.15,.15) {\scriptsize{$\mathbf{b}$}};
    \node at (1.15,-.15) {\scriptsize{$\mathbf{b}$}};
    \node at (-.2,-1.1) {\scriptsize{$\bfF$}};
    \node at (1.2,-1.1) {\scriptsize{$\bfF$}};
    \node at (.7,-1.6) {\scriptsize{$\bfF$}};
\end{tikzpicture}
\qquad
\qquad
(f^a \otimes g^a)^*
=
\begin{tikzpicture}[baseline=-.1cm]
    \draw[red] (-.5,0) -- (.5,0);
    \draw (0,-1.3) -- (0,1.3);
    \roundNbox{unshaded}{(0,.6)}{.3}{.4}{.4}{$j_a^\bfG(g^a)$}
    \roundNbox{unshaded}{(0,-.6)}{.3}{.4}{.4}{$j_a^\bfF(f^a)$}
    \node at (.2,1.1) {\scriptsize{$\bfG$}};
    \node at (.15,.15) {\scriptsize{$\overline{\mathbf{a}}$}};
    \node at (.15,-.15) {\scriptsize{$\overline{\mathbf{a}}$}};
    \node at (.2,-1.1) {\scriptsize{$\bfF$}};
\end{tikzpicture}.
\end{equation}

\subsection{Hilbert space realizations}

Suppose now we have two Hilbert space objects $\bfH \in \Hilb(\cC^{\op})$ and $\bfK \in \Hilb(\cC)$.
That is, we have two linear dagger functors $\bfH: \cC\to \Hilb$ and $\bfK : \cC^{\op} \to \Hilb$.
As in the previous section, we form the Hilbert space:
$$
|\bfH \otimes \bfK| =
\bigoplus_{c\in\Irr(\cC)} \bfH(c)\otimes_c \bfK(c)
$$
Here, we the $\boxtimes$ symbol denotes the Hilbert space tensor product with a normalized inner product as in \cite[Def.~2.16 and Def.~2.29]{MR3687214}.
For $c\in\Irr(\cC)$, the inner product on the $\bfH(c)\otimes_c \bfK(c)$ summand of $|\bfH\otimes \bfK|$ is given by
\begin{equation}
\label{eq:InnerProductForSimples}
\langle
\eta_1\otimes \xi_1
,
\eta_2\otimes \xi_2
\rangle_{\bfH(c)\otimes_c \bfK(c)}
=
\frac{1}{d_c}
\langle \eta_1, \eta_2\rangle_{\bfH(c)}
\langle \xi_1, \xi_2\rangle_{\bfK(c)}
=
\frac{1}{d_c}\,
\begin{tikzpicture}[baseline=-.1cm]
    \draw (0,-.8) -- (0,.8) arc (180:0:.3cm) -- (.6,-.8) arc (0:-180:.3cm);
    \roundNbox{unshaded}{(0,.5)}{.3}{0}{0}{$\eta_2^*$}
    \roundNbox{unshaded}{(0,-.5)}{.3}{0}{0}{$\eta_1$}
    \node at (-.15,1) {\scriptsize{$\mathbf{c}$}};
    \node at (-.2,0) {\scriptsize{$\bfH$}};
    \node at (-.15,-1) {\scriptsize{$\mathbf{c}$}};
\end{tikzpicture}
\,\cdot
\begin{tikzpicture}[baseline=-.1cm]
    \draw (0,-.8) -- (0,.8) arc (180:0:.3cm) -- (.6,-.8) arc (0:-180:.3cm);
    \roundNbox{unshaded}{(0,.5)}{.3}{0}{0}{$\xi_2^*$}
    \roundNbox{unshaded}{(0,-.5)}{.3}{0}{0}{$\xi_1$}
    \node at (-.15,1) {\scriptsize{$\mathbf{c}$}};
    \node at (-.2,0) {\scriptsize{$\bfK$}};
    \node at (-.15,-1) {\scriptsize{$\mathbf{c}$}};
\end{tikzpicture}
\end{equation}
We warn the reader that the above closed diagrams are in different categories!
The diagram on the left is in $\Hilb(\cC^{\op})$, and the diagram on the right is in $\Hilb(\cC)$.
Both of these diagrams are just numbers in $\bbC$, so we may multiply them.

We have a second realization of this Hilbert space similar to the previous section.
First, we form the (pre-Hilbert) vector space
$$
\|\bfH \otimes \bfK\| =
\bigoplus_{c\in \cC} \bfH(c)\otimes_c \bfK(c)
$$
as an algebraic direct sum.
Next, we define a positive semi-definite sesquilinear form as follows.
Suppose $a,b\in\cC$ are not necessarily simple, 
$\eta_1\otimes \xi_1 \in \bfH(a)\otimes_a \bfK(a)$, and 
$\eta_2\otimes \xi_2 \in \bfH(b)\otimes_b \bfK(b)$.
Recall that for all $c\in\cC$, 
$\bfH(c)\cong \Hom_{\Hilb(\cC^{\op})}(\mathbf{c}, \bfH)$
and
$\bfK(c)\cong \Hom_{\Hilb(\cC)}(\mathbf{c}, \bfK)$.
Thus we see that 
$\xi_2^*\circ \xi_1 \in \Hom_{\Hilb(\cC)}(\mathbf{a},\mathbf{b}) \cong \cC(a,b)$,
and
$\eta_2^*\circ \eta_1 \in \Hom_{\Hilb(\cC^{\op})}(\mathbf{a},\mathbf{b}) \cong \cC^{\op}(a,b)$, so $(\eta_2^*\circ \eta_1)^{\op} \in \cC(b,a)$.
We define
\begin{equation}
\label{eq:InnerProductForAll}
\langle
\eta_1\otimes \xi_1
,
\eta_2\otimes \xi_2
\rangle_{\|\bfH\otimes \bfK\|}
=
\tr_\cC(
\underbrace{(\xi_2^*\circ \xi_1)}_{\in\cC(a,b)}
\circ 
\underbrace{(\eta_2^*\circ \eta_1)^{\op}}_{\in \cC(b,a)}
),
\end{equation}
which is a scalar.
Notice that this inner product agrees with \eqref{eq:InnerProductForSimples} when $a=b=c\in\Irr(\cC)$.
By Lemma \ref{lem:I maps to zero}, the algebraic subspace $I\subset \|\bfH\otimes \bfK\|$ generated by
$$
\{ \bfH(\psi)(\eta) \otimes \xi - \eta \otimes \bfK(\psi)(\xi)\,  |\,  \eta \in \bfH(a),  \xi \in \bfK(b),\text{ and }\psi \in \cC(a, b) \}
$$
sits in the kernel of the sesquilinear form.
We obtain a Hilbert space isomorphism $|\bfH\otimes \bfK | \cong \|\bfH\otimes \bfK\|/I$, where we take the completion of the algebraic quotient on the right hand side.

This allows us to use a similar graphical calculus with a $\cC$-permeable membrane as in \eqref{eq:PermeableMembrane};
vectors $\eta\otimes \xi \in \bfH(c)\otimes_c \bfK(c)$ are denoted graphically by
$$
\eta\otimes \xi
=
\begin{tikzpicture}[baseline=-.1cm]
    \draw[red] (-.5,0) -- (.5,0);
    \draw (0,-1.2) -- (0,1.2);
    \roundNbox{unshaded}{(0,.6)}{.3}{0}{0}{$\xi$}
    \roundNbox{unshaded}{(0,-.6)}{.3}{0}{0}{$\eta$}
    \node at (.2,1.05) {\scriptsize{$\bfK$}};
    \node at (.15,.15) {\scriptsize{$\mathbf{c}$}};
    \node at (.15,-.15) {\scriptsize{$\mathbf{c}$}};
    \node at (.2,-1.05) {\scriptsize{$\bfH$}};
\end{tikzpicture}
\,.
$$

We conclude this section with the following lemmas about the graphical calculus and the inner product.

\begin{lem} 
The inner product is given graphically for $\eta_1\otimes \xi_1\in \bfH(a)\otimes_a \bfK(a)$ and $\eta_2\otimes \xi_2\in \bfH(b)\otimes_b \bfK(b)$ by
$$
\langle
\eta_1\otimes \xi_1
,
\eta_2\otimes \xi_2
\rangle_{|\bfH\otimes \bfK|}
=
\begin{tikzpicture}[baseline=-.1cm]
    \draw[red] (-.5,0) -- (.9,0);
    \draw (0,-1.9) -- (0,1.9) arc (180:0:.3cm) -- (.6,-1.9) arc (0:-180:.3cm);
    \roundNbox{unshaded}{(0,1.6)}{.3}{0}{0}{$\xi_2^*$}
    \roundNbox{unshaded}{(0,.6)}{.3}{0}{0}{$\xi_1$}
    \roundNbox{unshaded}{(0,-.6)}{.3}{0}{0}{$\eta_1$}
    \roundNbox{unshaded}{(0,-1.6)}{.3}{0}{0}{$\eta_2^*$}
    \node at (-.15,2.1) {\scriptsize{$\mathbf{b}$}};
    \node at (-.2,1.1) {\scriptsize{$\bfK$}};
    \node at (-.15,.15) {\scriptsize{$\mathbf{a}$}};
    \node at (-.15,-.15) {\scriptsize{$\mathbf{a}$}};
    \node at (-.2,-1.1) {\scriptsize{$\bfH$}};
    \node at (-.15,-2.1) {\scriptsize{$\mathbf{b}$}};
\end{tikzpicture}
\,.
$$
\end{lem}
\begin{proof}
Using \eqref{eq:HilbC-op Relation2}, the diagram on the right is given by
\begin{align*}
\bigg(
\coev_{\mathbf{b}}^* &\circ (\xi_2^* \otimes \id_{\overline{\mathbf{b}}}) \circ (\xi_1 \otimes \id_{\overline{\mathbf{b}}})
\bigg)
\circ
\bigg(
\coev_{\mathbf{b}}^{\op}\circ (\eta_2^* \otimes \id_{\overline{\mathbf{b}}})\circ (\eta_1 \otimes \id_{\overline{\mathbf{b}}})
\bigg)^{\op}
\\&=
\bigg(\coev_{\mathbf{b}}^* 
\circ 
((\xi_2^* \circ \xi_1) \otimes \id_{\overline{\mathbf{b}}})
\bigg)
\circ
\bigg(
(
(\eta_2^* \circ \eta_1)^{\op} \otimes \id_{\overline{\mathbf{b}}}
)
\circ
\coev_{\mathbf{b}}\bigg)
\\&=
\coev_{\mathbf{b}}^* 
\circ 
\bigg(
((\xi_2^* \circ \xi_1) \circ (\eta_2^* \circ \eta_1)^{\op})
\otimes \id_{\overline{\mathbf{b}}}
\bigg)
\circ
\coev_{\mathbf{b}}
\\&=
\tr_\cC((\xi_2^* \circ \xi_1) \circ (\eta_2^* \circ \eta_1)^{\op})
\\&=
\langle
\eta_1\otimes \xi_1
,
\eta_2\otimes \xi_2
\rangle_{|\bfH\otimes \bfK|}.
\qedhere
\end{align*}
\end{proof}

\begin{lem}
\label{lem:GraphicalInnerProduct}
Suppose $a,b,c\in \cC$.
For all $\eta_1\otimes \xi_1 \in \bfH(a\otimes b)\otimes_{a\otimes b} \bfK(a\otimes b)$
and $\eta_2\otimes \xi_2 \in \bfH(c)\otimes_c \bfK(c)$,
$$
\begin{tikzpicture}[baseline=-.1cm]
    \draw[red] (-.5,0) -- (.8,0);
    \draw (-.15,-.6) -- (-.15,.6);
    \draw (.15,-.6) -- (.15,.6);
    \draw (0,.6) -- (0,1.9) arc (180:0:.3cm) -- (.6,-1.9) arc (0:-180:.3cm) -- (0,-.6);
    \roundNbox{unshaded}{(0,1.6)}{.3}{0}{0}{$\xi_2^*$}
    \roundNbox{unshaded}{(0,.6)}{.3}{0}{0}{$\xi_1$}
    \roundNbox{unshaded}{(0,-.6)}{.3}{0}{0}{$\eta_1$}
    \roundNbox{unshaded}{(0,-1.6)}{.3}{0}{0}{$\eta_2^*$}
    \node at (-.15,2.1) {\scriptsize{$\mathbf{c}$}};
    \node at (-.2,1.1) {\scriptsize{$\bfK$}};
    \node at (-.3,.15) {\scriptsize{$\mathbf{a}$}};
    \node at (-.3,-.15) {\scriptsize{$\mathbf{a}$}};
    \node at (.3,.15) {\scriptsize{$\mathbf{b}$}};
    \node at (.3,-.15) {\scriptsize{$\mathbf{b}$}};
    \node at (-.2,-1.1) {\scriptsize{$\bfH$}};
    \node at (-.15,-2.1) {\scriptsize{$\mathbf{c}$}};
\end{tikzpicture}
=
\begin{tikzpicture}[baseline=-.1cm]
    \draw[red] (-.7,0) -- (1,0);
    \draw (-.15,-.3) arc (0:180:.15cm) -- (-.45,-1.9) arc (-180:0:.625cm) -- (.8,1.9) arc (0:180:.625cm) -- (-.45,.3) arc (-180:0:.15cm);
    \draw (.15,-.6) -- (.15,.6);
    \draw (0,.6) -- (0,1.9) arc (180:0:.3cm) -- (.6,-1.9) arc (0:-180:.3cm) -- (0,-.6);
    \roundNbox{unshaded}{(0,1.6)}{.3}{0}{0}{$\xi_2^*$}
    \roundNbox{unshaded}{(0,.6)}{.3}{0}{0}{$\xi_1$}
    \roundNbox{unshaded}{(0,-.6)}{.3}{0}{0}{$\eta_1$}
    \roundNbox{unshaded}{(0,-1.6)}{.3}{0}{0}{$\eta_2^*$}
    \node at (-.15,2.1) {\scriptsize{$\mathbf{c}$}};
    \node at (-.2,1.1) {\scriptsize{$\bfK$}};
    \node at (0,.15) {\scriptsize{$\mathbf{a}$}};
    \node at (0,-.15) {\scriptsize{$\mathbf{a}$}};
    \node at (-.6,1.1) {\scriptsize{$\overline{\mathbf{a}}$}};
    \node at (-.6,-1.1) {\scriptsize{$\overline{\mathbf{a}}$}};
    \node at (.3,.15) {\scriptsize{$\mathbf{b}$}};
    \node at (.3,-.15) {\scriptsize{$\mathbf{b}$}};
    \node at (-.2,-1.1) {\scriptsize{$\bfH$}};
    \node at (-.15,-2.1) {\scriptsize{$\mathbf{c}$}};
\end{tikzpicture}
\,.
$$
\end{lem}
\begin{proof}
Starting with the right hand side above, we have
\begin{align*}
&\tr_\cC\bigg(
(
(\id_{\overline{\mathbf{a}}}\otimes\xi_2)^*
\circ 
(\overline{\mathbf{a}}\otimes \bfK)(\ev^*_{\mathbf{a}}\otimes \id_{\mathbf{b}})
[\id_{\overline{\mathbf{a}}}\otimes \xi_1]
)
\circ
(\id_{\overline{\mathbf{a}}} \otimes\eta_2)^*
\circ 
(\overline{\mathbf{a}}\otimes \bfH)(\ev_{\mathbf{a}}\otimes \id_{\mathbf{b}})[\id_{\overline{\mathbf{a}}} \otimes \eta_1]
)^{\op} 
\bigg)
\\&=
\tr_\cC\bigg(
(
(\id_{\overline{\mathbf{a}}}\otimes\xi_2)^*
\circ 
(\id_{\overline{\mathbf{a}}}\otimes \xi_1)
\circ
(\ev^*_{\mathbf{a}}\otimes \id_{\mathbf{b}})
)
\circ
(\ev_{\mathbf{a}}\otimes \id_{\mathbf{b}})
\circ
((\id_{\overline{\mathbf{a}}} \otimes\eta_2)^*
\circ 
(\id_{\overline{\mathbf{a}}} \otimes \eta_1)
)^{\op} 
\bigg)
\\&=
\tr_\cC\bigg(
(
\id_{\overline{\mathbf{a}}}\otimes (\xi_2^*\circ\xi_1)
)
\circ
(\ev^*_{\mathbf{a}}\otimes \id_{\mathbf{b}})
\circ
(\ev_{\mathbf{a}}\otimes \id_{\mathbf{b}})
\circ
(
\id_{\overline{\mathbf{a}}} \otimes(\eta_2^*\circ \eta_1)
)^{\op} 
\bigg)
\\&=
\tr_\cC((\xi_2^*\circ \xi_1)\circ (\eta_2^*\circ \eta_1)^{\op} ).
\end{align*}
The first equality follows from \eqref{eq:HilbC-op Relation2},
the second equality follows from \eqref{eq:HilbC-op Relation1}, 
and the final equality follows using sphericality in $\cC$, since the morphisms $\xi_2^*\circ \xi_1 \in \cC(a\otimes b , c)$ and $(\eta_2^*\circ \eta_1)^{\op} \in \cC( c, a\otimes b)$ are just morphisms in $\cC$.
\end{proof}

\subsection{Representations of realizations}

Suppose we have $*$-algebra objects $\bfA\in \Vec(\cC^{\op})$ and $\bfB\in \Vec(\cC)$, together with bounded $*$-algebra representations
$\pi^\bfA: \bfA \Rightarrow \bfB(\bfH)$
and
$\pi^\bfB: \bfB \Rightarrow \bfB(\bfK)$
for Hilbert space objects 
$\bfH\in \Hilb(\cC^{\op})$
and
$\bfK\in \Hilb(\cC)$.
We define a bounded $*$-representation of the $*$-algebra $|\bfA \otimes \bfB|$ on the Hilbert space $|\bfH\otimes \bfK|$.

\begin{defn}
\label{defn:RepresentationsOfRealizations}
Given $f\otimes g \in \bfA(a) \otimes \bfB(a) \subset \|\bfA \otimes \bfB\|$,
we define its action on $\eta\otimes \xi\in \bfH(b)\otimes_b \bfK(b) \subset \|\bfH\otimes \bfK\|$ by
$$
(f\otimes g)\blacktriangleright (\eta\otimes \xi)
=
\begin{tikzpicture}[baseline=-.1cm]
    \draw[red] (-1.3,0) -- (.5,0);
    \draw (-.8,-.6) -- (-.8,.6);
    \draw (0,-1.8) -- (0,1.8);
    \draw (-.8,.9) .. controls ++(90:.4cm) and ++(180:.3cm) .. (0,1.4);
    \draw (-.8,-.9) .. controls ++(270:.4cm) and ++(180:.3cm) .. (0,-1.4);
    \node at (0,1.4) {$\blacktriangleright$};
    \node at (0,-1.4) {$\blacktriangleright$};
    \roundNbox{unshaded}{(-.8,.6)}{.3}{0}{0}{$g$}
    \roundNbox{unshaded}{(0,.6)}{.3}{0}{0}{$\xi$}
    \roundNbox{unshaded}{(0,-.6)}{.3}{0}{0}{$\eta$}
    \roundNbox{unshaded}{(-.8,-.6)}{.3}{0}{0}{$f$}
    \node at (.2,1.7) {\scriptsize{$\bfK$}};
    \node at (.2,1.1) {\scriptsize{$\bfK$}};
    \node at (-1,1.1) {\scriptsize{$\bfB$}};
    \node at (-.95,.15) {\scriptsize{$\mathbf{a}$}};
    \node at (-.95,-.15) {\scriptsize{$\mathbf{a}$}};
    \node at (.15,.15) {\scriptsize{$\mathbf{b}$}};
    \node at (.15,-.15) {\scriptsize{$\mathbf{b}$}};
    \node at (-1,-1.1) {\scriptsize{$\bfA$}};
    \node at (.2,-1.1) {\scriptsize{$\bfH$}};
    \node at (.2,-1.7) {\scriptsize{$\bfH$}};
\end{tikzpicture}
:=
\begin{tikzpicture}[baseline=-.1cm]
    \draw[red] (-1,0) -- (.5,0);
    \draw (-.7,-1.6) -- (-.7,1.6);
    \draw (0,-2.3) -- (0,2.3);
    \roundNbox{unshaded}{(0,1.6)}{.3}{.7}{0}{$\pi^\bfB_a(g)$}
    \roundNbox{unshaded}{(0,.6)}{.3}{0}{0}{$\xi$}
    \roundNbox{unshaded}{(0,-.6)}{.3}{0}{0}{$\eta$}
    \roundNbox{unshaded}{(0,-1.6)}{.3}{.7}{0}{$\pi^\bfA_a(f)$}
    \node at (.2,2.1) {\scriptsize{$\bfK$}};
    \node at (.2,1.1) {\scriptsize{$\bfK$}};
    \node at (-.85,.6) {\scriptsize{$\mathbf{a}$}};
    \node at (-.85,-.6) {\scriptsize{$\mathbf{a}$}};
    \node at (.15,.15) {\scriptsize{$\mathbf{b}$}};
    \node at (.15,-.15) {\scriptsize{$\mathbf{b}$}};
    \node at (.2,-1.1) {\scriptsize{$\bfH$}};
    \node at (.2,-2.1) {\scriptsize{$\bfH$}};
\end{tikzpicture}
\in \bfH(a\otimes b)\otimes_{a\otimes b} \bfK(a\otimes b).
$$
First, note that the algebraic subspace $I\subset \|\bfA\otimes \bfB\|$ always acts as zero, so we get a well-defined action of $|\bfA \otimes \bfB|$.
Second, note that the action of every $f\otimes g$ preserves the algebraic subspace $I\subset \|\bfH\otimes \bfK\|$, and thus descends to a well-defined map on the algebraic quotient $\|\bfH \otimes \bfK\|/I$.
Finally, when we consider the action of $|\bfA\otimes \bfB|$ on this algebraic quotient, since $\pi$ is a algebra representation, we see that for all $f_1\otimes g_1\in \bfA(a)\otimes \bfB(a)$, $f_2\otimes g_2\in \bfA(b)\otimes \bfB(b)$, and $\eta\otimes \xi\in \bfH(c)\otimes_c \bfK(c)$, we have
$$
(f_1\otimes g_1)
\blacktriangleright
[(f_2\otimes g_2)\blacktriangleright (\eta\otimes \xi)]
=
((f_1\otimes g_1)\cdot(f_2\otimes g_2))\blacktriangleright (\eta\otimes \xi).
$$
(This equality fails when considered as an action of the non-associative algebra $\|\bfA\otimes \bfB\|$ on $\|\bfH\otimes \bfK\|$, since $(a\otimes b)\otimes c \neq a\otimes (b\otimes c)$.)
\end{defn}

\begin{lem}
\label{lem:StarRepresentation}
Suppose
$f\otimes g\in \bfA(a)\otimes\bfB(a)$, $\eta_1\otimes \xi_1 \in \bfH(b)\otimes_b \bfK(b)$, and $\eta_2\otimes \xi_2 \in \bfH(c)\otimes_c \bfK(c)$.
Then
$ \langle (f \otimes g)\blacktriangleright (\eta_{1} \otimes \xi_{1}), \eta_{2} \otimes \xi_{2} \rangle
=
\langle \eta_{1} \otimes \xi_{1}, (j_a^\bfA(f)\otimes j_a^\bfB(g))\blacktriangleright (\eta_{2} \otimes \xi_{2})\rangle$.
\end{lem}
\begin{proof}
We use Lemma \ref{lem:GraphicalInnerProduct} together with the fact that $\pi^\bfA, \pi^\bfB$ are $*$-algebra natural transformations, to see that
$$
\langle (f \otimes g)\blacktriangleright (\eta_{1} \boxtimes \xi_{1}), \eta_{2} \boxtimes \xi_{2} \rangle
=
\begin{tikzpicture}[baseline=-.1cm]
    \draw[red] (-1,0) -- (1,0);
    \draw (-.7,-1.6) -- (-.7,1.6);
    \draw (0,-2.3) -- (0,2.3);
    \draw (0,2.9) arc (180:0:.3cm) -- (.6,-2.9) arc (0:-180:.3cm);
    \roundNbox{unshaded}{(0,2.6)}{.3}{0}{0}{$\xi_2^*$}
    \roundNbox{unshaded}{(0,1.6)}{.3}{.7}{0}{$\pi^\bfB_a(g)$}
    \roundNbox{unshaded}{(0,.6)}{.3}{0}{0}{$\xi_1$}
    \roundNbox{unshaded}{(0,-.6)}{.3}{0}{0}{$\eta_1$}
    \roundNbox{unshaded}{(0,-1.6)}{.3}{.7}{0}{$\pi^\bfA_a(f)$}
    \roundNbox{unshaded}{(0,-2.6)}{.3}{0}{0}{$\eta_2^*$}
    \node at (-.15,3.1) {\scriptsize{$\mathbf{c}$}};
    \node at (.2,2.1) {\scriptsize{$\bfK$}};
    \node at (.2,1.1) {\scriptsize{$\bfK$}};
    \node at (-.85,.6) {\scriptsize{$\mathbf{a}$}};
    \node at (-.85,-.6) {\scriptsize{$\mathbf{a}$}};
    \node at (.15,.15) {\scriptsize{$\mathbf{b}$}};
    \node at (.15,-.15) {\scriptsize{$\mathbf{b}$}};
    \node at (.2,-1.1) {\scriptsize{$\bfH$}};
    \node at (.2,-2.1) {\scriptsize{$\bfH$}};
    \node at (-.15,-3.1) {\scriptsize{$\mathbf{c}$}};
\end{tikzpicture}
=
\begin{tikzpicture}[baseline=-.1cm]
    \draw[red] (-1.3,0) -- (1.2,0);
    \draw (-.7,1.3) arc (0:-180:.3cm) -- (-1.3,2.9) .. controls ++(90:.7cm) and ++(90:.7cm) .. (.8,2.9) -- (.8,-2.9) .. controls ++(270:.7cm) and ++(270:.7cm) .. (-1.3,-2.9) -- (-1.3,-1.3) arc (180:0:.3cm);
    \draw (0,-2.3) -- (0,2.3);
    \draw (0,2.9) arc (180:0:.3cm) -- (.6,-2.9) arc (0:-180:.3cm);
    \roundNbox{unshaded}{(0,2.6)}{.3}{0}{0}{$\xi_2^*$}
    \roundNbox{unshaded}{(0,1.6)}{.3}{.7}{0}{$\pi^\bfB_a(g)$}
    \roundNbox{unshaded}{(0,.6)}{.3}{0}{0}{$\xi_1$}
    \roundNbox{unshaded}{(0,-.6)}{.3}{0}{0}{$\eta_1$}
    \roundNbox{unshaded}{(0,-1.6)}{.3}{.7}{0}{$\pi^\bfA_a(f)$}
    \roundNbox{unshaded}{(0,-2.6)}{.3}{0}{0}{$\eta_2^*$}
    \node at (-.15,3.1) {\scriptsize{$\mathbf{c}$}};
    \node at (.2,2.1) {\scriptsize{$\bfK$}};
    \node at (.2,1.1) {\scriptsize{$\bfK$}};
    \node at (-.55,1.1) {\scriptsize{$\mathbf{a}$}};
    \node at (-.55,-1.1) {\scriptsize{$\mathbf{a}$}};
    \node at (.15,.15) {\scriptsize{$\mathbf{b}$}};
    \node at (.15,-.15) {\scriptsize{$\mathbf{b}$}};
    \node at (.2,-1.1) {\scriptsize{$\bfH$}};
    \node at (.2,-2.1) {\scriptsize{$\bfH$}};
    \node at (-.15,-3.1) {\scriptsize{$\mathbf{c}$}};
\end{tikzpicture}
=
\begin{tikzpicture}[baseline=-.1cm]
    \draw[red] (-1.3,0) -- (1.2,0);
    \draw (-1.3,1.9) -- (-1.3,2.9) .. controls ++(90:.7cm) and ++(90:.7cm) .. (.8,2.9) -- (.8,-2.9) .. controls ++(270:.7cm) and ++(270:.7cm) .. (-1.3,-2.9) -- (-1.3,-1.9);
    \draw (0,-2.3) -- (0,2.3);
    \draw (0,2.9) arc (180:0:.3cm) -- (.6,-2.9) arc (0:-180:.3cm);
    \roundNbox{unshaded}{(0,2.6)}{.3}{0}{0}{$\xi_2^*$}
    \roundNbox{unshaded}{(0,1.6)}{.3}{1.4}{.1}{$\pi^\bfB_{\overline{a}}(j_a^\bfB(g))^*$}
    \roundNbox{unshaded}{(0,.6)}{.3}{0}{0}{$\xi_1$}
    \roundNbox{unshaded}{(0,-.6)}{.3}{0}{0}{$\eta_1$}
    \roundNbox{unshaded}{(0,-1.6)}{.3}{1.4}{.1}{$\pi^\bfA_{\overline{a}}(j_a^\bfA(f))^*$}
    \roundNbox{unshaded}{(0,-2.6)}{.3}{0}{0}{$\eta_2^*$}
    \node at (-.15,3.1) {\scriptsize{$\mathbf{c}$}};
    \node at (.2,2.1) {\scriptsize{$\bfK$}};
    \node at (.2,1.1) {\scriptsize{$\bfK$}};
    \node at (-.55,1.1) {\scriptsize{$\mathbf{a}$}};
    \node at (-.55,-1.1) {\scriptsize{$\mathbf{a}$}};
    \node at (.15,.15) {\scriptsize{$\mathbf{b}$}};
    \node at (.15,-.15) {\scriptsize{$\mathbf{b}$}};
    \node at (.2,-1.1) {\scriptsize{$\bfH$}};
    \node at (.2,-2.1) {\scriptsize{$\bfH$}};
    \node at (-.15,-3.1) {\scriptsize{$\mathbf{c}$}};
\end{tikzpicture}
\,.
$$
This last diagram is exactly 
$\langle \eta_{1} \otimes \xi_{1}, (j_a^\bfA(f)\otimes j_a^\bfB(g))\blacktriangleright (\eta_{2} \otimes \xi_{2})\rangle$.
\end{proof}

\begin{lem}
\label{lem:SchurProduct}
Suppose $A$ is a unital \emph{C*}-algebra with a faithful tracial state $\tr_A$.
Suppose $x,y\in M_n(A)$ are positive operators, and $x\star y$ is the Schur product given by $(x\star y)_{i,j} = x_{i,j}y_{i,j}$.
Then $(\tr_A\otimes \Tr_n)(x\star y) \geq 0$.
\end{lem}
\begin{proof}
First, we note that $x_{i,i},y_{i,i}\geq 0$ for all $i$.
Indeed, $x_{i,i}=e_i^* x e_i$ where $e_i\in M_{n\times 1}(A)$ is the column vector whose $i$-th entry is $1$ and all other entries are $0$, and similarly for $y$.
Thus we have
$(\tr_A\otimes \Tr_n)(x\star y) = \sum_{i=1}^n \tr_A(x_{i,i}y_{i,i})=\sum_{i=1}^n \tr_A(x_{i,i}^{1/2} y_{i,i}x_{i,i}^{1/2})\geq 0$.
\end{proof}

\begin{cor}
\label{cor:SchurProduct}
Suppose $w,x,y,z\in M_n(A)$, with $0\leq w\leq x$ and $0\leq y\leq z$.
Then $(\tr_A\otimes \Tr_n)(w\star y) \leq (\tr_A\otimes \Tr_n)(x\star z)$.
\end{cor}
\begin{proof}
Since $x-w\geq 0$, we have
$$
0\leq (\tr_A\otimes \Tr_n)((x-w)\star y) = (\tr_A\otimes \Tr_n)(x\star y) - (\tr_A\otimes \Tr_n)(w\star y).
$$
Similarly, we have $0\leq (\tr_A\otimes \Tr_n)(x\star z)-(\tr_A\otimes \Tr_n)(x\star y)$.
\end{proof}

\begin{thm}
\label{thm:BoundedRepresentationOfRealization}
The action $\blacktriangleright$ gives a unital $*$-algebra representation $|\bfA \otimes \bfB|\to B(|\bfH\otimes \bfK|)$.
\end{thm}
\begin{proof}
We show that the action of $(f\otimes g) \in \bfA(a)\otimes \bfB(a)\subset |\bfA \otimes \bfB|$ on $\|\bfH \otimes \bfK\|/I$ is bounded.
Suppose for $i=1,\dots, n$, we have $\eta_i\otimes \xi_i \in \bfH(b_i)\otimes_{b_i} \bfK(b_i)$, where $b_1,\dots, b_n\in \cC$ are objects which can occur with multiplicity.
Consider the finite direct sum $b = \bigoplus_{i=1}^n b_i \in \cC$. 
For $g\in \bfB(a)$, we get a positive operator in $M_n(\cC(a\otimes b,a\otimes b))$ given by 
\begin{equation}
\label{eq:B-part}
\left(
(\id_{\mathbf{a}}\otimes \xi_j^*) \circ \pi^\bfB_a(g)^* \circ \pi^\bfB_a(g) \circ (\id_{\mathbf{a}} \otimes \xi_i)
=
\begin{tikzpicture}[baseline=-.1cm]
    \draw (-.7,-2.1) -- (-.7,-.6);
    \draw (-.7,2.1) -- (-.7,.6);
    \draw (0,-2.1) -- (0,2.1);
    \roundNbox{unshaded}{(0,1.5)}{.3}{0}{0}{$\xi_j$}
    \roundNbox{unshaded}{(0,.5)}{.3}{.7}{0}{$\pi^\bfB_a(g)^*$}
    \roundNbox{unshaded}{(0,-.5)}{.3}{.7}{0}{$\pi^\bfB_a(g)$}
    \roundNbox{unshaded}{(0,-1.5)}{.3}{0}{0}{$\xi_i$}
    \node at (.2,2) {\scriptsize{$\mathbf{b}_j$}};
    \node at (-.85,1.5) {\scriptsize{$\mathbf{a}$}};
    \node at (.2,1) {\scriptsize{$\bfK$}};
    \node at (.2,0) {\scriptsize{$\bfK$}};
    \node at (.2,-1) {\scriptsize{$\bfK$}};
    \node at (-.85,-1.5) {\scriptsize{$\mathbf{a}$}};
    \node at (.2,-2) {\scriptsize{$\mathbf{b}_i$}};
\end{tikzpicture}
\right)_{i,j=1}^n
\end{equation}
where we suppress the canonical inclusions $b_i \hookrightarrow b$.
Since the $*$-representation $\pi^\bfB$ is bounded, the positive matrix \eqref{eq:B-part} is bounded above by the positive matrix
\begin{equation}
\label{eq:B-UpperBound}
\|\pi_a^\bfB(g)\|^2_{\Hilb(\cC)}
\bigg(
\id_{\mathbf{a}}\otimes (\xi_j^* \circ \xi_i)
\bigg)_{i,j=1}^n.
\end{equation}
Similarly, we have a positive operator in $M_n(\cC(a\otimes b,a\otimes b))$ given by
\begin{equation}
\label{eq:A-part}
\bigg(
[(\id_{\mathbf{a}}\otimes \eta_j^*) \circ \pi^\bfA_a(f)^* \circ \pi^\bfA_a(f) \circ (\id_{\mathbf{a}} \otimes \eta_i) ]^{\op}
\bigg)_{i,j=1}^n
\end{equation}
which is bounded above by
\begin{equation}
\label{eq:A-UpperBound}
\|\pi_a^\bfA(f)\|^2_{\Hilb(\cC^{\op})}
\bigg(
\id_{\mathbf{a}}\otimes (\eta_j^* \circ \eta_i)^{\op}
\bigg)_{i,j=1}^n.
\end{equation}
Notice now by \eqref{eq:InnerProductForAll} that
$\|(f\otimes g)\blacktriangleright \sum \eta_i\otimes \xi_i\|^2$
is exactly equal to the trace $\tr_\cC\otimes \Tr_n$ applied to the Schur product of \eqref{eq:B-part} and \eqref{eq:A-part} as described in Lemma \ref{lem:SchurProduct}.
Similarly, the trace $\tr_\cC\otimes \Tr_n$ applied to the Schur product of \eqref{eq:B-UpperBound} and \eqref{eq:A-UpperBound} is given by
$$
(\tr_\cC\otimes \Tr_n)(\eqref{eq:B-UpperBound} \star \eqref{eq:A-UpperBound})
=d_a\sum_{i,j}
\begin{tikzpicture}[baseline=-.1cm]
    \draw[red] (-.5,0) -- (1,0);
    \draw (0,-1.9) -- (0,1.9) arc (180:0:.3cm) -- (.6,-1.9) arc (0:-180:.3cm);
    \roundNbox{unshaded}{(0,1.6)}{.3}{0}{0}{$\xi_j^*$}
    \roundNbox{unshaded}{(0,.6)}{.3}{0}{0}{$\xi_i$}
    \roundNbox{unshaded}{(0,-.6)}{.3}{0}{0}{$\eta_i$}
    \roundNbox{unshaded}{(0,-1.6)}{.3}{0}{0}{$\eta_j^*$}
    \node at (-.15,2.1) {\scriptsize{$\mathbf{b}_j$}};
    \node at (-.2,1.1) {\scriptsize{$\bfK$}};
    \node at (-.15,.15) {\scriptsize{$\mathbf{b}_i$}};
    \node at (-.15,-.15) {\scriptsize{$\mathbf{b}_i$}};
    \node at (-.2,-1.1) {\scriptsize{$\bfH$}};
    \node at (-.15,-2.1) {\scriptsize{$\mathbf{b}_j$}};
\end{tikzpicture}
=
d_a\cdot
\left\|\sum_i \eta_i\otimes \xi_i\right\|^2.
$$
Now since \eqref{eq:B-UpperBound} and \eqref{eq:A-UpperBound} are upper bounds for \eqref{eq:B-part} and \eqref{eq:A-part}, by Corollary \ref{cor:SchurProduct}, we have
\begin{equation}
\label{eq:FinalUpperBound}
\|(f\otimes g)\blacktriangleright \sum \eta_i\otimes \xi_i\|^2
\leq
d_a
\cdot
\|\pi_a^\bfA(f)\|^2_{\Hilb(\cC^{\op})}
\cdot
\|\pi_a^\bfB(g)\|^2_{\Hilb(\cC)}
\cdot
\left\|\sum \eta_i\otimes \xi_i\right\|^2.
\end{equation}
We conclude that the representation is bounded.
It is a $*$-representation by Lemma \ref{lem:StarRepresentation}.
\end{proof}

\subsection{Universal realized C*-algebras}
\label{sec:Universal}

In the case that $\bfA \in \Vec(C^{\op})$ and $\bfB\in \Vec(\cC)$ are C*-algebra objects, the upper bound \eqref{eq:FinalUpperBound} from the proof of Theorem \ref{thm:BoundedRepresentationOfRealization} is bounded above by a constant independent of $\pi^\bfA : \bfA \Rightarrow \bfB(\bfH)$ and $\pi^\bfB : \bfB \Rightarrow \bfB(\bfK)$.
Indeed, for all $f\otimes g\in \bfA(a)\otimes \bfB(a)$, we have
\begin{equation}
\label{eq:UniversalBound}
\|\pi_a^\bfA(f)\|^2_{\Hilb(\cC^{\op})}
\leq
\|f\|^2_{\cM_\bfA(a_\bfA, 1_\bfA)}
\qquad
\text{and}
\qquad
\|\pi_a^\bfB(g)\|^2_{\Hilb(\cC)}
\leq
\|g\|^2_{\cM_\bfB(a_\bfB, 1_\bfB)}.
\end{equation}
This universal upper bound allows us to define a univeral C*-algebra.

\begin{defn}
Let $\Phi$ be the set of linear functionals $\phi:|\bfA \otimes \bfB|\rightarrow \C$ that appear as vector states in realized Hilbert space representations.  
The \emph{universal realized representation} of $|\bfA \otimes \bfB|$ is the Hilbert space
\begin{equation}
\label{eq:UniversalRealizedRepresentation}
\cH_\Phi:=\bigoplus_{\phi\in \Phi} L^2(|\bfA\otimes \bfB|, \phi)
\end{equation}
where $L^2(|\bfA\otimes \bfB|, \phi)$ is the GNS Hilbert space coming from the vector state $\phi$.
By \eqref{eq:UniversalBound}, the $*$-representation of $|\bfA \otimes \bfB|$ on $\cH_\Phi$ is bounded.
We define \emph{universal realized} C*-\emph{algebra} C*$|\bfA\otimes \bfB|$ is the norm closure of the image of $|\bfA\otimes \bfB|$ acting on the universal realized representation.
\end{defn}

\begin{cor} 
\label{cor:UniversalRepresentation}
Every realized Hilbert space representation of $|\bfA\otimes \bfB|$ extends to a bounded representation of \emph{C*}$|\bfA \otimes \bfB|$.
\end{cor}
\begin{proof}
Every representation of a C*-algebra is a direct sum of cyclic representations.
If $|\bfH\otimes \bfK|$ is a realized representation of $|\bfA \otimes \bfB|$, then $|\bfH\otimes \bfK|$ is isomorphic to a subrepresentation of some amplification of the universal realized representation \eqref{eq:UniversalRealizedRepresentation}.
Thus there is a $|\bfA\otimes \bfB|$-linear isometry $\iota: |\bfH\otimes \bfK| \to \cH_\Phi\otimes_\bbC K$.
The projection $p=\iota\iota^*$ onto $\iota|\bfH\otimes \bfK|$ lies in the commutant $|\bfA\otimes \bfB|'\cap B(\cH_\Phi\otimes_\bbC K)=\text{C*}|\bfA\otimes \bfB|'\cap B(\cH_\Phi\otimes_\bbC K)$.
Thus we get a bounded representation of C*$|\bfA\otimes \bfB|$ on $|\bfH\otimes \bfK|$ by $x\mapsto \iota^*x\iota$.
\end{proof}

We now provide some useful folklore results on working with universal C*-algebras.
We provide proofs for the convenience of the reader.
We thank Ben Hayes for the proof of Lemma \ref{lem:UniformlyBounded}.

Suppose we have a $*$-algebra $A$ together with a class of bounded Hilbert space $*$-representations $\sC$ such that for every $a\in A$, there is a universal constant $C$ such that $\|\pi(a)\|< C$ for every $(\pi,H) \in \sC$.
We define C*$(A)$ to be the universal C*-algebra for the class $\sC$.
One way to define C*$(A)$ is completion in the universal norm
$$
\|a\|_u := \sup\set{\|a\|_\pi}{(\pi, H)\in \sC}.
$$
Another is to take the set $\Phi$ of linear maps $A \to \bbC$ which are restrictions of vector states in some representation $(\pi, H) \in \sC$, and look at the C*-algebra generated by $A$ in the universal representation $H_\Phi := \bigoplus_{\phi \in \Phi} L^2(A, \phi)$.
The C*-algebra C*$(A)$ satisfies the universal property that for every $(\pi,H)\in \sC$, $\pi$ extends to C*$(A)$, and the C*-algebra $\text{C}_\pi^{*}(A)$ generated by $\pi(A)\subset B(H)$ is a quotient of C*$(A)$.

\begin{lem}
\label{lem:UniformlyBounded}
Suppose $(\pi, H) \in \sC$ and $a\in \pi(A)''$.
There is a net $(a_\lambda)\subset A$ with $\pi(a_\lambda) \to a$ strongly* such that 
$\|a_\lambda\|_{\text{\emph{C*}}(A)} \leq \|a\|_{B(H)}$ for all $\lambda$.
(If $H$ is separable, we may use a sequence instead of a net.)
\end{lem}
\begin{proof}
%
Let $q:\text{C*}(A)\to \text{C}^{*}_\pi(A)$ be the canonical surjection. 
Let $B^\circ$ be the open unit ball of C*$(A)$, let $B_\pi^\circ$ be the open unit ball of $\text{C}^{*}_\pi(A)$, let $B_\pi$ be the norm closure of $B_\pi^\circ$, and let $B''_\pi$ be the closed unit ball of $\pi(A)''$.
Since $q$ descends to an isomorphism C*$(A)/\ker(q) \cong \text{C}^*_\pi(A)$, $q(B^\circ) = B^\circ_\pi$, and by the Kaplanky Density Theorem \cite{MR1873025},
\begin{equation*}
\overline{q(B^\circ)}^{\text{strong*}}
=
\overline{B_\pi^\circ}^{\text{strong*}}
=
\overline{B_\pi}^{\text{strong*}}
=
B''_\pi.
\qedhere
\end{equation*}
\end{proof}

The main technical result of this section, which we will need later for our equivalence of categories, is the following proposition.

\begin{prop}
\label{prop:ExtendUCPMaps}
Let $(\pi_1, H_1), (\pi_2,H_2) \in \sC$, let $K$ be a Hilbert space, let $T: B(H_2) \to B(K)$ be a normal ucp map, and set $\theta = T\circ \pi_2 : A \to B(K)$.
Suppose there is a dense subspace $D\subset K$ such that for every vector state $\omega_\xi$ for $\xi\in D$, there is an $n\in\bbN$ and vectors $\eta_i, \zeta_i \in H_1$ for $i=1,\dots, n$ such that 
$\omega_\xi(\theta(a)) = \sum_{i=1}^n\langle \pi_1(a) \eta_i, \zeta_i\rangle$.
Then $\theta$ uniquely extends to a normal ucp map $\pi_1(A)'' \to B(K)$, which we still denote by $\theta$.
\end{prop}
\begin{proof}
Suppose $a\in \pi_1(A)''$, and pick a net $(a_\lambda)\subset A$ with $\pi_1(a_\lambda)\to a$ strongly* such that $\|a_\lambda\|_{\text{C*}(A)} \leq \|a\|_{B(H_1)}$ for all $\lambda$ as in Lemma \ref{lem:UniformlyBounded}.
By assumption, for all $\xi \in D$, there are $\eta_i, \zeta_i \in H_1$ for $i=1,\dots, n$ such that
$$
\omega_\xi(\theta(a_\lambda))
=
\sum_{i=1}^n\langle \pi_1(a_\lambda) \eta_i, \zeta_i\rangle
\longrightarrow
\sum_{i=1}^n\langle a \eta_i, \zeta_i\rangle.
$$
By polarization, for all $\eta, \xi\in D$, $\lim_\lambda (\langle \theta(a_\lambda)\eta, \xi\rangle)$ exists.
Since $T$ is ucp, $\|T\|\leq 1$, and thus for all $b\in A$, 
$$
\|\theta(b)\|_{B(K)} 
= 
\|T(\pi_2(b))\|_{B(K)}
\leq 
\|T\|\cdot \|\pi_2(b)\|_{B(H_2)}
\leq
\|b\|_{\text{C*}(A)}.
$$
Thus the sesquilinear form $s$ on $D\times D$ defined by $s(\eta, \xi) = \lim_\lambda \langle \theta(a_\lambda) \eta, \xi\rangle$ is bounded, since 
\begin{align*}
\sup_\lambda|\langle \langle \theta(a_\lambda) \eta, \xi\rangle| 
&\leq 
\sup_\lambda\|\theta(a_\lambda)\|_{B(K)}\cdot\|\eta\|\cdot \|\xi\|
\leq
\sup_\lambda
\|a_\lambda\|_{\text{C*}(A)}\cdot\|\eta\|\cdot \|\xi\|
\\&\leq 
\|a\|_{B(H_1)}\cdot\|\eta\|\cdot \|\xi\|.
\end{align*}
Thus $s$ extends uniquely to a bounded sesquilinear form on $K\times K$ with norm at most 
$\|a\|_{B(H_1)}$, which corresponds to a unique bounded operator in $B(K)$ which we call $\theta(a)$ which satisfies $\|\theta(a)\|_{B(K)}\leq 
\|a\|_{B(H_1)}$.

To show $\theta(a)$ is well-defined, we look at the vector states $\omega_\xi$ for $\xi \in D$.
If $(b_\lambda)\subset A$ is any net with $\pi_1(b_\lambda) \to a$ strongly*, then $\pi_1(a_\lambda - b_\lambda) \to 0$ strongly*, so again by assumption, 
$$
\omega_\xi(\theta(a_\lambda)) - \omega_\xi( \theta(b_\lambda) ) 
= 
\omega_\xi( \theta(a_\lambda - b_\lambda) ) 
=
\sum_{i=1}^n\langle \pi_1(a_\lambda-b_\lambda) \eta_i, \zeta_i\rangle
\longrightarrow
0.
$$
We conclude that 
$
\lim_\lambda\omega_\xi( \theta(a_\lambda))= \lim_\lambda\omega_\xi (\theta(b_\lambda)),
$
so $\theta(a)$ is well-defined.
Since addition, scalar multiplication, and the adjoint $*$ are strongly* continuous, the map $\theta: \pi_1(A)'' \to B(K)$ is a linear $*$-map.

Suppose now that $a\in \pi_1(A)''_+$.
Choose a net $(b_\lambda)\subset A$ with $\pi_1(b_\lambda) \to a^{1/2}$ strongly* such that $\|b_\lambda\|_{\text{C*}(A)} \leq \|a^{1/2}\|_{B(H_1)}$.
Then since multiplication is jointly strong*-continuous on norm bounded sets, 
the net $(a_\lambda = b_\lambda^*b_\lambda)\subset A$ satisfies $\pi_1(a_\lambda) = \pi_1(b_\lambda)^*\pi_1(b_\lambda) \to a$ strongly*, and $\|a_\lambda\|_{\text{C*}(A)} \leq \|a\|_{B(H_1)}$.
For all $\xi \in D$, by definition of $\theta(a)$, we have 
$$
\omega_\xi(\theta(a))
=
\lim_\lambda \omega_\xi(\theta(a_\lambda))
=
\lim_\lambda \omega_\xi(T(\pi_2(b_\lambda)^*\pi_2(b_\lambda)))
\geq 0.
$$
We conclude that $\theta(a)\geq 0$ by density of $D\subset K$.
Since the open unit ball of $M_n(\bbC)\otimes \pi_1(A)$ is strongly*-dense in the closed unit ball of $M_n(\bbC)\otimes \pi_1(A)''$, arguing as above using an amplification trick shows that $(\id_n\otimes \theta)(a)\geq 0$ for all $a\in [M_n(\bbC)\otimes \pi_1(A)'']_+$.

We now show $\theta: \pi_1(A)'' \to B(K)$ is normal.
Suppose $(a_\lambda)\subset \pi_1(A)''_+$ with $a_\lambda \nearrow a$.
Since $\theta$ is positive, we know $(\theta(a_\lambda))$ is an increasing net, which is bounded above, since
$$
\|\theta(a_\lambda)\|_{B(K)}
\leq 
\|a_\lambda\|_{B(H_1)} 
\leq
\|a\|_{B(H_1)}.
$$
To show $\theta(a_\lambda) \nearrow \theta(a)$, it suffices to show $\omega_\xi(\theta(a_\lambda)) \to \omega_\xi(\theta(a))$ for all $\xi \in D$ (for example, see \cite[Lem.~A.2]{MR3040370}).
Fix $\xi \in D$.
For all $b\in \pi_1(A)''$, taking $(b_\lambda)\subset A$ with $\pi_1(b_\lambda) \to b$ strongly*, we have
$$
\omega_\xi(\theta(b))
=
\lim_\lambda \omega_\xi(\theta(b_\lambda))
=
\lim_\lambda \sum_{i=1}^n\langle \pi_1(b_\lambda) \eta_i, \zeta_i\rangle
=
\sum_{i=1}^n\langle b \eta_i, \zeta_i\rangle.
$$
Since $a_\lambda \to a$ WOT, we have
$$
\omega_\xi(\theta(a_\lambda))
= 
\sum_{i=1}^n\langle a_\lambda \eta_i, \zeta_i\rangle 
\longrightarrow
\sum_{i=1}^n\langle a \eta_i, \zeta_i\rangle 
=
\omega_\xi(\theta(a)).
$$
The extension of $\theta$ is unique by strong*-density of $\pi_1(A) \subset \pi_1(A)''$.
\end{proof}

\begin{rems}
\label{rems:ExtendUCPMaps}
We have the additional following corollaries from Proposition \ref{prop:ExtendUCPMaps}.
\begin{enumerate}[(1)]
\item
Suppose $N\subset A$ is a $*$-subalgebra.
If $K$ is an $N-N$ bimodule and $T: B(H_2)\to B(K)$ is $N-N$ bilinear, then so is the extension $\theta: \pi_1(A)'' \to B(K)$.
First, note that $\theta: A \to B(K)$ is $N-N$ bilinear.
Let $a\in \pi_1(A)''$, and pick $(a_\lambda)\subset A$ with $\pi_1(a_\lambda) \to a$ strongly* with $\|a_\lambda\|_{\text{C*}(A)}\leq \|a\|$.
For all $n_1, n_2\in N$, $n_1\pi_1(a_\lambda) n_2\to n_1 a n_2$ strongly*, and thus for all $\xi \in H_2$,
$$
\omega_\xi( \theta(n_1 a n_2) )
=
\lim_\lambda \omega_\xi( \theta(n_1 a_\lambda n_2) )
=
\lim_\lambda \omega_\xi( n_1\theta( a_\lambda )n_2 )
=
\omega_\xi( n_1\theta( a )n_2 ).
$$
\item
If $T : B(H_2) \to B(K)$ is a $*$-homomorphism, then so is the extension $\theta: \pi_1(A)'' \to B(K)$.
To see this, it remains to show $\theta$ is multiplicative.
First, note that $\theta: A \to B(K)$ is multiplicative.
Let $a,b\in \pi_1(A)''$, pick $(a_\lambda)\subset A$ with $\pi_1(a_\lambda) \to a$ strongly* with $\|a_\lambda\|_{\text{C*}(A)}\leq \|a\|$, and pick $(b_\mu) \subset A$ with $\pi_1(b_\mu)\to b$ strongly* with $\|b_\mu\|_{\text{C*}(A)}\leq \|b\|$.
Since $\theta|_{\pi_1(A)}$ is multiplicative and $\theta$ is normal, for all $\eta,\xi\in H_2$,
\begin{align*}
\langle \theta(a)\theta(b) \eta, \xi\rangle
&=
\langle \theta(b) \eta, \theta(a)^*\xi\rangle
=
\lim_\mu
\langle \theta(b_\mu) \eta, \theta(a)^*\xi\rangle
\\&=
\lim_\mu
\langle \theta(a)\theta(b_\mu) \eta, \xi\rangle
=
\lim_\mu \lim_\lambda
\langle \theta(a_\lambda)\theta(b_\mu) \eta, \xi\rangle
\\&=
\lim_\mu \lim_\lambda
\langle \theta(a_\lambda b_\mu) \eta, \xi\rangle
=
\lim_\mu \lim_\lambda
\langle \theta(a b_\mu) \eta, \xi\rangle
=
\langle \theta(a b) \eta, \xi\rangle.
\end{align*}
\item
Proposition \ref{prop:ExtendUCPMaps} also applies to cp maps, not just ucp maps.
One must just include $\|T\|$ in the appropriate estimates.
\end{enumerate}
\end{rems}

\section{Connected algebras and discrete subfactors}

We now specialize to the case of a fully faithful bi-involutive representation $\bfH:\cC\to \spbfBim(N)$ and a connected W*-algebra object $\bfM \in \Vec(\cC)$.
Following Notation \ref{nota:RepresentationAndConnectedAlgebra}, we let $\bfH^\circ \in \Vec(\cC^{\op})$ be the W*-algebra object corresponding to taking bounded vectors of $\bfH$ (see Proposition \ref{prop:BoundedVectorsGivesAnAlgebra}).
Again, we warn the reader that by a slight abuse of notation, we will also use $\bfH$ to denote the composite of $\bfH : \cC \to \spbfBim(N)$ with the forget functor $\spbfBim(N) \to \Hilb$.

In this setting, the algebraic realization $|\bfM|_\bfH^\circ:=|\bfH^\circ\otimes \bfM|$ is very close to a von Neumann algebra.
Indeed, if $\bfM$ is compact and thus tracial by Proposition \ref{prop:Compact} (e.g., if $\cC$ is unitary fusion), then $|\bfM|_\bfH^\circ$ is actually a ${\rm II}_1$ factor without any completion necessary (see Corollary \ref{cor:AlreadyFactor}).
Moreover, the inclusion $N \subseteq |\bfM|_\bfH^\circ$ is a finite index inclusion of ${\rm II}_1$ factors with index equal to the dimension of the algebra object $\bfM \in \cC^{\natural}$.

When $\bfM$ is not finitely supported, $|\bfM|_\bfH^\circ$ is not a von Neumann algebra, and we must complete to obtain a von Neumann algebra $|\bfM|_\bfH$.
Moreover, the inclusion $N\subseteq |\bfM|_\bfH$ is no longer finite index, but it is extremal, irreducible, and \emph{discrete}/\emph{quasi-regular}.
We will show that all extremal irreducible discrete inclusions arise in this way.

This allows us to connect to the work of Popa-Shlyakhtenko-Vaes on quasi-regular inclusions of ${\rm II}_1$ factors \cite{1511.07329}, Popa's symmetric enveloping inclusion \cite{MR1302385,MR1729488}, and approximation properties for subfactors (see Section \ref{sec:AnalyticProperties}).
We also obtain a better understanding of the diagrammatic reproof of Popa's celebrated subfactor reconstruction theorem \cite{MR1334479} due to Guionnet-Jones-Shlyakhtenko \cite{MR2732052,MR2645882}, which produces an interpolated free group factor from any subfactor planar algebra \cite{MR2807103,MR3110503} (see Section \ref{sec:PlanarAlgebras}).

\subsection{Representations of connected algebras}

We continue the use of Notation \ref{nota:RepresentationAndConnectedAlgebra}.

\begin{defn}
The \emph{algebraic realization} of $\bfM$ in the representation $\bfH$ is the $*$-algebra 
$$
|\bfM|_{\bfH}^\circ
=
|\bfH^\circ\otimes \bfM|
=
\bigoplus_{c\in\Irr(\cC)} \bfH(c)^\circ\otimes \bfM(c).
$$
\end{defn}

We now analyze two canonical representations of $|\bfM|_{\bfH}^\circ$, and we show they are canonically isomorphic.
Taking the bicommutant then gives us the \emph{von Neumann realization} $|\bfM|_\bfH$.

First, since $\bfM$ is connected, it has a canonical state given by $i_\bfM \mapsto 1_\bbC$.
Thus the canonical state on $\bfM(1_\cC)$ is given by $\lambda i_\bfM\mapsto \lambda$.
In turn, we get a canonical isomorphism $N \cong \bfH^\circ(1_\cC)\otimes \bfM(1_\cC)\subset |\bfM|_\bfH^\circ$ given by $n \mapsto n\Omega \otimes i_\bfM$.
In particular, every element of $\bfH^\circ(1_\cC)\otimes \bfM(1_\cC)$ is of the form $n\Omega \otimes i_\bfM$.
This leads us to the following convention.

\begin{nota}
\label{nota:Sweedler}
For each $x\in |\bfM|_\bfH^\circ$, we use Sweedler notation to denote the $c$-component of $x$ for each $c\in\Irr(\cC)$.
That is, we write
$x = \sum_{c\in \Irr(\cC)} x_{(1)}^c \otimes x_{(2)}^c$,
where $x_{(1)}^c \otimes x_{(2)}^c \in \bfH(c)^\circ \otimes \bfM(c)$ is a finite sum of elementary tensors, which is non-zero for at most finitely many $c\in \Irr(\cC)$.
Moreover, by the preceding discussion, we may always assume that $x_{(1)}^{1_\cC}\otimes x_{(2)}^{1_\cC}$ is a single elementary tensor of the form $x^{1_\cC}\Omega \otimes i_\bfM$ for a unique $x^{1_\cC}\in N$.
\end{nota}

We now define the GNS representation of $|\bfM|_\bfH^\circ$ coming from the canonical trace $\tau$ on $N$.
Following Notation \ref{nota:Sweedler}, we define
$|\tau|_{\bfH}^\circ: |\bfM|_{\bfH}^\circ\rightarrow \C$ by
$$
|\tau|_{\bfH}^\circ\left(
\sum_{a\in \text{Irr}(\cC)} \eta^{a}_{(1)}\otimes f^{a}_{(2)}
\right)
=
\tau(x^{1_\cC}).
$$
In order to take the Hilbert space completion with respect to $|\tau|_\bfH^\circ$, we need to show $|\tau|_\bfH^\circ$ is positive on $|\bfM|_\bfH^\circ$ and that the action by left multiplication is by bounded operators.
To do this, we describe $|\tau|_\bfH^\circ$ in a second way.

As in Definition \ref{defn:RightL2}, we may view $\bfM$ as the Hilbert space object $L^2(\bfM)\in \Hilb(\cC)$ by considering the $\bfM(1_\cC)\cong \bbC$ valued inner product on each $\bfM(c)$ given by
$$
\langle f| g\rangle_{c}
=
\begin{tikzpicture}[baseline = .1cm, xscale=-1]
    \draw (-.5,-.3) arc (-180:0:.5cm);
    \draw (-.5,.3) arc (180:0:.5cm);
    \filldraw (0,.8) circle (.05cm);
    \draw (0,.8) -- (0,1.2);
    \roundNbox{unshaded}{(-.5,0)}{.3}{0}{0}{$g$}
    \roundNbox{unshaded}{(.5,0)}{.3}{.2}{.2}{$j_c(f)$}
    \node at (-.65,-.5) {\scriptsize{$\mathbf{c}$}};
    \node at (-.7,.5) {\scriptsize{$\bfM$}};
    \node at (.65,-.5) {\scriptsize{$\overline{\mathbf{c}}$}};
    \node at (.7,.5) {\scriptsize{$\bfM$}};
    \node at (.2,1) {\scriptsize{$\bfM$}};
\end{tikzpicture}
=
\bfM(\ev_c^*)[\mu_{\overline{c},c}^\bfM(j_c^\bfM(f)\otimes g)].
$$
(Recall that we are using the \emph{right} $L^2(\bfM)$, and in general this will not agree with the left $L^2(\bfM)$ unless $\bfM$ is tracial.)
We now get a $*$-representation of $|\bfM|^{\circ}_{\bfH}=|\bfH^\circ \otimes \bfM|$ on $|\bfH \otimes L^2(\bfM)|$ from Definition \ref{defn:RepresentationsOfRealizations}, which is bounded by Theorem \ref{thm:BoundedRepresentationOfRealization}.

\begin{lem}
\label{lem:TauInnerProduct}
For $x,y\in |\bfM|^{\circ}_{\bfH}$, $\langle y | x\rangle_{|\bfH\otimes L^2(\bfM)|} =|\tau|_{\bfH}^\circ(y^{*} x)$.
\end{lem}
\begin{proof}
Following Notation \ref{nota:Sweedler}, let $x=\sum_{a\in\Irr(\cC)} x^{a}_{(1)}\otimes x^{a}_{(2)}$ and $y=\sum_{b\in\Irr(\cC)} y^{b}_{(1)}\otimes y^{b}_{(2)}$.
Since $\bfH$ is full, for all $a,b\in\Irr(\cC)$, 
$\Hom_{\Bim(N)}(L^2(N),\bfH(a)\boxtimes_N \bfH(b))$ is at most one dimensional, and it is non-zero precisely when $a\cong \overline{b}$.
As in Definition \ref{defn:AlgebraRealization}, for each $a\in \Irr(\cC)$, let $a'\in \Irr(\cC)$ such that $a'\cong \overline{a}$, and pick a unitary isomorphism $\gamma_a: a'\to \overline{a}$.
For all $a\in\Irr(\cC)$, $\Hom_{\Bim(A)}(L^{2}(N), \bfH(a')\boxtimes_N \bfH(a))$ is spanned by the isometry $\frac{1}{\sqrt{d_a}} \bfH((\gamma_a^*\boxtimes \id_a)\circ\ev^*_{a})$.
We now compute that
\begin{align*}
(y^*x)^{1_\cC}
&=
\sum_{a\in\Irr(\cC)}
\frac{1}{d_a}
\bfH^\circ(\ev_a\circ (\gamma_a\otimes \id_a))[\mu^{\bfH^\circ}_{a',a}(\bfH^\circ(\gamma_a^*)[j_{a}^{\bfH^\circ}(y_{(1)}^{a})] \otimes x_{(1)}^a)]
\\&\hspace{2cm}\otimes 
\bfM((\gamma_a^*\otimes \id_a)\circ\ev^*_{a})[\mu_{a',a}(\bfM(\gamma_a)[j^\bfM_{a}(y_{(1)}^{a})]\otimes x_{(2)}^a)]
\\&=
\sum_{a\in\Irr(\cC)}
\frac{1}{d_a}
\bfH^\circ(\ev_a)[\mu^{\bfH^\circ}_{\overline{a},a}(j_{a}^{\bfH^\circ}(y_{(1)}^{a}) \otimes x_{(1)}^a)]
\otimes
\bfM(\ev^*_{a})[\mu_{\overline{a},a}(j^\bfM_{a}(y_{(1)}^{a})\otimes x_{(2)}^a)]
\\&=
\sum_{a\in \Irr(\cC)} 
\frac{1}{d_a}
\langle y_{(1)}^a | x_{(1)}^a\rangle_N^{\bfH(a)}
\langle y_{(2)}^a | x_{(2)}^a\rangle_a^{\bfM(a)}
\end{align*}
where the last equality used Lemma \ref{lem:HPreservesEvaluation} and \eqref{eq:EvAndCoev}.
Now applying $|\tau|_\bfH^\circ$ to both sides, we obtain
\begin{align*}
|\tau|_\bfH^\circ (y^*x)
&=
\sum_{a\in \Irr(\cC)} 
\frac{1}{d_a}
\tau(\langle y_{(1)}^a | x_{(1)}^a\rangle_N^{\bfH(a)})
\langle y_{(2)}^a | x_{(2)}^a\rangle_a^{\bfM(a)}
\\&=
\sum_{a\in \Irr(\cC)} 
\frac{1}{d_a}
\langle y_{(1)}^a | x_{(1)}^a\rangle_{\bfH(a)}
\langle y_{(2)}^a | x_{(2)}^a\rangle_{L^2(\bfM)(a)}
=
\langle y|x\rangle_{|\bfH\otimes L^2(\bfM)|}.
\qedhere
\end{align*}
\end{proof}

This lemma has the following immediate corollaries.

\begin{cor}
\label{cor:TauPositivity}
For all $x\in |\bfM|_{\bfH}^\circ$, $0\le |\tau|_{\bfF}^\circ(x^{*}x)$.
Moreover, $|\tau|_{\bfH}^\circ(x^{*}x)=0$ implies $x=0$.
\end{cor}
\begin{proof}
This follows from the previous lemma, and the fact that our pre-Hilbert space $|\bfM|_\bfH^\circ\Omega$ is a direct sum of dense subspaces of Hilbert spaces.
Thus there are no non-trivial length zero vectors in $|\bfM|_\bfH^\circ$.
\end{proof}

Theorem \ref{thm:BoundedRepresentationOfRealization} and Lemma \ref{lem:TauInnerProduct} immediate imply the following corollary.

\begin{cor}
\label{cor:Bounded}
For all $x,y\in |\bfM|^{\circ}_{\bfH}$, $|\tau|_{\bfH}^\circ(y^{*} x^{*} x y)\le C_x \cdot |\tau|_{\bfH}^\circ(y^{*} y)$ for some constant $C_x \geq 0$.
\end{cor}

We now perform the GNS construction to obtain the Hilbert space completion of $|\bfM|^\circ_\bfH$ with respect to $\|\cdot\|_2$, which we denote by $L^{2}|\bfM|_{\bfH}$.
By Corollaries \ref{cor:TauPositivity} and \ref{cor:Bounded}, the left regular action of $|\bfM|_{\bfH}^\circ$ on itself extends to a faithful $*$-homomorphism
$|\bfM|_{\bfH}^\circ\rightarrow B(L^{2}|\bfM|_{\bfH})$.

\begin{prop}
\label{prop:L2MH decomposition}
We have an $|\bfM|_\bfH^\circ-N$ bimodule isomorphism $L^2|\bfM|_\bfH \cong |\bfH \otimes L^2(\bfM)|$.
In particular, $L^{2}|\bfM|_{\bfH}\cong \bigoplus_{a\in \Irr(\cC)} \bfH(a)\otimes_a \bfM(a)$ as an $N-N$ bimodule.
\end{prop}
\begin{proof}
This follows immediately from identifying $|\bfM|_\bfH^\circ= \bigoplus_{c\in \Irr(\cC)} \bfH(c)^\circ \otimes \bfM(c)$ as a dense subspace of $|\bfH\otimes L^2(\bfM)| = \bigoplus_{c\in\Irr(\cC)} \bfH(c) \otimes_c L^2(\bfM)(c)$.
It is straightforward to show this identification intertwines the left $|\bfM|_\bfM^\circ$ and right $N$-actions, as $N=\bfH(1_\cC)^\circ \otimes \bfM(1_\cC)$.
\end{proof}

\begin{defn}
The \emph{von Neumann realization} of $\bfM$ with respect to $\bfH$ is the von Neumann algebra $|\bfM|_\bfH = (|\bfM|_\bfH^\circ)'' \subseteq B(L^2|\bfM|_\bfH)$.
\end{defn}

\subsection{The discrete inclusion of the von Neumann realization}
\label{sec:vonNeumannRealization}

We continue the use of Notation \ref{nota:RepresentationAndConnectedAlgebra}.
Recall $|\bfM|_\bfH$ is the von Neumann algebra realization of $\bfM$ in the representation $\bfH$.  The following is the main theorem of this section.  

\begin{thm}
\label{thm:Quasiregular}
The inclusion $N\subseteq |\bfM|_\bfH$ is an irreducible discrete inclusion of factors.
\end{thm}

Let $e_N\in B(L^2|\bfM|_\bfH)$ be the orthogonal projection with range $\bfH(1_\cC)\otimes \bfM(1_\cC) \cong L^2(N)$.

\begin{lem}
For all $x\in|\bfM|_\bfH$, $e_Nxe_N = E_N(x)e_N$ for a unique $E_N(x)\in N$.
The assignment $x\mapsto E_N(x)$ is a normal conditional expectation $|\bfM|_\bfH \to N$.
\end{lem}
\begin{proof}
For $x\in |\bfM|_\bfF$, the operator $e_Nxe_N\in B(L^2|\bfM|_\bfH)$ can be identified with an operator on $L^2(N)$ which is easily seen to commute with the right $N$-action.
Thus $e_Nxe_N$ is equal to some operator $E_N(x)\in N$ times $e_N$.
Now the formula $x\mapsto E_N(x)$ is easily seen to define a normal conditional expectation, since it is implemented by $e_N$.
(In fact, viewing $e_N$ as a map $L^2|\bfM|_\bfH \to L^2(N)$, we have that $E_N(x) = e_N x e_N^*$, which is known to be completely positive c.f.~Stinespring dilation.)
\end{proof}

It is now easy to see that on $|\bfM|_\bfH^\circ$, $\tau\circ E_N = |\tau|_\bfH^\circ$.
Indeed, if $\Lambda\subset \Irr(\cC)$ is a finite subset such that $1_\cC\in\Lambda$, then
$$
E_N\left(\sum_{a\in \Lambda} x^{a}_{(1)}\otimes x^{a}_{(2)}\right)=x^{1_\cC}\otimes i_\bfM.
$$
Thus setting $|\tau|_\bfH=\tau\circ E_N$, we have that $|\tau|_\bfH$ is equal to the vector state on $|\bfM|_\bfF$ corresponding to $\Omega \otimes i_\bfM \in \bfH(1_\cC)\otimes L^2(\bfM)(1_\cC)\subset L^2|\bfM)|_\bfH$, and is thus a normal state.

\begin{prop}
The conditional expectation $E_N$ is faithful, so $|\tau|_\bfH$ is faithful.
\end{prop}
\begin{proof}
For all $c\in\Irr(\cC)$, let $\cB_c = \{\beta_{c,i}\}\subset \bfH^\circ(c)\otimes \bfM(c)$ be an orthogonal right $N$ basis for the $N-N$ bimodule summand $\bfH(c)\otimes \bfM(c)$ of $L^2|\bfM|_\bfH$.
This means we can express $1_{B(L^2|\bfM|_\bfH)}$ as a sum of orthogonal projections $\sum_{c\in\Irr(\cC)} \sum_{\beta_{c,i}\in \cB_c} L_{\beta_{c,i}}L_{\beta_{c,i}}^*$, where the convergence is in the strong operator topology.
Moreover, these elements satisfy
$$
\langle \beta_{c,i}|\beta_{c',i'}\rangle_N=L_{\beta_{c,i}}^*L_{\beta_{c',i'}} = E_N(\beta_{c,i}^* \beta_{c',i'}) = \delta_{(c,i)=(c',i')}p_{c,i},
$$
which is a projection in $N$, non-zero exactly when $c=c'$ and $i=i'$.

Using this basis, we may express every
$x\Omega\in|\bfM|_\bfH\Omega$ uniquely as an infinite sum
\begin{equation}
\label{eq:L2Convergence}
x\Omega
=
\sum_{c\in\Irr(\cC)} \sum_{\beta_{c,i}\in \cB_c} L_{\beta_{c,i}}L_{\beta_{c,i}}^*(x\Omega)
=
\sum_{c\in\Irr(\cC)}\sum_{\beta_{c,i}\in \cB_c} \beta_{c,i} n_{c,i}\Omega
\end{equation}
where $n_{c,i}=p_{c,i}n_{c,i}\in N$.
Indeed, we calculate $n_{c,i} = E_N(\beta_{c,i}^* x)$ for all $c,i$, since $x\Omega = \sum_{c\in\Irr(\cC)}\sum_{\beta_{c,i}\in \cB_c} \beta_{c,i} E_N(\beta_{c,i}^* x)\Omega$.

We claim that $E_N(x^*x) = 0$ implies $x=0$.
To see this, we calculate that for all $n\in N$,
$$
\langle E_N(x^*x)n\Omega, n\Omega\rangle_{L^2(N)}
=
\tau_N(n^*E_N(x^*x)n)
=
\phi(n^*x^*xn)
=
\langle xn\Omega , xn\Omega\rangle_{L^2(M,\phi)}.
$$
Now using the boundedness of the right $N$-action, the above inner product is equal to
\begin{align*}
\sum_{c,c'\in\Irr(\cC)}
\sum_{\substack{\beta_{c,i}\in\cB_c
\\
\beta_{c',i'}\in\cB_{c'}}}
\langle \beta_{c,i}n_{c,i}n\Omega, \beta_{c',i'}n_{c',i'n\Omega} \rangle_{L^2(M,\phi)}
&=
\sum_{c,c'\in\Irr(\cC)}
\sum_{\substack{\beta_{c,i}\in\cB_c
\\
\beta_{c',i'}\in\cB_{c'}}}
\phi(n^*n_{c',i'}^*\beta_{c',i'}^*\beta_{c,i}n_{c,i}n)
\\&=
\sum_{c,c'\in\Irr(\cC)}
\sum_{\substack{\beta_{c,i}\in\cB_c
\\
\beta_{c',i'}\in\cB_{c'}}}
\tau(n^*n_{c',i'}^*E_N(\beta_{c',i'}^*\beta_{c,i})n_{c,i}n)
\\&=
\sum_{c\in\Irr(\cC)}\sum_{\beta_{c,i}\in\cB_c}
\tau(n^*n_{c,i}^*n_{c,i}n).
\end{align*}
Thus if $E_N(x^*x)=0$, then for all $n\in N$, every $n_{c,i}n$ is zero.
Since this holds for all $n\in N$, we have every $n_{c,i}$ is zero.
Thus $x\Omega =0$, and so $x = 0$ as $\Omega$ is separating for $|\bfM|_\bfH$. 
The result follows.
\end{proof}

\begin{cor}
\label{cor:IsoToGNS}
The $|\bfM|_\bfH-N$ bimodule $L^2|\bfM|_\bfH$ is isomorphic to the GNS Hilbert space $L^2(|\bfM|_\bfH, |\tau|_\bfH)$.
\end{cor}
\begin{proof}
From \eqref{eq:L2Convergence}, it follows that the $N$-coefficients of $x\in|\bfM|_\bfH$ define an $L^2$-convergent sequence.
Thus the image of $|\bfM|_\bfH^\circ$ in $L^2(|\bfM|_\bfH, |\tau|_\bfH)$ is dense in $\|\cdot \|_2$.
Since $|\tau|_\bfH$ restricted to $|\bfM|_\bfH^\circ$ is equal to $|\tau|_\bfH^\circ$, we see that the two Hilbert spaces may be identified.
It is also easy to see that the left $|\bfM|_\bfH$ actions agree.
The right $N$-actions agree by \eqref{eq:RightAction}.
\end{proof}

\begin{proof}[Proof of Theorem \ref{thm:Quasiregular}]
We already know that $E_N: |\bfM|_\bfH \to N$ is a faithful, normal conditional expectation.
By Corollary \ref{cor:IsoToGNS}, we have that $L^2|\bfM|_\bfH$ carries a right $|\bfM|_\bfH$-action from the Tomita-Takesaki theory, which is the commutant of the left $|\bfM|_\bfH$-action.

Consider the canonical map
\begin{equation}
\label{eq:InjectRelativeCommutant}
N'\cap |\bfM|_\bfH
\cong
\End_{N-|\bfM|_\bfH}(L^2|\bfM|_\bfH)
\longrightarrow
\Hom_{N-N}(L^2(N), L^2|\bfM|_\bfH)
\end{equation}
given by mapping $\psi\in \End_{N-|\bfM|_\bfH}(L^2|\bfM|_\bfH)$ to the composite
$$
L^2(N) \xrightarrow{e_N^*} L^2|\bfM|_\bfH 
\cong 
L^2|\bfM|_\bfH\boxtimes_{|\bfM|_\bfH} L^2|\bfM|_\bfH  
\xrightarrow{\psi\boxtimes \id} 
L^2|\bfM|_\bfH\boxtimes_{|\bfM|_\bfH} L^2|\bfM|_\bfH 
\cong 
L^2|\bfM|_\bfH.
$$
Notice that this map is injective by the faithfulness of $E_N$.
Indeed, if $(\psi\boxtimes \id)e_N^* = 0$, then $e_N (\psi^*\boxtimes \id)(\psi\boxtimes \id)e_N^* = E_N(\psi^*\psi\boxtimes \id)e_N=0$.
(Notice that $\psi^*\psi \boxtimes \id \in |\bfM|_\bfH$ as it commutes with the right $|\bfM|_\bfH$ action.)
Thus $E_N(\psi^*\psi\boxtimes \id)=0$ as $N$ is a factor and $e_N \in N'\cap B(L^2|\bfM|_\bfH)$. 
But then $\psi^*\psi\boxtimes \id = 0$ since $E_N$ is faithful, and fusion with $\id$ is known to be injective (for example, see \cite[Lem.~2.13]{MR3663592}).

Now by Proposition \ref{prop:L2MH decomposition}, we have that
$L^2|\bfM|_\bfH\cong \bigoplus_{c\in \Irr(\cC)} \bfH(c)\otimes \bfM(c)$.
Since $\bfM$ is connected, we know that $\bfM(1_\cC) \cong \bbC$, and thus
$\Hom_{N-N}(L^2(N), L^2|\bfM|_\bfH) \cong \bbC$.
Thus \eqref{eq:InjectRelativeCommutant} is an injective map $N'\cap |\bfM|_\bfH \to \bbC$, and we conclude $N\subseteq |\bfM|_\bfH$ is an irreducible inclusion of factors.
Finally, the $N-N$ bimodule decomposition of $L^2|\bfM|_\bfH$ is as desired by Proposition \ref{prop:L2MH decomposition} and Corollary \ref{cor:IsoToGNS}.
\end{proof}

\begin{rem}
Notice that we cannot use Theorem \ref{thm:BifiniteFrobeniusReciprocity} to prove Theorem \ref{thm:Quasiregular}, since the former theorem requires that we already know that the inclusion $N\subseteq |\bfM|_\bfH$ is irreducible.
\end{rem}

\subsection{Modular theory for von Neumann realizations}
\label{sec:MoreModularTheory}

By Theorem \ref{thm:Quasiregular}, we know that $(N\subseteq |\bfM|_\bfH , E_N)$ is an irreducible discrete inclusion of factors with $N$ type ${\rm II}_1$.
Hence by Corollary \ref{cor:IrreducibleInclusionIIorIII}, $|\bfM|_\bfH$ is either type ${\rm II}_1$ or type ${\rm III}$ depending on whether $|\tau|_{\bfH}$ is tracial.

\begin{prop}
\label{prop:TracialRealization}
The canonical state $|\tau|_{\bfH}$ is tracial on $|\bfM|_{\bfH}$ if and only if $\bfM$ is tracial.
\end{prop}
\begin{proof}
Suppose the canonical state $|\tau|_{\bfH}$ on $|\bfM|_{\bfH}$ is tracial.  
For $a\in \Irr(\cC)$, 
let $\xi_1\otimes f_1, \xi_2\otimes f_2 \in \bfH(a)^\circ \otimes \bfM(a)$.
By \eqref{eq:ModularOperator} and Lemma \ref{lem:TauInnerProduct}, we calculate
\begin{align*}
|\tau|_\bfH(
(\overline{\xi_1} \otimes j_a(f_1))
\cdot 
(\xi_2 \otimes f_2)
)
&=
d_a^{-1} \langle \xi_1, \xi_2\rangle_{\bfH(a)}\langle f_1 |f_2\rangle_a 
\\
|\tau|_\bfH(
(\xi_2 \otimes f_2)
\cdot
(\overline{\xi}_1 \otimes j_a(f_1))
)
&=
d_a^{-1} \langle \overline{\xi}_2, \overline{\xi}_1\rangle_{\bfH(\overline{a})} \langle j_a(f_1) |j_a(f_2)\rangle_{\overline{a}}
=
d_a^{-1} \langle \xi_1, \xi_2\rangle_{\bfH(a)}  \langle f_2,f_1\rangle_{a}.
\end{align*}
Thus picking $\xi_1=\xi_2$ to be non-zero, we see that $|\tau|_\bfH$ tracial implies $\bfM$ is tracial.
Conversely, if $\bfM$ is tracial, then $|\tau|_\bfH^\circ$ is tracial on $|\bfM|_\bfH^\circ$, since if $a, b\in\Irr(\cC)$ are distinct, the $\bfH(a)\otimes_a L^2(\bfM)(a)$ and $\bfH(b)\otimes_b L^2(\bfM)(b)$ summands of $L^2|\bfM|_\bfH$ are orthogonal.
By normality, $|\tau|_\bfH$ is tracial on $|\bfM|_\bfH$.
\end{proof}

\begin{cor}
\label{cor:AlreadyFactor}
If $\bfM$ is compact, then $|\bfM|_\bfH^\circ$ is a ${\rm II}_1$ factor before taking the completion.
\end{cor}
\begin{proof}
By Proposition \ref{prop:Compact}, $\bfM$ is tracial, and thus $|\bfM|_\bfH$ is a ${\rm II}_1$ factor by Proposition \ref{prop:TracialRealization}.
Moreover, $L^2|\bfM|_\bfH$ is compact, and thus a bifinite $N-N$ bimodule.
Thus $|\bfM|_\bfH$ is equal to the $L^2|\bfM|_\bfH^\circ$, the set of $N$-bounded vectors in $L^2|\bfM|_\bfH$ by \cite{MR1049618}.
Finally, by \cite{MR2501843} and Proposition \ref{prop:L2MH decomposition}, $L^{2}|\bfM|_{\bfH}^\circ\cong \bigoplus_{a\in \Irr(\cC)} (\bfH(a)\otimes_a \bfM(a))^\circ  =  \bigoplus_{a\in \Irr(\cC)} \bfH(a)^\circ\otimes \bfM(a)=|\bfM|_\bfH^\circ$.
\end{proof}

We will see in Section \ref{sec:QuantumGroups} that $|\tau|_{\bfH}$ fails to be a trace in general, e.g., there are examples coming from non Kac-type discrete quantum groups.
We now relate modular the spectral spectra $S(\bfM)$ and $S(|\bfM|_\bfH)$.

\begin{prop}\label{prop:ModularInvariant}
$S(|\bfM|_{\bfH})=S(\bfM)$.
\end{prop}
\begin{proof}
Since $(N\subseteq |\bfM|_\bfH, E_N)$ is an irreducible discrete inclusion, $\Delta_{|\tau|_{\bfH}}$ is diagonalizable by Corollary \ref{cor:QuasiPeriodic}.  
In fact, a straightfoward computation shows that $\Delta_{|\tau|_{\bfH}}|_{ \bfH(a)\otimes_a L^{2}(\bfM)(a)}=\id_{\bfH(a)}\otimes \Delta_{a}$, where $\Delta_{a}$ is the modular operator for $\bfM$ from Section \ref{sec:ModularTheory}.
By Proposition \ref{prop:L2MH decomposition}, $L^2|\bfM|_\bfH\cong \bigoplus_{a\in \Irr(\cC)} \bfH(a) \otimes_a L^{2}(\bfM)(a)$, and thus the set of eigenvalues of $\Delta_{|\tau|_{\bfH}}$ is equal to the union of the sets of eigenvalues of the operators $\Delta_{a}\otimes \id_{\bfH(a)}$ for $a\in\Irr(\cC)$. 
Since $\Delta_{|\tau|_{\bfH}}$ is diagonalizable, $\Spec(\Delta_{|\tau|_{\bfH}})$ is the closure of the set of eigenvalues, which is $S(\bfM)$ by definition. 
By \cite[Cor.~3.2.7(a)]{MR0341115}, we have $S(|\bfM|_{\bfH})=\Spec(\Delta_{|\tau|_{\bfH}})=S(\bfM).$ 
\end{proof}

\subsection{The realization functor}

We now extend $|\cdot |_\bfH$ to a functor.
We continue the use of Notation \ref{nota:RepresentationAndConnectedAlgebra}.
As we will need to have multiple connected W*-algebra objects in $\Vec(\cC)$  simultaneously, we will denote these by $\bfA$ and $\bfB$ instead of $\bfM$.

\begin{defn} 
We define the following two categories.
\begin{itemize}
\item
Let $\ConAlg$ be the category whose objects are connected W*-algebra objects in $\Vec(\cC)$ and whose morphims are categorical ucp morphisms.
Recall that a $*$-natural transformation $\theta: \bfA \Rightarrow \bfB$ is called cp if for all $c\in \cC$, $\theta_{\overline{c}\otimes c} : \bfA(\overline{c}\otimes c) \to \bfB(\overline{c}\otimes c)$ maps positive operators to positive operators, and it is called ucp if moreover $\theta_{1_\cC} (i_\bfA) = i_\bfB$.
\item
Let $\DisInc$ be the category whose objects are irreducible discrete inclusions $(N\subseteq M, E)$, where $N$ is our fixed ${\rm II}_1$ factor with canonical trace $\tau$, and whose morphims are $N-N$ bilinear ucp maps which preserve the canonical state.
This means that a morphism $\psi \in \DisInc((N\subseteq M, E_N^M), (N\subseteq P, E_N^P))$ is an $N-N$ bilinear ucp map $\psi:M \to N$ satisfying 
$$
\phi_P \circ \psi=\tau \circ E^P_N \circ \psi = \tau \circ E^M_N = \phi_M.
$$
Let $\DisInc_\bfH$ be the full subcategory consisting of those $(N\subseteq M, E)$ such that $L^2(M,\phi)$ is supported on $\bfH(\cC)$, i.e., $L^2(M,\phi)\in \Hilb(\bfH(\cC))$.
\end{itemize}
\end{defn}

Given an $\bfA \in \ConAlg$, we have already shown how to produce an irreducible, discrete inclusion $N\subseteq |\bfA|_{\bfH}$.  
We now decribe how to extend $|\cdot|_\bfH$ to a functor $\ConAlg\rightarrow \DisInc_\bfH$.

Suppose $\bfA$, $\bfB$ are connected W*-algebra objects.
Let $\theta:\bfA\Rightarrow\bfB$ be a categorical ucp morphism, and note that we may think of $\theta$ as a ucp morphism $\bfA\Rightarrow \bfB(L^2(\bfB))$.
By the Stinespring Dilation Theorem in $\cC$ \cite[Thm.~4.28]{MR3687214}, there is a $\bfK\in \Hilb(\cC)$, a $*$-representation $\pi^\bfA:\bfA\Rightarrow\bfB(\bfK)$, and an isometry $v:L^{2}(\bfB)\Rightarrow \bfK$ such that $\theta = \Ad(v)\circ \pi$.

Now, by Theorem \ref{thm:BoundedRepresentationOfRealization}, we get a bounded $*$-representation of the algebraic realization $|\bfA|_\bfH^\circ$ on the Hilbert space realization $|\bfH \otimes \bfK|$, which we denote by $|\pi|_\bfH$. 

\begin{lem}
\label{lem:RealizedIsometry}
The map $|v|_\bfH:L^{2}|\bfB|_{\bfH}\rightarrow |\bfH \otimes \bfK|$ defined by $|v|_\bfH(\eta\otimes \xi)=\eta\otimes (v\circ \xi)$ uniquely extends to an isometry satisfying $|v^{*}|_{\bfH}=|v|^{*}_{\bfH}$, where $|v^{*}|_{\bfH}$ is defined analogously.
\end{lem}
\begin{proof}
For finitely many $\eta_i\otimes \xi_i \in \bfH(a_i)\otimes L^2(\bfB)(a_i)$, by \eqref{eq:InnerProductForAll},
\begin{align*}
\left\||v|_\bfH\left(\sum_i\eta_i\otimes \xi_i\right)\right\|^2_{|\bfH\otimes \bfK|} 
&=
\left\|\sum_i \eta_i\otimes v(\xi_i)\right\|^2_{|\bfH\otimes \bfK|} 
=
\sum_{i,j}
\tr_\cC( 
\underbrace{(\xi_j^*\circ v^*\circ v\circ \xi_i)}_{\in\cC(a_i,a_j)}
\circ 
\underbrace{(\eta_j^*\circ \eta_i)^{\op}}_{\in \cC(a_j,a_i)}
)
\\&=
\sum_{i,j}
\tr_\cC( 
(\xi_j^*\circ \xi_i)
\circ 
(\eta_j^*\circ \eta_i)^{\op}
)
=
\left\|\sum_i\eta_i\otimes \xi_i\right\|^2_{|\bfH\otimes \bfK|},
\end{align*}
and thus $|v|_\bfH$ uniquely extends to an isometry.
We calculate for all $\eta_1\otimes \xi_1 \in \bfH(a)\otimes L^2(\bfB)(a)$ and $\eta_2\otimes \xi_2 \in \bfH(b)\otimes \bfK(b)$,
$$
\langle |v|_\bfH(\eta_1 \otimes \xi_1), \eta_2\otimes \xi_2\rangle_{|\bfH \otimes \bfK|}
=
\tr_\cC( 
(\xi_2^*\circ v\circ \xi_1)
\circ 
(\eta_2^*\circ \eta_1)^{\op}
)
=
\langle \eta_1 \otimes \xi_1,|v^*|_\bfH( \eta_2\otimes \xi_2)\rangle_{|\bfH \otimes L^2(\bfB)|},
$$
and thus $|v|_\bfH^* = |v^*|_\bfH$.
\end{proof}

\begin{defn}
\label{defn:RealizationOfUCPMap}
Similar to our definition of $|v|_\bfH$, we can define $|\theta|_\bfH: |\bfA|_\bfH^\circ\rightarrow |\bfB|_\bfH^\circ$ by $|\theta|_\bfH(\eta\otimes f)=\eta \otimes \theta_{a}(f)$ for $\eta\otimes f\in \bfH(a)^{\circ} \otimes \bfA(a)$.
\end{defn}

\begin{prop}
\label{prop:AlgebraicStinespring}
As maps $|\bfA|_\bfH^\circ \to |\bfB|_\bfH^\circ$, we have
$|\theta|_\bfH=\Ad(|v|_\bfH)\circ |\pi|_{\bfH}$.
\end{prop}
\begin{proof}
Recall that for all $g\in \bfA(a)$, we identify $\theta_a(g) \in \bfB(L^2(\bfB))(a)$ by
$$
\begin{tikzpicture}[baseline=-.1cm]
	\draw (0,-.8) -- (0,.8);
	\draw (-.6,-.8) -- (-.6,0);
	\roundNbox{unshaded}{(0,0)}{.3}{.6}{0}{$\theta_a(g)$}
	\node at (.5,-.6) {\scriptsize{$L^2(\bfB)$}};
	\node at (-.75,-.6) {\scriptsize{$\mathbf{a}$}};
	\node at (.5,.6) {\scriptsize{$L^2(\bfB)$}};
\end{tikzpicture}
=
\begin{tikzpicture}[baseline=-.1cm]
	\draw (0,-1.5) -- (0,1.5);
	\draw (-.6,-1.5) -- (-.6,0);
	\roundNbox{unshaded}{(0,.95)}{.25}{0}{0}{$v^*$}
	\roundNbox{unshaded}{(0,0)}{.3}{.6}{0}{$\pi^\bfA_a(g)$}
	\roundNbox{unshaded}{(0,-.95)}{.25}{0}{0}{$v$}
	\node at (.5,-1.4) {\scriptsize{$L^2(\bfB)$}};
	\node at (-.75,-.95) {\scriptsize{$\mathbf{a}$}};
	\node at (.15,-.5) {\scriptsize{$\bfK$}};
	\node at (.15,.5) {\scriptsize{$\bfK$}};
	\node at (.5,1.4) {\scriptsize{$L^2(\bfB)$}};
\end{tikzpicture}
\,.
$$
Thus we calculate that for $f \otimes g\in \bfH^\circ(a) \otimes \bfA(a)$ and $\eta\otimes \xi\in \bfH(b) \otimes L^2(\bfB)(b)$, we have
$$
|\theta|_\bfH(f \otimes g) \blacktriangleright (\eta\otimes \xi)
=
\begin{tikzpicture}[baseline=-.1cm]
    \draw[red] (-1,0) -- (.5,0);
    \draw (-.7,-1.6) -- (-.7,1.6);
    \draw (0,-2.3) -- (0,2.3);
    \roundNbox{unshaded}{(0,1.6)}{.3}{.7}{0}{$\theta_a(g)$}
    \roundNbox{unshaded}{(0,.6)}{.3}{0}{0}{$\xi$}
    \roundNbox{unshaded}{(0,-.6)}{.3}{0}{0}{$\eta$}
    \roundNbox{unshaded}{(0,-1.6)}{.3}{.7}{0}{$\pi^{\bfH^\circ}_a(f)$}
    \node at (.5,2.1) {\scriptsize{$L^2(\bfB)$}};
    \node at (.5,1.1) {\scriptsize{$L^2(\bfB)$}};
    \node at (-.85,.6) {\scriptsize{$\mathbf{a}$}};
    \node at (-.85,-.6) {\scriptsize{$\mathbf{a}$}};
    \node at (.15,.15) {\scriptsize{$\mathbf{b}$}};
    \node at (.15,-.15) {\scriptsize{$\mathbf{b}$}};
    \node at (.2,-1.1) {\scriptsize{$\bfH$}};
    \node at (.2,-2.1) {\scriptsize{$\bfH$}};
\end{tikzpicture}
=
\begin{tikzpicture}[baseline=.9cm]
    \draw[red] (-1,0) -- (.5,0);
    \draw (-.7,-1.6) -- (-.7,2.6);
    \draw (0,-2.3) -- (0,4.3);
    \roundNbox{unshaded}{(0,3.6)}{.3}{0}{0}{$v^*$}
    \roundNbox{unshaded}{(0,2.6)}{.3}{.7}{0}{$\pi^\bfA_a(g)$}
    \roundNbox{unshaded}{(0,1.6)}{.3}{0}{0}{$v$}
    \roundNbox{unshaded}{(0,.6)}{.3}{0}{0}{$\xi$}
    \roundNbox{unshaded}{(0,-.6)}{.3}{0}{0}{$\eta$}
    \roundNbox{unshaded}{(0,-1.6)}{.3}{.7}{0}{$\pi^{\bfH^\circ}_a(f)$}
    \node at (.5,4.1) {\scriptsize{$L^2(\bfB)$}};
    \node at (.2,3.1) {\scriptsize{$\bfK$}};
    \node at (.2,2.1) {\scriptsize{$\bfK$}};
    \node at (.5,1.1) {\scriptsize{$L^2(\bfB)$}};
    \node at (-.85,.6) {\scriptsize{$\mathbf{a}$}};
    \node at (-.85,-.6) {\scriptsize{$\mathbf{a}$}};
    \node at (.15,.15) {\scriptsize{$\mathbf{b}$}};
    \node at (.15,-.15) {\scriptsize{$\mathbf{b}$}};
    \node at (.2,-1.1) {\scriptsize{$\bfH$}};
    \node at (.2,-2.1) {\scriptsize{$\bfH$}};
\end{tikzpicture}
\,.
$$
The right hand side above is equal to 
$\Ad(|v|_\bfH)[ |\pi|_\bfH)(f \otimes g)]\blacktriangleright (\eta\otimes \xi)$ 
by Lemma \ref{lem:RealizedIsometry}.
\end{proof}

This proposition has the following two nice corollaries.

\begin{cor}
\label{cor:PositivityOfTheta}
The map $|\theta|_\bfH:|\bfA|^{\circ}_{\bfH}\rightarrow |\bfB|_\bfH\subset B(L^{2}|\bfB|_{\bfH})$ uniquely extends to a ucp map on the universal realized \emph{C*}-algebra \emph{C*}$|\bfH^\circ \otimes \bfA|$
from Corollary \ref{cor:UniversalRepresentation}, which we will again denote by $|\theta|_\bfH$.
\end{cor}
\begin{proof}
The Stinespring representation $|\pi|_\bfH: |\bfA|_\bfH^\circ \to B(|\bfH\otimes\bfK|)$ extends to the universal realized C*-algebra C*$|\bfH^\circ \otimes \bfA|$ by Corollary \ref{cor:UniversalRepresentation}, and thus so does the bounded ucp map $\Ad(|v|_\bfH)\circ |\pi|_{\bfH} : |\bfA|_\bfH^\circ \to B(L^2|\bfB|_\bfH)$.
\end{proof}

\begin{rem}
Since $|\theta|_\bfH$ is ucp on C*$|\bfH^\circ \otimes \bfA|$, it satisfies Kadison's inequality on $|\bfA|_\bfH^{\circ}$.
That is, for all $x\in|\bfA|_\bfH^{\circ}$, 
$|\theta|_\bfH(x^{*})|\theta|_\bfH(x)\leq |\theta|_\bfH(x^{*}x)$ in $|\bfB|_\bfH\subset B(L^2|\bfB|_\bfH)$ \cite[Cor.~IV.3.8]{MR1873025}.
\end{rem}


\begin{lem}
\label{lem:StatePreserving}
The map $|\theta|_\bfH$ preserves the canonical states of $|\bfA|^\circ_\bfH$ and $|\bfB|_\bfH^\circ$.
\end{lem}
\begin{proof}
Since $\bfA(1_\cC) \cong\bbC \cong \bfB(1_\cC)$ and since $\theta(i_\bfA) = i_\bfB$, 
for all $f\in \bfA(1_\cC)$, $\phi_\bfA(f) = \phi_\bfB(\theta_{1_\cC}(f))$.
Thus for $x=\sum_{a\in\Irr(\cC)} x^a_{(1)}\otimes x^a_{(2)} \in \bfA(a)\otimes \bfH^\circ(a) \subset |\bfA|_\bfH^\circ$,
we have $|\tau|_\bfH^\circ(x)=\tau(x^{1_\cC})=(|\tau|_\bfH^\circ\circ |\theta|_\bfH)(x)$.
\end{proof}

We may now refer to the canonical states of $|\bfA|^\circ_\bfH$ and $|\bfB|_\bfH^\circ$ both simultaneously as $|\tau|_\bfH^\circ$ and no confusion can arise.
Applying $|\tau|_\bfH^\circ$ to both sides of Kadison's inequality together with Lemma \ref{lem:StatePreserving} yields the following.

\begin{cor}
\label{cor:UCPGivesContraction}
The map $|\theta|_\bfH$ extends to a contraction $L^2|\theta|_\bfH: L^{2}|\bfA|_{\bfH}\rightarrow L^{2}|\bfB|_{\bfH}$ satisfying $L^2|\theta|_\bfH(x\Omega) = |\theta|_\bfH(x)\Omega$ for all $x\in |\bfA|_{\bfH}^\circ$.
\end{cor}
\begin{proof}
For all $x\in |\bfA|_\bfF^\circ$, we have
\begin{align*}
\|L^2|\theta|_\bfH(x\Omega)\|^2_{L^2|\bfB|_\bfH}
&=
\|\,|\theta|_\bfH(x)\Omega\,\|^2_{L^2|\bfB|_\bfH}
=
|\tau|_\bfH^\circ(|\theta|_\bfH(x^*)\theta|_\bfH(x))
\\&\leq
|\tau|_\bfH^\circ(|\theta|_\bfH(x^*x))
=
|\tau|_\bfH^\circ(x^*x)
=
\|x\Omega\|^2_{L^2|\bfA|_\bfH}.
\qedhere
\end{align*}
\end{proof}

Using Proposition \ref{prop:ExtendUCPMaps} and Corollary \ref{cor:UCPGivesContraction}, we prove the following result.

\begin{prop}
\label{prop:ExtendUCP}
Let $\theta:\bfA\Rightarrow \bfB$ be a categorical ucp morphism.
Then $|\theta|_\bfH:|\bfA|_\bfH^\circ\rightarrow |\bfB|_\bfH^{\circ}$ extends uniquely to a normal ucp map
$|\theta|_{\bfH}:|\bfA|_{\bfH}\rightarrow |\bfB|_{\bfH}$ which is $N-N$ bilinear.
\end{prop}
\begin{proof}
Recall that $|\bfB|^{\prime}_{\bfH}\Omega=J|\bfB|_\bfH\Omega$ is dense in $L^{2}|\bfB|_{\bfH}$ by Tomita-Takesaki theory.
Now for any $x\in |\bfA|_{\bfH}^\circ$ and $\xi\in |\bfB|^{\prime}_{\bfF}$, by Corollary \ref{cor:UCPGivesContraction}, we have
\begin{align*}
\omega_{\xi\Omega}(|\theta|_\bfH(x))
&=
\langle |\theta|_\bfH(x) \xi \Omega , \eta \Omega\rangle_{L^{2}|\bfB|_{\bfF}}
=
\langle \xi |\theta|_\bfH(x) \Omega , \eta \Omega\rangle_{L^{2}|\bfB|_{\bfF}}
\\&=
\langle L^2|\theta|_\bfH(x\Omega) , \xi^{*}\eta \Omega\rangle_{L^{2}|\bfB|_{\bfF}}
=
\langle x\Omega , L^2|\theta|_\bfH^{*}(\xi^{*}\eta \Omega)\rangle_{L^{2}|\bfA|_{\bfF}}.
\end{align*}
We may thus extend $|\theta|_\bfH$ to $|\bfA|_\bfH$ by Proposition \ref{prop:ExtendUCPMaps}, where 
$A=|\bfA|_\bfH^\circ$,
$(\pi_1, H_1)$ is the standard representation on $L^2|\bfA|_\bfH$,
$(\pi_2, H_2)$ is the representation $(|\pi|_\bfH, |\bfH\otimes \bfK|)$,
$K = L^2|\bfB|_\bfH$,
$T : B(H_2) \to B(K)$ is $\Ad(|v|_\bfH)$, 
and $D = |\bfM|_\bfH'\Omega \subset L^2|\bfB|_\bfH$.
Finally, $|\theta|_\bfH$ is $N-N$ bilinear by (1) of Remarks \ref{rems:ExtendUCPMaps}.
\end{proof}

\begin{rem}
By combining (1) and (2) of Remarks \ref{rems:ExtendUCPMaps}, we see that if $\theta: \bfA \Rightarrow \bfB$ is a unital $*$-algebra morphism, then $|\theta|_\bfH$ is an $N-N$ bilinear unital $*$-algebra homomorphism.
\end{rem}

\begin{thm} 
We have $|\cdot|_\bfH:\ConAlg\rightarrow \DisInc_\bfH$ is a functor.
\end{thm}
\begin{proof}
Given $\bfA \in \ConAlg$, the von Neumann algebra $|\bfA|_\bfH \in \DisInc_\bfH$ by Theorem \ref{thm:Quasiregular}.
Given a categorical ucp morphism $\theta \in \ConAlg(\bfA,\bfB)$, $|\theta|_\bfH \in \DisInc_\bfH(|\bfA|_\bfH, |\bfB|_\bfH)$ by Lemma \ref{lem:StatePreserving} and Proposition \ref{prop:ExtendUCP}.

It remains to show that $|\id_\bfA|_\bfH =\id_{|\bfA|_\bfH}$ and that $|\cdot|_\bfH$ preserves composition.
The first statement is obvious.
For the second statement, note that if $\theta_1\in \ConAlg(\bfA, \bfB)$ and $\theta_2\in \ConAlg(\bfB, \bfC)$, then $|\theta_2\circ \theta_1|_{\bfH}^\circ=|\theta_2|_\bfH^\circ \circ |\theta_1|_\bfH^\circ$ on the algebraic realization $|\bfA|_\bfF^\circ$, which is strongly* dense in the von Neumann realization.
We are then finished by the uniqueness statement from Proposition \ref{prop:ExtendUCP}.
\end{proof}

\subsection{From discrete subfactors to connected algebras}

We now construct a functor in the other direction.
We continue the use of Notation \ref{nota:RepresentationAndConnectedAlgebra}.
For an irreducible discrete inclusion $(N, \subseteq M, E)$, we let $\phi = \tau \circ E$ be the canonical state on $M$ as in Notation \ref{nota:Discrete}.

\begin{defn}
\label{defn:UnderlyingAlgebraObject}
Suppose $(N\subseteq M, E)\in \DisInc_\bfH$. 
The \emph{underlying algebra object} of $(N\subseteq M, E)$ is the object $\alg{M}\in  \Vec(\cC)$ given by $\alg{M}(c) := \Hom_{N-N}(\bfH(c), L^2(M,\phi))$.
\end{defn}

\begin{lem}
\label{lem:ConnectedW*Algebra}
The algebra object $\alg{M}\in \Vec(\cC)$ is the connected \emph{W*}-algebra object corresponding to the cyclic left $\cC$-module \emph{W*}-subcategory $\cM$ of $\Bim(N,M)$ (the $N-M$ bimodules) generated by the basepoint $L^2(M,\phi)$.
\end{lem}
\begin{proof}
Notice that since $c\in \cC$ is dualizable, so is $\bfH(c)$, and thus $\alg{M}(c)$ is finite dimensional by Proposition \ref{prop:MultiplicationMapBounded}.
To define the $*$-structure and multiplication, we note that by Theorem \ref{thm:BifiniteFrobeniusReciprocity}, we have a natural isomorphism
\begin{equation}
\label{eq:EssSurjFrobRecip}
\alg{M}(c)
=
\Hom_{N-N}(\bfH(c), L^2(M,\phi))
\cong
\Hom_{N-M}(\bfH(c)\boxtimes_N L^2(M,\phi), L^2(M,\phi)).
\end{equation}
This means that $\alg{M}\in \Vec(\cC)$ is exactly the W*-algebra object corresponding to the cyclic left $\cC$-module W*-subcategory $\cM$ of $\Bim(N,M)$ generated by the basepoint $L^2(M,\phi)$.
It is easy to see the algebra $\alg{M}$ is connected by applying \eqref{eq:EssSurjFrobRecip} above with $\bfH(1_\cC)=L^2(N)$.
Indeed, since $N\subseteq M$ is irreducible,
\begin{equation*}
\alg{M}(1_\cC)
=
\Hom_{N-N}(L^2(N), L^2(M,\phi))
\cong
\Hom_{N-M}(L^2(M,\phi), L^2(M,\phi))
\cong 
N'\cap M
=
\bbC.
\qedhere
\end{equation*}
\end{proof}

We may now use the proof of \cite[Thm.~3.20]{MR3687214} to determine the multiplication, unit, and $*$-structure for $\alg{M}$. 

\begin{lem}
\label{lem:InverseFunctorMultiplication}
The multiplication map
$\mu^{\alg{M}}_{a,b} : \alg{M}(a)\otimes \alg{M}(b) \to \alg{M}(a\otimes b)$ for $a,b\in\cC$
is given by 
$$
\alg{M}(a)\otimes \alg{M}(b)
\ni
f\otimes g 
\mapsto 
\mu_M\circ (f\boxtimes g) \circ (\mu_{a,b}^{\bfH})^{-1} 
\in 
\alg{M}(a\otimes b)=
\Hom_{N-N}(\bfH(a\otimes b), L^2(M,\phi))
$$
where $\mu_M$ is the multiplication map on $M$.
\end{lem}
\begin{proof}
The expression is well-defined by Proposition \ref{prop:MultiplicationMapBounded}, as $\bfH(a),\bfH(b)$ are bifinite.
Notice that for $f\in \alg{M}(c)$, the corresponding map on the right hand side of \eqref{eq:EssSurjFrobRecip} is given on $\bfH(c)^\circ \boxtimes_N L^2(M,\phi)$ by $\eta \boxtimes \xi \mapsto f(\eta)\xi$ by Theorem \ref{thm:BifiniteFrobeniusReciprocity}.
Now $f(\eta)$ is an $N$-bounded vector in a bifinite $N-N$ sub-bimodule of $L^2(M,\phi)$, and thus $f(\eta) \in M\Omega$ by Proposition \ref{prop:MultiplicationMapBounded}.
Thus for $f\otimes g\in \alg{M}(a)\otimes \alg{M}(b)$, the composite of the corresponding morphisms on the right hand side of \eqref{eq:EssSurjFrobRecip} in the module category $\cM$ is given on $\bfH(a\otimes b)^\circ \boxtimes_N L^2(M,\phi)$ by 
$$
\mu^{\bfH}_{a,b}(\eta\boxtimes \zeta)\boxtimes \xi 
\xrightarrow{(\mu^\bfH_{a,b})^{-1}} 
\eta \boxtimes \zeta \boxtimes \xi 
\xrightarrow{}
f(\eta)g(\zeta)\xi.
$$
Thus the map $\mu^{\alg{M}}_{a,b}(f\otimes g)\in \alg{M}(a\otimes b)$ is given by $\mu^{\bfH}_{a,b}(\eta \boxtimes \zeta) \mapsto f(\eta)g(\zeta)$ as claimed.
\end{proof}

In order to define the $*$-structure, for $f\in \alg{M}(c)=\Hom_{N-N}(\bfH(c), L^2(M,\phi))$, we define $S_\phi f \in \alg{M}(\overline{c}) \cong\Hom_{N-N}(J_\phi \bfH(c), L^2(M,\phi))$ as the extension of the map $J_\phi \zeta \mapsto f(\zeta)^*$, where we suppress the isomorphism $\bfH(\overline{c})\cong J_\phi \bfH(c)$.

\begin{lem}
\label{lem:InverseFunctorStar}
The $*$-structure $j$ for $\alg{M}$ is given as follows.
For $f\in \alg{M}(c)$, $j_c(f)\in \alg{M}(\overline{c})$ is the map $S_\phi f$.
\end{lem}
\begin{proof}
As in the proof of \cite[Thm.~3.20]{MR3687214}, we take the adjoint morphism in the module category $\cM$ and apply the evaluation map $\bfH(\ev_c) = \ev_{\bfH(c)}$ as in \eqref{eq:NonStandard}.
For $\eta, \xi \in L^2(M,\phi)$, $\zeta\in \bfH(c)^\circ$, and $\{\alpha\}\subset \bfH(c)^\circ$ an ${}_N \bfH(c)$-basis, we calculate the inner product
$$
\langle 
[
(\ev_{\bfH(c)}\boxtimes_N \id_{L^2(M,\phi)})
\circ
(\id_{\bfH(\overline{c})}\boxtimes_{N} f^{*}\boxtimes_{N} \id_{L^2(M,\phi)})
\circ 
(\id_{\bfH(\overline{c})})\boxtimes_{N} \mu^{*}_{\bfH(\overline{c})})
]
(J_{\phi}\zeta\boxtimes_{N}\eta)
, 
\xi\rangle_{L^{2}(M,\phi)} 
$$
which is again well-defined as $\bfH(c)$ is bifinite.
Taking adjoints, and using that $\ev_{\bfH(c)}^*:L^2(N) \to J_\phi\bfH(c)\boxtimes_N \bfH(c)$ is given by $n\Omega \mapsto n\sum_\alpha J_\phi\alpha \boxtimes \alpha$ by \eqref{eq:EvAndCoevBar},  
the above inner product is equal to
\begin{align*}
\sum_i
\langle J_{\phi}\zeta\boxtimes_{N}\eta,
J_{\phi}\alpha_i\boxtimes_{N}f(\alpha_i)\xi
\rangle
&=
\sum_i
\langle \langle J_{\phi}\alpha_i|J_{\phi}\zeta\rangle_{N}\eta,
f(\alpha_i)\xi
\rangle
=
\sum_i
\langle \eta,
\langle J_{\phi}\alpha_i|J_{\phi}\zeta\rangle_{N}^*f(\alpha_i)\xi
\rangle
\\&=
\sum_i
\langle \eta,
\langle J_{\phi}\zeta|J_{\phi}\alpha_i\rangle_{N}f(\alpha_i)\xi
\rangle
=
\sum_i
\langle \eta,
f(\langle J_{\phi}\zeta|J_{\phi}\alpha_i\rangle_{N}\alpha_i)\xi
\rangle
\\&=
\sum_i
\langle \eta,
f({}_N\langle \zeta,\alpha_i\rangle\alpha_i)\xi
\rangle
=
\langle \eta,
f(\zeta)\xi
\rangle
\\&=
\langle f(\zeta)^*\eta,
\xi
\rangle.
\end{align*}
Thus under the isomorphism \eqref{eq:EssSurjFrobRecip}, the $*$-structure is as claimed.
\end{proof}

We now extend $\alg{\cdot}$ to a functor.
We begin with the following lemma.

\begin{lem}
\label{lem:L2UCP}
Consider two irreducible discrete inclusions $(N\subseteq M, E_N^M), (N\subseteq P, E_N^P)$.
Suppose $\psi : M \to P$ is an $N-N$ bilinear ucp map satisfying $\phi_P\circ \psi = \phi_M$.
Then $\psi$ extends to a contractive $N-N$ bilinear map $L^2\psi : L^2(M,\phi_M) \to L^2(P, \phi_P)$.
\end{lem}
\begin{proof}
The proof is an easy application of Kadison's inequality, which follows from Stinespring dilation \cite[Cor.~6.1.2]{PetersonNotes}.
\end{proof}

\begin{defn}
Consider two irreducible discrete inclusions $(N\subseteq M, E_N^M), (N\subseteq P, E_N^P)$.
Suppose $\psi : M \to P$ is an $N-N$ bilinear ucp map satisfying $\phi_P\circ \psi = \phi_M$.
We define $\alg{\psi} : \alg{M} \Rightarrow \alg{P}$ by $\alg{\psi}_c : \alg{M}(c) \to \alg{P}(c)$ is given by
$$
\Hom_{N-N}(\bfH(c), L^2(M,\phi_M)) 
\ni
f
\mapsto
L^2\psi \circ f
\in
\Hom_{N-N}(\bfH(c), L^2(P,\phi_P)).
$$
It is straightforward to calculate that $L^2\id_{(N\subseteq M, E)} = \id_{\alg{M}}$ and if we have morphisms $\psi_1: (N\subseteq M_1 , E^{M_1}_N) \to (N\subseteq M_2 , E^{M_2}_N))$ and $\psi_2: (N\subseteq M_2 , E^{M_2}_N) \to (N\subseteq M_3 , E^{M_3}_N))$, then $L^2(\psi_2\circ \psi_1) = L^2\psi_2 \circ L^2\psi_1$.
This means that $\alg{\cdot}$ is a functor, provided we show $\alg{\psi}$ is a categorical ucp morphism.
\end{defn}

The following lemma is similar to \cite[Lem.~3.7]{MR3406647}.

\begin{lem}
\label{lem:SufficientForUCP}
Suppose $\bfA, \bfB \in \Vec(\cC)$ are two \emph{W*}-algebra objects, and let $\cM_\bfA,\cM_\bfB$ be the corresponding cyclic $\cC$-module \emph{W*}-categories as in \cite[Thm.~3.24]{MR3687214}.
Suppose $\theta: \cM_\bfA \to \cM_\bfB$ is a multiplier.
Then $\theta$ is cp if and only if for all $c\in \cC$ and $f\in \bfA(c)$, the following morphism in $\cM_\bfB(c_\bfB, c_\bfB)$ is positive:
$$
\theta\left(
\begin{tikzpicture}[baseline=-.1cm]
    \draw (.5,.5) arc (90:-90:.5cm);
    \draw (1,0) -- (1.5,0);
    \draw (-.5,.5) arc (-90:-180:.5cm);
    \draw (-.5,-.5) arc (90:180:.5cm);
    \filldraw (1,0) circle (.05cm);
    \roundNbox{unshaded}{(0,.5)}{.3}{.2}{.2}{$j_c(f)$}
    \roundNbox{unshaded}{(0,-.5)}{.3}{.2}{.2}{$f$}
    \node at (1.3,.2) {\scriptsize{$\bfA$}};
    \node at (.7,-.7) {\scriptsize{$\bfA$}};
    \node at (.7,.7) {\scriptsize{$\bfA$}};
    \node at (-1.2,-.9) {\scriptsize{$\mathbf{c}$}};
    \node at (-1.2,.9) {\scriptsize{$\mathbf{c}$}};
\end{tikzpicture}
\right).
$$
\end{lem}
\begin{proof}
The forward direction is trivial, so we need only prove the reverse direction.
Suppose we have a multiplier $\theta: \cM_\bfA \to \cM_\bfB$.
We must show that for all $d\in \cC$ and $f\in \cM_\bfA(d_\bfA, d_\bfA)$, $\theta(f^*\circ f)\geq 0$.
Setting $c= \overline{d}\otimes d$, we may write
$$
\theta(f^*\circ f)
=
\theta\left(
\begin{tikzpicture}[baseline=-.1cm]
    \draw (.5,.5) arc (90:-90:.5cm);
    \draw (1,0) -- (1.5,0);
    \draw (-.5,.65) arc (-90:-180:.35cm);
    \draw (-.5,-.65) arc (90:180:.35cm);
    \draw (-.5,-.35) arc (270:90:.35cm);
    \filldraw (1,0) circle (.05cm);
    \roundNbox{unshaded}{(0,.5)}{.3}{.2}{.2}{$j_c(f)$}
    \roundNbox{unshaded}{(0,-.5)}{.3}{.2}{.2}{$f$}
    \node at (1.3,.2) {\scriptsize{$\bfA$}};
    \node at (.7,-.7) {\scriptsize{$\bfA$}};
    \node at (.7,.7) {\scriptsize{$\bfA$}};
    \node at (-1,-.9) {\scriptsize{$\mathbf{c}$}};
    \node at (-1,.9) {\scriptsize{$\mathbf{c}$}};
    \node at (-1,0) {\scriptsize{$\mathbf{c}$}};
\end{tikzpicture}
\right)
=
\theta\left(
\begin{tikzpicture}[baseline=-.1cm]
    \draw (.5,.5) arc (90:-90:.5cm);
    \draw (1,0) -- (1.5,0);
    \draw (-.5,.35) arc (-90:-180:.65cm);
    \draw (-.5,.65) arc (-90:-180:.35cm) -- (-.85, 1.2);
    \draw (-.5,-.65) arc (90:180:.35cm) -- (-.85,-1.2);
    \draw (-.5,-.35) arc (90:180:.65cm);
    \draw (-1.15,1) arc (0:180:.15cm) -- (-1.45, -1) arc (-180:0:.15cm);
    \filldraw (1,0) circle (.05cm);
    \roundNbox{unshaded}{(0,.5)}{.3}{.2}{.2}{$j_c(f)$}
    \roundNbox{unshaded}{(0,-.5)}{.3}{.2}{.2}{$f$}
    \node at (1.3,.2) {\scriptsize{$\bfA$}};
    \node at (.7,-.7) {\scriptsize{$\bfA$}};
    \node at (.7,.7) {\scriptsize{$\bfA$}};
    \node at (-.7,-1.1) {\scriptsize{$\mathbf{c}$}};
    \node at (-.7,1.1) {\scriptsize{$\mathbf{c}$}};
    \node at (-1.2,-.5) {\scriptsize{$\overline{\mathbf{c}}$}};
    \node at (-1.2,.5) {\scriptsize{$\overline{\mathbf{c}}$}};
    \node at (-1.6,0) {\scriptsize{$\mathbf{c}$}};
\end{tikzpicture}
\right)
=
(E_{\overline{c}} \circ \theta)
\left(
\begin{tikzpicture}[baseline=-.1cm]
    \draw (.5,.5) arc (90:-90:.5cm);
    \draw (1,0) -- (1.5,0);
    \draw (-.5,.5) arc (-90:-180:.5cm);
    \draw (-.5,-.5) arc (90:180:.5cm);
    \filldraw (1,0) circle (.05cm);
    \roundNbox{unshaded}{(0,.5)}{.3}{.2}{.2}{$j_c(f)$}
    \roundNbox{unshaded}{(0,-.5)}{.3}{.2}{.2}{$f$}
    \node at (1.3,.2) {\scriptsize{$\bfA$}};
    \node at (.7,-.7) {\scriptsize{$\bfA$}};
    \node at (.7,.7) {\scriptsize{$\bfA$}};
    \node at (-1.4,-.9) {\scriptsize{$\overline{\mathbf{c}}\otimes \mathbf{c}$}};
    \node at (-1.4,.9) {\scriptsize{$\overline{\mathbf{c}}\otimes \mathbf{c}$}};
\end{tikzpicture}
\right)
$$
where $E_{\overline{c}} : \cM_\bfB((\overline{c}\otimes c)_\bfB,(\overline{c}\otimes c)_\bfB) \to \cM_\bfB(c_\bfB,c_\bfB)$ is the partial trace which caps the $\overline{\mathbf{c}}$ string on the left.
Finally, we see $\theta(f^*\circ f)\geq 0$ by first applying our positivity hypothesis and then applying the completely positive map $E_{\overline{c}}$.
\end{proof}

\begin{prop}
Suppose $\psi : M \to P$ is an $N-N$ bilinear ucp map such that $\phi_P \circ \psi = \phi_M$.
Then $\alg{\psi}: \alg{M} \Rightarrow \alg{P}$ is a categorical ucp morphism, and thus $\alg{\cdot}$ is a functor.
\end{prop}
\begin{proof}
We use Lemma \ref{lem:SufficientForUCP} for $\alg{\psi}$.
Let $c\in \cC$ and $f\in \alg{M}(c)\cong \Hom_{N-N}(\bfH(c), L^2(M,\phi_M))$.
We must show that
$L^2\psi \circ \mu^{\alg{M}}_{\overline{c}, c} (j_c^{\alg{M}}(f)\otimes f)$ is positive in $\bfB(\overline{c}\otimes c)$.
Now we use the natural isomorphism $\alg{P}(\overline{c}\otimes c) \cong \End_{N-P}(\bfH(c) \boxtimes_N L^2(P, \phi_P))$ from Theorem \ref{thm:BifiniteFrobeniusReciprocity}, where our map in question corresponds to the map
\begin{equation}
\label{eq:PositiveMap}
\begin{split}
(\id_{\bfH(c)}\boxtimes \mu_P)
&\circ
(\id_{\bfH(c)}\boxtimes L^2\psi \boxtimes\id_{L^2(P, \phi_P)})
\circ
(\id_{\bfH(c)}\boxtimes \mu_{c,c}^{\alg{M}} \boxtimes\id_{L^2(P, \phi_P)})
\\&\circ
(\id_{\bfH(c)} \boxtimes j_c(f) \boxtimes f \boxtimes \id_{L^2(P, \phi_P)} )
\circ
(\coev_{\bfH(c)} \boxtimes \id_{\bfH(c)}\boxtimes \id_{L^2(P, \phi_P)} ),
\end{split}
\end{equation}
which is well-defined by Proposition \ref{prop:MultiplicationMapBounded} since $(L^2\psi \circ \mu^{\alg{M}}_{\overline{c}, c})(\bfH(\overline{c})\boxtimes_N \bfH(c))$ is bifinite.
Denote the map \eqref{eq:PositiveMap} by $T\in \End_{N-P}(\bfH(c) \boxtimes_N L^2(P, \phi_P))$, and consider a finite sum $\sum_i \eta_i \boxtimes \xi_i \in  \bfH(c)^\circ\boxtimes_N L^2(P,\phi_P)$.
Picking an $\bfH(c)_N$ basis $\{\beta\} \in \bfH(c)^\circ $, we calculate using Lemmas \ref{lem:InverseFunctorMultiplication}, \ref{lem:InverseFunctorStar}, \ref{lem:L2UCP} and the fact that 
$\coev_{\bfH(c)}:L^2(N) \to \bfH(c)\boxtimes_N J_\phi\bfH(c)$ is given by $n\Omega \mapsto n\sum_\beta  \beta\boxtimes J_\phi\beta$ 
by \eqref{eq:EvAndCoev} that
\begin{align*}
\left\langle T \left(\sum_i\eta_i \boxtimes \xi_i\right)\right.,&\left. \sum_j\eta_j \boxtimes \xi _j\right\rangle_{\bfH(c) \boxtimes_N L^2(P, \phi_P)}
\\&=
\sum_{i,j}\sum_\beta
\langle \beta \boxtimes \psi(f(\beta)^* f(\eta_i)) \xi_i , \eta_j \boxtimes \xi_j\rangle_{\bfH(c) \boxtimes_N L^2(P, \phi_P)}
\\&=
\sum_{i,j}\sum_\beta
\langle \langle \eta_j|\beta\rangle_N \psi(f(\beta)^* f(\eta_i)) \xi_i , \xi_j\rangle_{L^2(P, \phi_P)}
\\&=
\sum_{i,j}\sum_\beta
\langle \psi(\langle \eta_j|\beta\rangle_N f(\beta)^* f(\eta_i)) \xi_i , \xi_j\rangle_{L^2(P, \phi_P)}
\\&=
\sum_{i,j}\sum_\beta
\langle \psi( f(\beta \langle \beta|\eta_j\rangle_N)^* f(\eta_i)) \xi_i , \xi_j\rangle_{L^2(P, \phi_P)}
\\&=
\sum_{i,j}
\langle \psi( f(\eta_j)^* f(\eta_i)) \xi_i , \xi_j\rangle_{L^2(P, \phi_P)}.
\end{align*}
The final expression is easily seen to be positive since $\psi$ is completely positive.
It follows that $T\geq 0$ since $\bfH(c)^\circ \boxtimes_N L^2(P,\phi_P)$ is dense in $\bfH(c)\boxtimes_N L^2(P, \phi)$.
\end{proof}

\subsection{Equivalence of categories}
\label{sec:EquivalenceOfCategories}

Finally, we prove our main theorem.

\begin{thm}
\label{thm:Equivalence}
The realization functor $|\cdot|_\bfH:  \ConAlg \to \DisInc_\bfH$ and the underlying algebra functor $\alg{\cdot}: \DisInc_\bfH \to \ConAlg$ witness an equivalence of categories $\ConAlg \cong \DisInc_\bfH$.
\end{thm}

First, we will construct a natural isomorphism $\delta:|\alg{\cdot}|_\bfH \Rightarrow \id$.
Let $(N\subseteq M, E)\in \DisInc_\bfH$, and let $\bfA = \alg{M} \in \ConAlg$.
First, we apply Lemmas \ref{lem:InverseFunctorMultiplication} and \ref{lem:InverseFunctorStar} to recall the multiplication and $*$-structure of
$$
|\bfA|_\bfH^\circ
= 
\bigoplus_{c\in \Irr(\cC)} \bfH(c)^\circ \otimes_\bbC \bfA(c)
=
\bigoplus_{c\in \Irr(\cC)} \bfH(c)^\circ \otimes_\bbC \Hom_{N-N}(\bfH(c), L^2(M,\phi)).
$$ 
For $\xi_1\otimes f_1\in \bfH(a)^\circ\otimes\bfA(a)$ and $\xi_2\otimes f_2\in \bfH(b)^\circ\otimes\bfA(b)$, the product is given by 
\begin{align*}
(\xi_1\otimes f_1) \cdot (\xi_2\otimes f_2) 
&=
\sum_{\alpha \in \Isom(c, a\otimes b)}
\bfH^\circ(\alpha^*)[\mu_{a,b}^\bfH(\xi_1\boxtimes \xi_2)] 
\otimes 
\bfA(\alpha)[\mu^\bfA_{a,b}(f_1\otimes g_2)]
\\&=
\sum_{\alpha \in \Isom(c, a\otimes b)}
\bfH^\circ(\alpha^*)\circ[\mu_{a,b}^\bfH(\xi_1\boxtimes \xi_2)] 
\otimes 
[\mu_M\circ(f_1\boxtimes f_2)\circ (\mu^\bfH_{a,b})^{-1}]\circ \bfH(\alpha).
\end{align*}
For $(\xi\otimes f)\in \bfH(c)^\circ \otimes \bfA(c)$, the adjoint is given by
$$
(\xi\otimes f)^*
=
j_c(\xi) \otimes j_c(f)
=
J_\phi\xi \otimes S_\phi f,
$$
where $S_\phi f \in \bfA(\overline{c}) =\Hom_{N-N}( J_\phi \bfH(c) , L^2(M,\phi))$ is the map given by $J_\phi \zeta\mapsto S_\phi f(\zeta) = f(\zeta)^*$.

\begin{defn}
We define an $N-N$ bilinear map $\delta=\delta_{M}:|\bfA|_\bfH^\circ \to L^2(M,\phi)$ by 
$$
\bfH(c)^\circ \otimes \bfA(c)\ni\xi\otimes f 
\mapsto 
f(\xi)\in L^2(M,\phi),
$$
i.e., $\delta$ is the evaluation homomorphism \eqref{eq:EvaluationHomomorphism}.
Since $\xi\in \bfH(c)^\circ$, $f(\xi)\in M\Omega$, and thus $\delta$ is actually a map $|\bfA|_\bfH^\circ \to M$.
\end{defn}

\begin{lem}
\label{lem:EvaluationHomomorphismAlgebraMap}
The map $\delta:|\bfA|_\bfH^\circ \to M$ is a unital $*$-homomorphism.
\end{lem}
\begin{proof}
For all $\xi_1\otimes f_1\in \bfH(a)^\circ\otimes\bfA(a)$ and $\xi_2\otimes f_2\in \bfH(b)^\circ\otimes\bfA(b)$, we have
\begin{align*}
\delta[(\xi_1\otimes f_1)&\cdot(\xi_2\otimes f_2)]
\\&=
\delta\left(
\sum_{\alpha \in \Isom(c, a\otimes b)}
\bfH^\circ(\alpha^*)\circ[\mu_{a,b}^\bfH(\xi_1\boxtimes \xi_2)] 
\otimes 
[\mu_M\circ(f_1\boxtimes f_2)\circ (\mu^\bfH_{a,b})^{-1}]\circ \bfH(\alpha)
\right)
\\&=
\sum_{\alpha \in \Isom(c, a\otimes b)}
[\mu_M\circ(f_1\boxtimes f_2)\circ (\mu^\bfH_{a,b})^{-1}]\circ \bfH(\alpha)
\circ
\bfH^\circ(\alpha^*)\circ[\mu_{a,b}^\bfH(\xi_1\boxtimes \xi_2)] 
\\&=
\mu_M\circ(f_1\boxtimes f_2)\circ (\mu^\bfH_{a,b})^{-1}
\circ
\mu_{a,b}^\bfH(\xi_1\boxtimes \xi_2)
\\&=
f_1(\xi_1)f_2(\xi_2).
\end{align*}
For $\xi\otimes f \in \bfH(c)^\circ \otimes \bfA(c)$, we have
$$
\delta[(\xi\otimes f)^*]
=
\delta(J_\phi \xi \otimes S_\phi f)
=
S_\phi f(J_\phi \xi)
=
f(\xi)^*.
$$
Finally, $\delta(\Omega_{L^2(N)} \otimes i_\bfA) = \Omega_{L^2(M,\phi)}$.
\end{proof}

\begin{lem}
\label{lem:*-isoPreservesStates}
The map $\delta: |\bfA|_\bfH^\circ \to M$ preserves the canonical state, i.e., $\phi\circ \delta = |\tau|_\bfH^\circ$.
\end{lem}
\begin{proof}
For all $x=\sum_{c\in\Irr(\cC)} \xi_{(1)}^c \otimes f_{(2)}^c \in |\bfA|_\bfH^\circ$, we have
\begin{align*}
(\phi\circ\delta)\left(
x
\right)
&=
\phi\left(
\sum_{c\in\Irr(\cC)} f_{(2)}^c(\xi_{(1)}^c)
\right)
=\phi\left(
i_{L^2(N)}\xi^{1_\cC}
\right)
=
\tau(\xi^{1_\cC})
=
|\tau|_\bfH^\circ\left(
x
\right).
\end{align*}
Above, $\xi^{1_\cC}\in N\Omega$ and $i_{L^2(N)}\in \bfA(1_\cC)=\Hom_{N-N}(L^2(N), L^2(M,\phi))$ is the canonical inclusion induced by $\Omega_{L^2(N)}\mapsto \Omega_{L^2(M,\phi)}$.
\end{proof}

\begin{prop}
The map $\delta$ gives a natural isomorphism $|\alg{\cdot}|_\bfH \Rightarrow \id$.
\end{prop}
\begin{proof}
By Lemma \ref{lem:*-isoPreservesStates}, the map $\delta: |\bfA|_\bfH^\circ \to M$ extends to an $N-N$ bilinear unitary isomorphism $L^2\delta:L^2|\bfA|_\bfH \to L^2(M,\phi)$, and $\delta$ is implemented by $L^2\delta$.
Thus $\delta$ extends uniquely to an injective unital $*$-homomorphism $|\bfA|_\bfH \to M$.

To show $\delta$ is onto, we use (1) of Proposition \ref{prop:DenseSubalgebraOfDiscrete}.
Let $Q^\circ =\cQ\cN_M(N)$ be the quasi-normalizer.
It is easy to see that $Q^\circ = \delta(|\bfA|_\bfH^\circ)$.
But $(N\subseteq M, E)$ is quasi-regular by Proposition \ref{prop:QuasiregularIffDiscrete}, so $Q = (Q^\circ)''=M$.
By normality of $\delta$, we have $M = \delta(|\bfA|_\bfH)$.

Finally, to prove naturality, suppose $\psi \in \DisInc_\bfH( (N\subseteq M, E^M_N), (N\subseteq P, E^P_N))$.
For all $\xi\otimes f \in \bfH^\circ(c) \otimes \alg{M}(c)$, applying Definition \ref{defn:RealizationOfUCPMap}, we have $|\alg{\psi}|_\bfH(\xi\otimes f) = \xi\otimes \alg{\psi}(f) = \xi \otimes (\psi\circ f)$.
Thus
$$
\xymatrix{
\xi \otimes f  \ar@{|->}[r]^{\delta_M} \ar@{|->}[d]^{|\alg{\psi}|_\bfH} & f(\xi)\ar@{|->}[d]^{\psi}
\\
\xi \otimes(\psi \circ f)  \ar@{|->}[r]^{\delta_P} & \psi(f(\xi))
}
$$
on $|\bfA|_\bfH^\circ$.
Since the desired maps commute on a strongly*-dense unital $*$-subalgebra, we are finished.
\end{proof}

Finally, we define a natural isomorphism $\kappa:\id \Rightarrow \alg{|\cdot|_\bfH}$.

\begin{defn}
For $\bfA\in \ConAlg$, define $\kappa^\bfA :  \bfA \Rightarrow\alg{|\bfA|_\bfH}$ to be the following natural isomorphism.
For $c\in \cC$, we define $\kappa^\bfA_c : \bfA(c)\to \alg{|\bfA|_\bfH}$ as the composite of the following natural isomorphisms:
\begin{align*}
\bfA(c)
&\cong
\bigoplus_{a\in\Irr(\cC)}
\cC(c, a)\otimes_a \bfA(a)
&&\text{(\cite[Def.~2.31]{MR3687214})}
\\&\cong
\bigoplus_{a\in\Irr(\cC)}
\Hom_{N-N}(\bfH(c),  \bfH(a) ) \otimes_a \bfA(a)
&&\text{($\bfH$ fully faithful)}
\\&\cong
\Hom_{N-N}(\bfH(c), \bigoplus_{a\in\Irr(\cC)} \bfH(a) \otimes_a \bfA(a))
&&\text{(Prop.~\ref{prop:L2MH decomposition})}
\\&\cong
\Hom_{N-N}(\bfH(c), L^2|\bfA|_\bfH)
\\&=:
\alg{|\bfA|_\bfH}(c).
&&
\text{(Def.~\ref{defn:UnderlyingAlgebraObject})}
\end{align*}
\end{defn}

In order to show $\kappa^\bfA$ is a $*$-algebra natural isomorphism, we provide another description of the above natural isomorphism.
Using the graphical calculus  \eqref{eq:PermeableMembrane} for $|\bfA|_\bfH^\circ$, we see that $\kappa^\bfA_c : \bfA(c) \to \alg{|\bfA|_\bfH}$ is given by
\begin{equation}
\label{eq:KappaInDiagrams}
\bfA(c)
\ni
f
\mapsto
\left(\bfH(c)^\circ \ni\eta 
\mapsto 
\begin{tikzpicture}[baseline=-.1cm]
    \draw[red] (-.5,0) -- (.5,0);
    \draw (0,-1.2) -- (0,1.2);
    \roundNbox{unshaded}{(0,.6)}{.3}{0}{0}{$f$}
    \roundNbox{unshaded}{(0,-.6)}{.3}{0}{0}{$\eta$}
    \node at (.2,1.05) {\scriptsize{$\bfA$}};
    \node at (.15,.15) {\scriptsize{$\mathbf{c}$}};
    \node at (.15,-.15) {\scriptsize{$\mathbf{c}$}};
    \node at (.3,-1.05) {\scriptsize{$\bfH^\circ$}};
\end{tikzpicture}
\in |\bfA|_\bfH^\circ 
\right).
\end{equation}

\begin{prop}
\label{prop:StarAlgebraNaturalIso}
The natural isomorphism $\kappa^\bfA$ is a $*$-algebra natural isomorphism.
\end{prop}
\begin{proof}
We use the graphical calculus \eqref{eq:GraphicalMultiplicationAndStar} for multiplication and the $*$-structure of $|\bfA|_\bfH^\circ$.
For $f\otimes g \in \bfA(a)\otimes \bfA(b)$, we have that
$\kappa^\bfA_{a\otimes b} (\mu_{a,b}^\bfA (f\otimes g))$ is the map
\begin{equation}
\label{eq:KappaOfMu}
\bfH(a)^\circ \boxtimes_N \bfH(b)^\circ
\ni
\eta\boxtimes \xi 
\mapsto
\begin{tikzpicture}[baseline=-.1cm]
    \draw[red] (-.5,0) -- (1.5,0);
    \draw (0,-.6) -- (0,.6);
    \draw (1,-.6) -- (1,.6);
    \draw (0,.9) arc (180:0:.5cm);
    \draw (0,-.9) arc (-180:0:.5cm);
    \filldraw (.5,1.4) circle (.05cm);
    \filldraw (.5,-1.4) circle (.05cm);
    \draw (.5,1.4) -- (.5,1.8);
    \draw (.5,-1.4) -- (.5,-1.8);
    \roundNbox{unshaded}{(0,.6)}{.3}{0}{0}{$f$}
    \roundNbox{unshaded}{(0,-.6)}{.3}{0}{0}{$\eta$}
    \roundNbox{unshaded}{(1,.6)}{.3}{0}{0}{$g$}
    \roundNbox{unshaded}{(1,-.6)}{.3}{0}{0}{$\xi$}
    \node at (.7,1.6) {\scriptsize{$\bfA$}};
    \node at (-.2,1.1) {\scriptsize{$\bfA$}};
    \node at (1.2,1.1) {\scriptsize{$\bfA$}};
    \node at (-.15,.15) {\scriptsize{$\mathbf{a}$}};
    \node at (-.15,-.15) {\scriptsize{$\mathbf{a}$}};
    \node at (1.15,.15) {\scriptsize{$\mathbf{b}$}};
    \node at (1.15,-.15) {\scriptsize{$\mathbf{b}$}};
    \node at (-.2,-1.1) {\scriptsize{$\bfH^\circ$}};
    \node at (1.2,-1.1) {\scriptsize{$\bfH^\circ$}};
    \node at (.8,-1.6) {\scriptsize{$\bfH^\circ$}};
\end{tikzpicture}
\in |\bfA|_\bfH^\circ.
\end{equation}
Now suppressing the tensorator for $\bfH$ and $\bfH^\circ$, by Lemma \ref{lem:InverseFunctorMultiplication}, we have that
$\mu^{\alg{|\bfA|_\bfH}}_{a,b} (\kappa^\bfA_a(f)\otimes \kappa^\bfA_b(g))$ 
is the map
$$
\bfH(a)^\circ \boxtimes_N \bfH(b)^\circ 
\ni
\eta \boxtimes \xi
\mapsto
\kappa^\bfA_a(f)(\eta)\kappa^\bfA_b(g)(\xi)
\in |\bfA|_\bfH^\circ,
$$
which is exactly equal to the right hand side of \eqref{eq:KappaOfMu} by \eqref{eq:KappaInDiagrams}.

Suppressing the isomorphism $\overline{\bfH(a)}\cong \bfH(\overline{a})$, for $f\in \bfA(a)$, we have that $\kappa_{\overline{a}}^\bfA(j_a^\bfA(f))$ is the map
\begin{equation}
\label{eq:KappaOfJ}
\overline{\bfH(a)}
\ni
\overline{\eta}
\mapsto
\begin{tikzpicture}[baseline=-.1cm]
    \draw[red] (-.5,0) -- (.5,0);
    \draw (0,-1.3) -- (0,1.3);
    \roundNbox{unshaded}{(0,.6)}{.3}{.3}{.3}{$j_a^\bfA(f)$}
    \roundNbox{unshaded}{(0,-.6)}{.3}{0}{0}{$\overline{\eta}$}
    \node at (.2,1.1) {\scriptsize{$\bfA$}};
    \node at (.15,.15) {\scriptsize{$\overline{\mathbf{a}}$}};
    \node at (.15,-.15) {\scriptsize{$\overline{\mathbf{a}}$}};
    \node at (.3,-1.1) {\scriptsize{$\bfH^\circ$}};
\end{tikzpicture}
=
\left(
\begin{tikzpicture}[baseline=-.1cm]
    \draw[red] (-.5,0) -- (.5,0);
    \draw (0,-1.3) -- (0,1.3);
    \roundNbox{unshaded}{(0,.6)}{.3}{0}{0}{$f$}
    \roundNbox{unshaded}{(0,-.6)}{.3}{0}{0}{$\eta$}
    \node at (.2,1.1) {\scriptsize{$\bfA$}};
    \node at (.15,.15) {\scriptsize{$\mathbf{a}$}};
    \node at (.15,-.15) {\scriptsize{$\mathbf{a}$}};
    \node at (.3,-1.1) {\scriptsize{$\bfH^\circ$}};
\end{tikzpicture}
\right)^*
\in |\bfA|_\bfH^\circ.
\end{equation}
Now suppressing the tensorator for $\bfH^\circ$, 
along with the isomorphisms $J_\phi \bfH(a)\cong \bfH(\overline{a})$,
by Lemma \ref{lem:InverseFunctorStar}, $j_a^{\alg{|\bfA|_\bfH}}\kappa_a^\bfA(f)$ is the map
$$
J_\phi \bfH(a)^\circ
\ni
J_\phi\eta
\mapsto
(S_\phi \kappa_a^\bfA(f))(\eta)
=
\kappa_a^\bfA(f)(\eta)^*\in |\bfA|_\bfH^\circ,
$$
which is exactly the right hand side of \eqref{eq:KappaOfJ} by \eqref{eq:KappaInDiagrams}.

Finally, it is straightforward to calculate that $\kappa^\bfA_{1_\cC} (i_\bfA) = i_{\alg{|\bfA|_\bfH}}$.
\end{proof}

\begin{prop}
The map $\kappa$ gives a natural isomorphism $\id \Rightarrow \alg{|\cdot|_\bfH}$.
\end{prop}
\begin{proof}
It remains to show naturality of $\kappa$.
It will then follow from Proposition \ref{prop:StarAlgebraNaturalIso} that $\kappa$ is an natural isomorphism.
Suppose $\theta: \bfA \Rightarrow \bfB$ is a categorical ucp morphism.
We have $\alg{|\theta|_\bfH}(\kappa_c^\bfA(f)) = L^2|\theta|_\bfH\circ \kappa_c^\bfA(f)$ on $\bfH(c)^\circ$, which by \eqref{eq:KappaInDiagrams} is the map
$$
\bfH(c)^\circ
\ni
\eta
\mapsto
L^2|\theta|_\bfH
\left(
\begin{tikzpicture}[baseline=-.1cm]
    \draw[red] (-.5,0) -- (.5,0);
    \draw (0,-1.3) -- (0,1.3);
    \roundNbox{unshaded}{(0,.6)}{.3}{0}{0}{$f$}
    \roundNbox{unshaded}{(0,-.6)}{.3}{0}{0}{$\eta$}
    \node at (.2,1.1) {\scriptsize{$\bfA$}};
    \node at (.15,.15) {\scriptsize{$\mathbf{a}$}};
    \node at (.15,-.15) {\scriptsize{$\mathbf{a}$}};
    \node at (.3,-1.1) {\scriptsize{$\bfH^\circ$}};
\end{tikzpicture}
\right)
=
\begin{tikzpicture}[baseline=.5cm]
    \draw[red] (-.5,0) -- (.5,0);
    \draw (0,-1.3) -- (0,2.3);
    \roundNbox{unshaded}{(0,1.6)}{.3}{0}{0}{$\theta$}
    \roundNbox{unshaded}{(0,.6)}{.3}{0}{0}{$f$}
    \roundNbox{unshaded}{(0,-.6)}{.3}{0}{0}{$\eta$}
    \node at (.2,2.1) {\scriptsize{$\bfB$}};
    \node at (.2,1.1) {\scriptsize{$\bfA$}};
    \node at (.15,.15) {\scriptsize{$\mathbf{a}$}};
    \node at (.15,-.15) {\scriptsize{$\mathbf{a}$}};
    \node at (.3,-1.1) {\scriptsize{$\bfH^\circ$}};
\end{tikzpicture}
\in |\bfB|_\bfH^\circ.
$$
This map is easily seen to be equal to the map $\kappa^\bfB_c(\theta(f))$.
\end{proof}

\section{Examples}

We now apply our main theorem to several classes of examples.
We begin with discrete groups and planar algebras, which give type ${\rm II}_1$ tracial examples.
We then discuss examples coming from discrete quantum groups and Temperley-Lieb-Jones module categories which give type ${\rm III}$ examples.

\subsection{Discrete groups}

A rigid C*-tensor category is called \emph{pointed} if $d_c=1$ for all $c\in\Irr(\cC)$.
It is straightforward to show that such a $\cC$ is equivalent to $\fdHilb(\Gamma, \omega)$, the rigid C*-tensor category of finite dimensional $\Gamma$-graded Hilbert spaces, where $\omega\in Z^3(\Gamma, U(1))$ is a normalized 3-cocycle, which is determined up to a coboundary.
Moreover, $\Vec(\cC) = \Vec(\Gamma, \omega)$, the category of $\Gamma$-graded vector spaces with 3-cocycle $\omega$.

\begin{prop}
Connected \emph{W*}-algebra objects $\bfA\in \Vec(\Gamma,\omega)$ are classified up to $*$-algebra isomorphism by
a subgroup $\Lambda \subseteq \Gamma$ such that $\omega|_\Lambda$ is cohomologically trivial and
a normalized 2-cocycle $[\mu] \in H^2(\Lambda, U(1))$ such that $\mu(e,g) = \mu(g,e)=1$ for all $g\in\Lambda$.
\end{prop}
\begin{proof}
It is well known that algebra objects in $\Vec(\Gamma, \omega)$ are classified up to algebra isomorphism by such subgroups $\Lambda$ and 2-cocycles $\mu\in Z^2(\Lambda, \bbC^\times)$, and the associated algebra is the twisted group algebra of $\Lambda$ by $\mu$.
(This is similar to \cite[Ex.~2.1]{MR1976233}.
The laxitor of the algebra is twisted by a 2-cochain, and associativity implies $\omega|_\Lambda$ is cohomologically trivial.)
Cohomologous 2-cocycles give isomorphic algebras.
Hence, it suffices to prove that $\bfA$ being a W*-algebra object implies $\mu$ takes values in $U(1)$, and that each twisted group algebra has a unique $*$-algebra structure under which it becomes a W*-algebra object.

First, suppose we have such a $*$-structure on the twisted group algebra under which it is a W*-algebra object.
By Proposition \ref{prop:DimensionBound}, $\dim(\bfA(g))=1$ for all $g\in \Gamma$.
Moreover, by the proof of Proposition \ref{prop:DimensionBound}, we see that for all $g\in\Gamma$, $1=d_g^{-1} \leq \Tr(\Delta_g) \leq d_g=1$, and thus $\Delta=\id$ and $\bfA$ is tracial.
It is easy to show that the multiplication $\mu: \bfA(g)\otimes \bfA(h) \to \bfA(gh)$ is isometric and thus unitary.
This implies $\mu$ takes values in $U(1)$.
The $*$-structure is determined by a 1-cochain $j\in C^1(\Lambda, U(1))$ such that for all $f\in \bfA(a)\cong \bbC$ and $g\in\Lambda$, $j_g(f) = j(g)\overline{f}$.
(The 1-cochain $j$ takes values in $U(1)$ since $\bfA$ is tracial and $\Delta=\id$.)
The $*$-algebra axioms tell us that 
$j(e)=1$, 
$j(g)=j(g^{-1})$, and 
\begin{equation}
\label{eq:MuJcompatible}
j(gh)\overline{\mu(g,h)} = j(g)j(h)\mu(h^{-1},g^{-1}) \text{ for all }g,h\in \Lambda
\end{equation}
(compare with \cite[Cor.~5.10]{1507.04794}).
Since $\bfA$ is a W*-algebra object, we must have that the right inner product is positive definite, which implies $j(g) \mu(g^{-1},g)>0$ for all $g\in \Lambda$.
But $j(g)\mu(g^{-1},g) \in U(1)$ and thus equals 1.
This means $j(g) = \overline{\mu(g^{-1},g)}$, so $j$ is completely determined by $\mu$.

Now it is straightforward (although a bit tedious) to show the condition  $j(g) = \overline{\mu(g^{-1},g)}$ automatically implies \eqref{eq:MuJcompatible} by using the fact that $d\mu(h^{-1}g^{-1},g,h)=1 = d\mu(h^{-1},g^{-1},g)$.
Hence if we define  $j(g) = \overline{\mu(g^{-1},g)}$, we get a $*$-algebra structure on the twisted group algebra.
It remains to show that this $*$-structure makes the twisted group algebra into a W*-algebra object.
It suffices to show that for all $c\in\fdHilb(\Gamma)$, the $*$-algebra $\bfA(\overline{c}\otimes c)$ is a (finite dimensional) C*-algebra.
To do so one shows that the right $\bfA(1_\cC)=\bbC$ valued inner product is a faithful state.
We omit this routine computation.
\end{proof}

\subsection{Planar algebras}
\label{sec:PlanarAlgebras}

Popa's celebrated reconstruction theorem \cite{MR1334479} shows that every standard $\lambda$-lattice arrises as the standard invariant of a finite index subfactor.
It was shown in \cite{MR2051399} that such a subfactor $N\subset$ exists where $N\cong M\cong L\bbF_\infty$.
In \cite{MR2732052}, Guionnet-Jones-Shlyakhtenko (GJS) gave a diagrammatic re-proof of Popa's reconstruction theorem using Jones' planar algebras \cite{math.QA/9909027}, which are equivalent to Popa's standard $\lambda$-lattices.
We now show that the factor constructed by GJS provides an example of a tracial connected W*-algebra object in a rigid C*-tensor category.

We refer the reader to \cite{MR3405915,MR3624399} for the definition of an unshaded, unoriented $*$-planar algebra.
For convenience and simplicity, we will work with unshaded, unoriented planar algebras.
We will always draw elements of $\cP_n$ as $n$-boxes whose distinguished boundary marking $\star$ is on the \emph{top left corner}.

\begin{defn}
A $*$-planar algebra $\cP_\bullet$ is called a \emph{factor planar algebra} if it satisfies the following axioms:
\begin{itemize}
\item
(finite dimensional) $\dim(\cP_n)<\infty$ for all $n\geq 0$.
\item
(connected) $\cP_0 \cong \bbC$ via the map that sends the empty diagram to $1_\cC$.
\item
(involutive)
For every tangle $T$ with $r$ input disks, if $T^*$ is the reflected tangle, then $T^*(\xi_1^*, \cdots, \xi_r^*) = T(\xi_1, \dots, \xi_r)^*$.
\item
(positive) for all $n\geq 0$, the following map $\langle \cdot, \cdot\rangle : \cP_n \times \cP_n \to \cP_0 \cong \bbC$ is a positive definite inner product:
$$
\langle x, y\rangle
=
\begin{tikzpicture}[baseline=-.1cm]
	\draw (0,0) -- (1,0);
	\roundNbox{unshaded}{(0,0)}{.3}{0}{0}{$x$}
	\roundNbox{unshaded}{(1,0)}{.3}{0}{0}{$y^*$}
	\node at (.5,.15) {\scriptsize{$n$}};
\end{tikzpicture}
\,.
$$
\item
(spherical) for all $x\in \cP_{2n}$, 
$$
\tr(x) = 
\begin{tikzpicture}[baseline=-.1cm]
	\draw (0,.3) arc (180:0:.3cm) -- (.6,-.3) arc (0:-180:.3cm);
	\roundNbox{unshaded}{(0,0)}{.3}{0}{0}{$x$}
	\node at (-.15,-.5) {\scriptsize{$n$}};
	\node at (-.15,.5) {\scriptsize{$n$}};
\end{tikzpicture}
=
\begin{tikzpicture}[baseline=-.1cm]
	\draw (0,.3) arc (0:180:.3cm) -- (-.6,-.3) arc (-180:0:.3cm);
	\roundNbox{unshaded}{(0,0)}{.3}{0}{0}{$x$}
	\node at (.15,-.5) {\scriptsize{$n$}};
	\node at (.15,.5) {\scriptsize{$n$}};
\end{tikzpicture}
\,.
$$ 
\end{itemize}
\end{defn}

Let $\cP_\bullet$ be a fixed factor planar algebra, and let $\cT\cL\cJ_\bullet$ be its canonical Temperley-Lieb-Jones planar sub-algebra.
Let $\cT\cL\cJ_\bullet \subseteq \cQ_\bullet \subseteq \cP_\bullet$ be any intermediate factor planar subalgebra.
When the loop parameter $\delta>1$, GJS showed how to construct a factor ${\rm II}_1$ factor from $\cQ_\bullet$.
Using Walker's filtered/orthogonal approach \cite{MR2645882}, we form a graded algebra by $\Gr(\cQ_\bullet) = \bigoplus_{n\geq 0} \cQ_n$, together with a filtered multiplication.
Drawing $\cQ_n$ as $n$-boxes whose string emanate \emph{downward} (this is opposite to the usual convention!), the Bacher-Walker product is given by 
$$
(x \in\cP_m) \star (y\in \cP_n) 
=
\sum_{j=0}^{\min\{m,n\}}
\begin{tikzpicture}[baseline=-.4cm]
	\draw (-.15,-.3) -- (-.15,-.8);
	\draw (1.15,-.3) -- (1.15,-.8);
	\draw (.15,-.3) arc (-180:0:.35cm);
	\roundNbox{unshaded}{(0,0)}{.3}{0}{0}{$x$}
	\roundNbox{unshaded}{(1,0)}{.3}{0}{0}{$y$}
	\node at (.5,-.4) {\scriptsize{$j$}};
	\node at (-.6,-.6) {\scriptsize{$m-j$}};
	\node at (1.6,-.6) {\scriptsize{$n-j$}};
\end{tikzpicture}.
$$
There is an obvious $*$-structure, and the unit is given by the empty diagram.
The canonical faithful trace on $\Gr(Q_\bullet)$ is given by $\tr(x) = \delta_{x\in \cQ_0} x$, since $\cQ_0 = \cT\cL\cJ_0 = \bbC$.
This trace turns out to be positive, and the left regular action of $\Gr(Q_\bullet)$ is bounded.
We let $M(\cQ_\bullet)$ be the von Neumann algebra generated by $\Gr(Q_\bullet)$ in the GNS representation of $\tr$.

In \cite{MR2645882}, they showed each $M(\cQ_\bullet)$ is a factor by showing that the von Neumann $A=\{\cap\}''$ is a MASA.
This means $A'\cap M(\cP_\bullet) = A$, and $N:=M(\cT\cL\cJ_\bullet) \subseteq M(P_\bullet) =: M$ is an irreducible ${\rm II}_1$ subfactor.
We now show that this irreducible inclusion is discrete by showing it is a realization of a connected W*-algebra object.

Let $\cC(\cQ_\bullet)=\Proj(\cQ_\bullet)$ be the rigid C*-tensor category of projections of $\cQ_\bullet$ \cite{MR2559686,MR3405915}.
Recall that $\cC(\cQ_\bullet)$ is the idempotent completion of the additive envelope of the category whose objects are orthogonal projections in $\cQ_\bullet$ and whose morphism space from $p\in\cQ_{2m}$ to $q\in \cQ_{2n}$ is given by $q \cQ_{m+n} p$.
Composition is given by vertical concatenation of diagrams.
The tensor product is given by horizontal juxtaposition:
$$
(p \in \cP_{2n} \otimes (q\in \cP_{2n}) 
=
\begin{tikzpicture}[baseline=-.1cm]
	\draw (0,-.6) -- (0,.6);
	\draw (1,-.6) -- (1,.6);
	\roundNbox{unshaded}{(0,0)}{.3}{0}{0}{$p$}
	\roundNbox{unshaded}{(1,0)}{.3}{0}{0}{$q$}
	\node at (.2,-.5) {\scriptsize{$m$}};
	\node at (.2,.5) {\scriptsize{$m$}};
	\node at (1.15,-.5) {\scriptsize{$n$}};
	\node at (1.15,.5) {\scriptsize{$n$}};
\end{tikzpicture}\,.
$$
The dagger structure is given by the adjoint in $\cQ_\bullet$, and the conjugate is given by $\overline{x} = (x^*)^{\vee}$, where $(\cdot)^{\vee}$ is rotation by $\pi$.

Let $\cC = \cC(\Proj(\cT\cL\cJ_\bullet))$ and $\cD = \cC(\Proj(\cP_\bullet))$.
The inclusion $\cT\cL\cJ_\bullet \subset \cP_\bullet$ gives natural \emph{non-full} dimension preserving inclusion of rigid C*-tensor categories $\iota:\cC \hookrightarrow \cD$.
Thus $\cD$ becomes a left $\cC$-module category via $\iota$.
We get a connected W*-algebra object $\bfA \in \Vec(\cC)$ by setting
$$
\bfA(p \in \cP_{2m}) 
=
\cD(p,1_\cC)
:= 
\cP_mp=
\set{
\,\,
\begin{tikzpicture}[baseline=.2cm]
	\draw (0,-.6) -- (0,1);
	\roundNbox{unshaded}{(0,0)}{.3}{0}{0}{$p$}
	\roundNbox{unshaded}{(0,1)}{.3}{0}{0}{$x$}
	\node at (.2,-.5) {\scriptsize{$m$}};
	\node at (.2,.5) {\scriptsize{$m$}};
\end{tikzpicture}
\,\,
}{\,x\in \cP_{m}}.
$$
Clearly $\bfA(1_\cC) = \cP_0 = \bbC$, and $\dim(\bfA(p))<\infty$ for all $p\in \cC$.
The unit map is the empty diagram, and the $*$-structure and multiplication are given as follows: 
$$
j_p\left(
\begin{tikzpicture}[baseline=.2cm]
	\draw (0,-.6) -- (0,1);
	\roundNbox{unshaded}{(0,0)}{.3}{0}{0}{$p$}
	\roundNbox{unshaded}{(0,1)}{.3}{0}{0}{$x$}
	\node at (.2,-.5) {\scriptsize{$m$}};
	\node at (.2,.5) {\scriptsize{$m$}};
\end{tikzpicture}
\right)
=
\begin{tikzpicture}[baseline=.2cm]
	\draw (0,-.6) -- (0,1);
	\roundNbox{unshaded}{(0,0)}{.3}{0}{0}{$\overline{p}$}
	\roundNbox{unshaded}{(0,1)}{.3}{0}{0}{$\overline{x}$}
	\node at (.2,-.5) {\scriptsize{$m$}};
	\node at (.2,.5) {\scriptsize{$m$}};
\end{tikzpicture}
\qquad\qquad
\mu_{p,q}\left(
\begin{tikzpicture}[baseline=.2cm]
	\draw (0,-.6) -- (0,1);
	\roundNbox{unshaded}{(0,0)}{.3}{0}{0}{$p$}
	\roundNbox{unshaded}{(0,1)}{.3}{0}{0}{$x$}
	\node at (.2,-.5) {\scriptsize{$m$}};
	\node at (.2,.5) {\scriptsize{$m$}};
\end{tikzpicture}
\otimes
\begin{tikzpicture}[baseline=.2cm]
	\draw (0,-.6) -- (0,1);
	\roundNbox{unshaded}{(0,0)}{.3}{0}{0}{$q$}
	\roundNbox{unshaded}{(0,1)}{.3}{0}{0}{$y$}
	\node at (.2,-.5) {\scriptsize{$n$}};
	\node at (.2,.5) {\scriptsize{$n$}};
\end{tikzpicture}
\right)
=
\begin{tikzpicture}[baseline=.2cm]
	\draw (0,-.6) -- (0,1);
	\draw (1,-.6) -- (1,1);
	\roundNbox{unshaded}{(0,0)}{.3}{0}{0}{$p$}
	\roundNbox{unshaded}{(0,1)}{.3}{0}{0}{$x$}
	\roundNbox{unshaded}{(1,0)}{.3}{0}{0}{$q$}
	\roundNbox{unshaded}{(1,1)}{.3}{0}{0}{$y$}
	\node at (.2,-.5) {\scriptsize{$m$}};
	\node at (.2,.5) {\scriptsize{$m$}};
	\node at (1.2,-.5) {\scriptsize{$n$}};
	\node at (1.2,.5) {\scriptsize{$n$}};
\end{tikzpicture}
\,.
$$
To show $\bfA$ is a W*-algebra object, we note that $\bfA$ is exactly the $*$-algebra corresponding to the cyclic $\cC$-module W*-category $(\cD, 1_\cC)$.
Moreover, $\bfA$ is tracial, since for all $x,y\in \bfA(p\in \cP_{2m})=\cP_mp$,
$$
{}_p\langle x, y\rangle
=
\begin{tikzpicture}[baseline=.2cm]
	\draw (0,0) -- (0,1);
	\draw (1,0) -- (1,1);
	\draw (0,-.3) arc (-180:0:.5cm);
	\roundNbox{unshaded}{(0,0)}{.3}{0}{0}{$p$}
	\roundNbox{unshaded}{(0,1)}{.3}{0}{0}{$x$}
	\roundNbox{unshaded}{(1,0)}{.3}{0}{0}{$\overline{p}$}
	\roundNbox{unshaded}{(1,1)}{.3}{0}{0}{$\overline{y}$}
	\node at (-.2,-.5) {\scriptsize{$m$}};
	\node at (-.2,.5) {\scriptsize{$m$}};
	\node at (1.2,-.5) {\scriptsize{$m$}};
	\node at (1.2,.5) {\scriptsize{$m$}};
\end{tikzpicture}
=
\begin{tikzpicture}[baseline=.2cm, xscale=-1]
	\draw (0,0) -- (0,1);
	\draw (1,0) -- (1,1);
	\draw (0,-.3) arc (-180:0:.5cm);
	\roundNbox{unshaded}{(0,0)}{.3}{0}{0}{$p$}
	\roundNbox{unshaded}{(0,1)}{.3}{0}{0}{$x$}
	\roundNbox{unshaded}{(1,0)}{.3}{0}{0}{$\overline{p}$}
	\roundNbox{unshaded}{(1,1)}{.3}{0}{0}{$\overline{y}$}
	\node at (-.2,-.5) {\scriptsize{$m$}};
	\node at (-.2,.5) {\scriptsize{$m$}};
	\node at (1.2,-.5) {\scriptsize{$m$}};
	\node at (1.2,.5) {\scriptsize{$m$}};
\end{tikzpicture}
=
\langle y|x\rangle_p.
$$

We get a fully-faithful bi-involutive representation $\bfH:\cC \to \spbfBim(N)$ by defining for $p\in \cP_{2m}$ the bimodule $\bfH(p)$ to be the closure in $\|\cdot\|_2$ of the following inner product space:
$$
\bigoplus_{n\geq 0} p\cT\cL\cJ_{n+m}
=
\bigoplus_{n\geq 0}
\set{
\,\,
\begin{tikzpicture}[baseline=.3cm]
	\draw (0,-.5) -- (0,1.3);
	\roundNbox{unshaded}{(0,0)}{.25}{0}{0}{$s$}
	\roundNbox{unshaded}{(0,.8)}{.25}{0}{0}{$p$}
	\node at (.2,-.4) {\scriptsize{$n$}};
	\node at (.2,.4) {\scriptsize{$m$}};
	\node at (.2,1.2) {\scriptsize{$m$}};
\end{tikzpicture}
\,\,
}{\,s\in \cT\cL\cJ_{n}}
\qquad
\qquad
\left\langle 
\,\,
\begin{tikzpicture}[baseline=.3cm]
	\draw (0,-.5) -- (0,1.3);
	\roundNbox{unshaded}{(0,0)}{.25}{0}{0}{$s$}
	\roundNbox{unshaded}{(0,.8)}{.25}{0}{0}{$p$}
	\node at (.2,-.4) {\scriptsize{$n$}};
	\node at (.2,.4) {\scriptsize{$m$}};
	\node at (.2,1.2) {\scriptsize{$m$}};
\end{tikzpicture}
, 
\begin{tikzpicture}[baseline=.3cm]
	\draw (0,-.5) -- (0,1.3);
	\roundNbox{unshaded}{(0,0)}{.25}{0}{0}{$t$}
	\roundNbox{unshaded}{(0,.8)}{.25}{0}{0}{$p$}
	\node at (.2,-.4) {\scriptsize{$k$}};
	\node at (.2,.4) {\scriptsize{$m$}};
	\node at (.2,1.2) {\scriptsize{$m$}};
\end{tikzpicture}
\right\rangle 
=
\begin{tikzpicture}[baseline=.3cm]
	\draw (0,-.5) -- (0,.8);
	\draw (.8,-.5) -- (.8,.8);
	\draw (0,1.05) arc (180:0:.4cm);
	\roundNbox{unshaded}{(0,0)}{.25}{0}{0}{$s$}
	\roundNbox{unshaded}{(0,.8)}{.25}{0}{0}{$p$}
	\roundNbox{unshaded}{(.8,0)}{.25}{0}{0}{$\overline{t}$}
	\roundNbox{unshaded}{(.8,.8)}{.25}{0}{0}{$\overline{p}$}
	\node at (.2,-.4) {\scriptsize{$n$}};
	\node at (.2,.4) {\scriptsize{$m$}};
	\node at (-.2,1.2) {\scriptsize{$m$}};
	\node at (1,-.4) {\scriptsize{$k$}};
	\node at (1,.4) {\scriptsize{$m$}};
\end{tikzpicture}
\,.
$$
We see that $\bfH(p)$ naturally forms an $N-N$ bimodule where the action is given by the Bacher-Walker product, where all the terms $i, j, (k-i), (\ell-j), (n-i-j)$ must be non-negative:
$$
(x \in\cT\cL\cJ_k) 
\vartriangleright 
\begin{tikzpicture}[baseline=.3cm]
	\draw (0,-.5) -- (0,1.3);
	\roundNbox{unshaded}{(0,0)}{.25}{0}{0}{$s$}
	\roundNbox{unshaded}{(0,.8)}{.25}{0}{0}{$p$}
	\node at (.2,-.4) {\scriptsize{$n$}};
	\node at (.2,.4) {\scriptsize{$m$}};
	\node at (.2,1.2) {\scriptsize{$m$}};
\end{tikzpicture}
\vartriangleleft
(y \in \cT\cL\cJ_\ell)
:=
\sum_{i,j}
\begin{tikzpicture}[baseline=-.1cm]
	\draw (-.15,-.3) -- (-.15,-.8);
	\draw (1,1.4) -- (1,-.8);
	\draw (.15,-.3) arc (-180:0:.35cm);
	\draw (1.15,-.3) arc (-180:0:.35cm);
	\draw (2.15,-.3) -- (2.15,-.8);
	\roundNbox{unshaded}{(0,0)}{.3}{0}{0}{$x$}
	\roundNbox{unshaded}{(1,0)}{.3}{0}{0}{$s$}
	\roundNbox{unshaded}{(1,.9)}{.25}{0}{0}{$p$}
	\roundNbox{unshaded}{(2,0)}{.3}{0}{0}{$y$}
	\node at (.5,-.4) {\scriptsize{$i$}};
	\node at (1.5,-.4) {\scriptsize{$j$}};
	\node at (-.15,-1) {\scriptsize{$k-i$}};
	\node at (1,-1) {\scriptsize{$n-i-j$}};
	\node at (2.15,-1) {\scriptsize{$\ell-j$}};
	\node at (1.2,.5) {\scriptsize{$m$}};
	\node at (1.2,1.3) {\scriptsize{$m$}};
\end{tikzpicture} 
.
$$

Finally, we show that the von Neumann realization $|\bfA|_\bfH$ is isomorphic to $M$.
The evaluation map to $\Gr(\cP_\bullet)$ from the following subset of the algebraic realization $|\bfA|_\bfH^\circ $ is onto:
\begin{align*}
\bigoplus_{a\in\Irr(\cC)} \bfH(a)^\circ \otimes \bfA(a) 
\supset
\bigoplus_{m\geq 0} \bigoplus_{n\geq 0}\jw{m}\cT\cL\cJ_{n+m} \otimes \cP_m \jw{m}
&\to 
\bigoplus_{m\geq 0}\cP_m
=
\Gr(\cP_\bullet)
\\
\begin{tikzpicture}[baseline=.3cm]
	\draw (0,-.5) -- (0,1.3);
	\roundNbox{unshaded}{(0,0)}{.25}{0}{0}{$s$}
	\roundNbox{unshaded}{(0,.8)}{.25}{.2}{.2}{$\jw{m}$}
	\node at (.2,-.4) {\scriptsize{$n$}};
	\node at (.2,.4) {\scriptsize{$m$}};
	\node at (.2,1.2) {\scriptsize{$m$}};
\end{tikzpicture}
\otimes
\begin{tikzpicture}[baseline=.2cm]
	\draw (0,-.6) -- (0,1);
	\roundNbox{unshaded}{(0,0)}{.3}{.2}{.2}{$\jw{m}$}
	\roundNbox{unshaded}{(0,1)}{.3}{0}{0}{$x$}
	\node at (.2,-.5) {\scriptsize{$m$}};
	\node at (.2,.5) {\scriptsize{$m$}};
\end{tikzpicture}
&\mapsto
\begin{tikzpicture}[baseline=.2cm]
	\draw (0,-1.6) -- (0,1);
	\roundNbox{unshaded}{(0,-1)}{.3}{0}{0}{$s$}
	\roundNbox{unshaded}{(0,0)}{.3}{.2}{.2}{$\jw{m}$}
	\roundNbox{unshaded}{(0,1)}{.3}{0}{0}{$x$}
	\node at (.2,-.5) {\scriptsize{$m$}};
	\node at (.2,-1.5) {\scriptsize{$n$}};
	\node at (.2,.5) {\scriptsize{$m$}};
\end{tikzpicture}
\,.
\end{align*}
This map clearly preserves the canonical states, and the subset on the left hand side is dense in $L^2|\bfA|_\bfH$.
Since $\bfA$ is tracial, $|\bfA|_\bfH$ is a ${\rm II}_1$ factor by Proposition \ref{prop:TracialRealization}.
We conclude that the evaluation map extends to an $N-N$ bilinear unitary intertwiner $L^2|\bfA|_\bfH \to L^2(M)$ which takes the left action of $|\bfA|_\bfH$ to the left action of $M$.

\subsection{Discrete quantum groups}
\label{sec:QuantumGroups}

Let $\bbG$ be a discrete quantum group.
Following \cite{MR3204665}, by Tannaka-Krein duality, we view $\bbG$ as a pair $(\bfF,\cC)$, where $\cC$ is a rigid C*-tensor category and $\bfF:\cC\rightarrow \Hilb$ is a dagger tensor functor, with tensorator isomorphisms $\mu_{a,b}:\bfF(a)\otimes \bfF(b)\rightarrow \bfF(a\otimes b)$.
Here we think of $\cC$ as $\Rep(\bbG)$, and $\bfF$ is the forgetful functor.  
In this situation, $\fdHilb$ is a semi-simple $\cC$-module W*-category.  

Choosing our basepoint $m=\bbC$ and applying the equivalence of categories between cyclic $\cC$-module W*-categories and W*-algebra objects \cite[Thm.~3.24]{MR3687214}, we get a connected W*-algebra object $\bfG\in\Vec(\cC)$ called the \textit{quantum group algebra object}, discussed in \cite[Ex.~5.35]{MR3687214}.  
We compute that the multiplication and $*$-structure of $\bfG$ are given as follows.
For $f\in \bfG(a)=\fdHilb(\bfF(a), \bbC)$ and $g\in \bfG(b)=\fdHilb(\bfF(b), \bbC)$,
\begin{align*}
\mu_{a,b}^\bfG(f\otimes g) 
&= 
(f\otimes g) \circ \mu_{a,b}^\bfF \in \fdHilb(\bfF(a\otimes b), \bbC) = \bfG(a\otimes b),
\\
j_a(f) 
&=
\bfF(\ev_a) \circ \mu_{\overline{a},a}^\bfF \circ (f^* \otimes \id_{\bfF(a)}) \in \fdHilb(\bfF(\overline{a}), \bbC) = \bfG(\overline{a}).
\end{align*}
By \eqref{eq:InnerProductsInHilbC}, we have ${}_a\langle f,g\rangle = f\circ g^* \in \fdHilb(\bbC, \bbC) = \bfG(1_\cC) = \bbC$.

\begin{defn} 
The discrete quantum group $\bbG=(\cC, \bfF)$ is called \textit{Kac-type} if $\bfF$ is dimension preserving.
\end{defn}

\begin{prop}
The discrete quantum group $\bbG$ is Kac-type if and only if $\bbG$ is tracial.
\end{prop}
\begin{proof}
The proof is similar to the proof of Proposition \ref{prop:DimensionBound}.
Choose an orthonormal basis $\{e_i\}$ for $\bfG(a)$ with respect to the \emph{left} inner product ${}_a\langle \cdot,\cdot\rangle$, and we calculate
\begin{align*}
\Tr(\Delta_a)
&=
\sum 
{}_a\langle \Delta_a e_i, e_i\rangle
=
\sum 
{}_a\langle S_a e_i, S_a e_i\rangle
\\&=
\sum
\bfF(\ev_a) \circ \mu_{\overline{a},a} \circ (\id_{\bfF(\overline{a})} \otimes (e_i^*\circ e_i)) \circ \mu_{\overline{a},a}^* \bfF(\ev_a^*)
\\&=
\bfF(\ev_a\circ \ev_a^*)
=
d_a.
\end{align*}
Similarly, by the same computation, we have 
$$
\Tr(\Delta_a^{-1}) 
= 
\Tr(S_{\overline{a}} S_{\overline{a}}^*) 
= 
\Tr(S_{\overline{a}}^*S_{\overline{a}} )
=
\Tr(\Delta_{\overline{a}})
=
d_{\overline{a}}
=
d_a.
$$
Thus if $\bfG$ is tracial, we have $d_a = \Tr(\Delta_a) = \dim(\bfG(a)) = \dim(\bfF(a))$, and $\bbG$ is Kac-type.

Conversely, suppose $\bbG$ is Kac-type so that 
$$
\dim(\bfG(a)) 
= 
\dim(\bfF(a))
=
d_a 
=
\Tr(\Delta_a) 
=
\Tr(\Delta_a^{-1}).
$$
Set $n = \dim(\bfG(a))$, and let $\{\lambda_{i}\}^{n}_{i=1}$ denote the set of eigenvalues of $\Delta_{a}$.
Then $\sum^{n}_{i=1} \lambda_{i}=n= \sum^{n}_{i=1} \lambda^{-1}_{i}$, and thus
$$ 
\sum^{n}_{i=1} \lambda_{i} +\lambda^{-1}_{i}
=
2n.
$$
But notice that $\lambda_{i} +\lambda^{-1}_{i}\ge 2$, and thus $\lambda_{i}=1$ for all $i$.
We conclude $\Delta_a=\id$, and thus $\bfG$ is tracial.
\end{proof}

Thus for every non Kac-type discrete quantum group, we can pick a fully faithful bi-involutive representation $\bfH : \cC \to \spbfBim(L\bbF_\infty)$, and we get an extremal irreducible discrete inclusion $(L\bbF_\infty\subseteq |\bfG|_\bfH, E)$ where $|\bfG|_\bfH$ is type ${\rm III}$ by Corollary \ref{cor:IrreducibleInclusionIIorIII} and Proposition \ref{prop:TracialRealization}.
We now list some examples of discrete quantum groups together with their modular spectra.

\begin{exs}
\mbox{}
\begin{enumerate}[(1)]
\item
Associated to a compact Lie group $G$ and a real number $q\ne0$, one can construct the Drinfeld-Jimbo quantum groups $U_{q}(G)$, which carry the structure of a Hopf $*$-algebra.  
In this setting, the finite dimensional type ${\rm I}$ $*$-representations of this Hopf algebra form a rigid C*-tensor category, which has a natural forgetful functor to Hilbert spaces. 
For a detailed exposition, see  \cite{MR3204665}.  
In this situation, the modular spectrum of the corresponding algebra object is easily seen to be the closure of the subgroup of $\bbR^{+}$ generated by the eigenvalues of $\pi_{\omega_i}(K_{2\rho})$, which can be read off from the Cartan matrix.  
Here, the $\omega_{i}$ are fundamental weights, the $\pi_{\omega_i}$ are the corresponding irreducible representations, and $\rho$ is the sum of fundamental weights (see \cite[\S2.4]{MR3204665} for further details).

\item
The free unitary quantum groups $A_{u}(F)$, \cite{MR1316765} are constructed from a matrix $F\in GL_{n}(\bbC)$ satisfying $\Tr(F^{*}F)=\Tr(F^{*}F^{-1})$. 
There is a fundamental $n$-dimensional generating representation $\rho$, and by construction, $\Delta_\rho=F^{*}F$.  
Thus the modular spectrum of the associated algebra object is the closure of the subgroup of $\bbR^{+}$ generated by the eigenvalues of $F^{*}F$.

\item
The free orthogonal quantum groups $A_{o}(F)$, \cite{MR1316765} are also constructed from a matrix $F\in GL_{n}(\bbC)$ such that $\Tr(F^{*}F)=\Tr(F^{*}F^{-1})$, but with the additional condition that $F\overline{F}=\pm1$, where $\overline{F}$ is the matrix obtained by taking the complex conjugate of each entry.  
Again $F^{*}F=\Delta_\rho$, and we get the same modular spectrum as the previous example.
\end{enumerate}
\end{exs}

\subsection{Temperley-Lieb-Jones module categories}
\label{sec:TLJmodules}

First we review the results of \cite{MR3420332}, which classify semi-simple module W*-categories for $\Rep(SU_{q}(2))$, $q\ne 0$, in terms of fair and balanced $\delta$-graphs.
Using the equivalence of categories \cite[Thm.~3.24]{MR3687214}, we get a classification of connected W*-algebra objects.
We focus on $\Rep(SU_{-q}(2))$ for $q>0$, which corresponds to unshaded $\cT\cL\cJ(\delta)$ where $\delta=q+q^{-1}$.
Reinterpreting \cite{MR3420332} into our language, we give an explicit description of the connected algebra object in terms of loops on the graph.
Then any fully faithful bi-involutive embedding of $\cT\cL\cJ(\delta)\to \spbfBim(N)$ for $N$ a ${\rm II}_{1}$ factor, along with a fair and balanced graph and a chosen basepoint, will produce an extremal irreducible discrete subfactor.  
An easily checkable condition on the graph determines when the factor is type ${\rm II}_{1}$ or ${\rm III}_{\lambda}$.

\begin{defn}
For $\delta>0$, a \textit{fair and balanced} $\delta$-\emph{graph} is a locally finite oriented graph $\Gamma=(V, E)$, together with
\begin{itemize}
\item
a \emph{weight function} $W:E\rightarrow (0,\infty)$ satisfying for each vertex $v\in V$, $\sum_{s(e)=v} W(e)=\delta$, and
\item
an involution $e \mapsto \overline{e}$ which interchanges source and target satisfying $W(e)W(\overline{e})=1$ for all $e\in E$.
\end{itemize}
We extend the weight function $W$ and the involution $\overline{\,\cdot\,}$ to loops $\ell=[e_1,\dots, e_n]$ on $\Gamma$ by 
$W(\ell):=W(e_{1})\dots W(e_n)$ and $\overline{\ell}:=[\overline{e}_{n},\dots, \overline{e}_{1}]$.  
\end{defn}

Graphs of the above form classify module categories for $\Rep(SU_{-q}(2))$, $q>0$, where $q+q^{-1}=\delta$, and the simple objects are classified by the vertices.  
(With the additional restriction that the number of self-loops of $\Gamma$ is even for each vertex, these also classify module categories for $q<0$.)
Technically speaking, the actual involution is not part of the data; the existence of an involution is sufficient to construct a module category, and two different involutions for the same weight function give equivalent module W*-categories.
For convenience, we fix a chosen involution.

Now, given such a graph $\Gamma$ and a vertex $v\in V$, we explicitly construct the corresponding connected W*-algebra object.
To define a functor $\bfA_{v}: \cT\cL\cJ(\delta)^{\op}\rightarrow \Vec$, it suffices to assign a vector space to each object $\rho^{\otimes n}$, where $\rho=\jw{1}$ is the strand, 
and to give the action of cups and caps.
Denoting the cups and caps by
$$
\cup_i =
\begin{tikzpicture}[baseline=-.1cm]
	\draw (0,-.4) -- (0,.4);
	\node at (.3,0) {\scriptsize{$\cdots$}};
	\draw (.6,-.4) -- (.6,.4);
	\draw (.8,.4) arc (-180:0:.2cm);
	\draw (1.4,-.4) -- (1.4,.4);
	\node at (1.7,0) {\scriptsize{$\cdots$}};
	\draw (2,-.4) -- (2,.4);
	\node at (0,.6) {\scriptsize{$1$}};
	\node at (0,-.6) {\scriptsize{$1$}};
	\node at (.8,.6) {\scriptsize{$i$}};
	\node at (1.4,-.6) {\scriptsize{$i$}};
	\node at (2,-.6) {\scriptsize{$n$}};
	\node at (2,.6) {\scriptsize{$n+2$}};
\end{tikzpicture}
\qquad\qquad
\cap_i =
\begin{tikzpicture}[baseline=-.1cm, yscale =-1]
	\draw (0,-.4) -- (0,.4);
	\node at (.3,0) {\scriptsize{$\cdots$}};
	\draw (.6,-.4) -- (.6,.4);
	\draw (.8,.4) arc (-180:0:.2cm);
	\draw (1.4,-.4) -- (1.4,.4);
	\node at (1.7,0) {\scriptsize{$\cdots$}};
	\draw (2,-.4) -- (2,.4);
	\node at (0,.6) {\scriptsize{$1$}};
	\node at (0,-.6) {\scriptsize{$1$}};
	\node at (.8,.6) {\scriptsize{$i$}};
	\node at (1.4,-.6) {\scriptsize{$i$}};
	\node at (2,-.6) {\scriptsize{$n-2$}};
	\node at (2,.6) {\scriptsize{$n$}};
\end{tikzpicture}
\,,
$$
we must give maps 
$\bfA_v(\cup_{i}): \bfA_{v}(\rho^{\otimes n})\rightarrow \bfA_{v}(\rho^{\otimes n-2})$ 
and
$\bfA_v(\cap_{i}): \bfA_{v}(\rho^{\otimes n})\rightarrow \bfA_{v}(\rho^{\otimes n+2})$
which satisfy the Temperley-Lieb-Jones relations.

\begin{defn}
We define $\bfA_{v}(\rho^{\otimes n})$ to be the complex vector space with basis consisting of loops starting and ending at $v$. For a loop $\ell=[e_1, e_2, \dots e_n]$, we define 
\begin{align*}
\bfA_v(\cup_{i})(\ell)
&=
\delta_{e_{i}=\overline{e}_{i+1}}W(e_{i})^{\frac{1}{2}} [e_{1},\dots, e_{i-1},e_{i+2},\dots e_{n}]
\\
\bfA_v(\cap_{i})(\ell)
&=
\sum_{s(e)=t(e_i)}W(e)^{\frac{1}{2}} [e_{1},\dots, e_{i}, e, \overline{e}, e_{i+1},\dots e_{n}].
\end{align*}
\end{defn}

It is easy to check that these maps satisfy the Temperley-Lieb-Jones relations, and thus we get an object $\bfA_v\in \Vec(\cT\cL\cJ(\delta))$.  
The algebra map $\mu: \bfA_{v}(\rho^{\otimes n})\otimes \bfA_{v}(\rho^{\otimes m})\rightarrow \bfA_{v}(\rho^{\otimes n+m})$ is given on the basis by concatenating loops.  
Thus $\bfA_v$ is a connected algebra by definition, since only the empty loop has length zero.
The $*$-structure is the conjugate linear map defined by $\ell\mapsto W(\overline{\ell})^{1/2}\overline{\ell}$.
The fact that these are W*-algebra objects simply follows from the fact that they arise from module W*-categories via \cite{MR3420332}, but it is also easy to check directly.  

\begin{prop} 
\label{prop:SpectralInvariantOfGraphAlgebra}
The closure of
$\set{W(\ell)}{\ell \text{ is a loop starting at }v}$ 
in $[0,\infty)$ is equal to the spectral invariant $S(\bfA_{v})$.
Thus $\bfA_v$ is tracial if and only if $W(\ell)=1$ for all loops $\ell$.
\end{prop}
\begin{proof}
First note that to compute $S(\bfA_{v})$, it suffices to compute the spectrum of $\Delta_n:=\Delta_{\rho^{\otimes n}}$ for each $n\in\bbN$.
Note that $\Delta_{n}$ is diagonal on the loop basis.  
Indeed, for all loops $\ell_1,\ell_2$ starting at $v$,
$$
\langle \ell_{1} | \ell_{2}\rangle_{n}
=
\delta_{ \ell_1=\ell_2}
W(\overline{\ell_2})
=
\delta_{\ell_1=\ell_2}
W(\ell_1)^{-1},
$$
and thus $\{W(\ell)^{1/2} \ell\}$ forms an orthonormal basis of $\bfA_v(n)$ with respect to the right inner product $\langle \cdot|\cdot\rangle_n$.
Now we calculate
\begin{align*}
\langle W(\ell_1)^{1/2}\ell_{1}|\Delta_n W(\ell_2)^{1/2}\ell_{2}\rangle_{n}
&=
W(\ell_1)^{1/2}W(\ell_2)^{1/2}
{}_n\langle \ell_{2}, \ell_{1}\rangle
\\&=
\delta_{ \ell_1=\ell_2}
W(\ell_1)^{1/2}W(\ell_2)^{1/2}
W(\ell_1)^{1/2}W(\overline{\ell_{2}})^{\frac{1}{2}}
\\&=
\delta_{ \ell_1=\ell_2}
W(\ell_1).
\end{align*}
This means for all loops $\ell$, $\Delta_n(\ell) = W(\ell) \ell$, and $S(\bfA_v)$ is as claimed.
\end{proof}

\begin{ex}
The $4$-regular graph with $n$ vertices
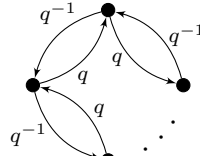
\begin{tikzpicture}[baseline=-.1cm]
	\node[vertex, label=right:{\scriptsize{$v$}}] (a) at (0:1cm) {};
	\node[vertex] (b) at (90:1cm) {};
	\node[vertex] (c) at (180:1cm) {};
	\node[vertex] (d) at (270:1cm) {};
	\draw[edge] (a) to node[right] {\scriptsize{$q^{-1}$}} (b);
	\draw[edge] (b) to node[left] {\scriptsize{$q$}} (a);
	\draw[edge] (b) to node[above] {\scriptsize{$q^{-1}$}} (c);
	\draw[edge] (c) to node[below] {\scriptsize{$q$}} (b);
	\draw[edge] (c) to node[left] {\scriptsize{$q^{-1}$}} (d);
	\draw[edge] (d) to node[right] {\scriptsize{$q$}} (c);
	\node at (-30:1cm) {$\cdot$};
	\node at (-45:1cm) {$\cdot$};
	\node at (-60:1cm) {$\cdot$};
\end{tikzpicture}
satisfies $S(\bfA_v) = q^{n\bbZ}$.
Using our embedding $\bfH:\cT\cL\cJ(\delta) \to \spbfBim(L\bbF_\infty)$ from Section \ref{sec:PlanarAlgebras}, when $q>1$, we get an extremal irreducible discrete inclusion $L\bbF_\infty\subseteq |\bfA_v|_\bfH$ where $|\bfA_v|_\bfH$ has type ${\rm III}_{q^{-n}}$.
It would be interesting to compare these discrete inclusions with the similar inclusions from \cite{MR1674759}.
\end{ex}

\begin{rem}
As explained in \cite{MR3420332}, the structure of a $\cT\cL\cJ$-module W*-category given by a fair and balanced $\delta$-graph is very similar to the structure of Jones' graph planar algebra \cite{MR1865703}.
The main difference is the weight function on the edges for the former rather than the Frobenius-Perron dimension function on the vertices for the latter.
Given a graph planar algebra with dimension function $D$, we get a fair and balanced $\delta$-graph by defining the weight function by $W(e) = D(t(e))D(s(e))^{-1}$.
This automatically implies that for any loop $\ell$, $W(\ell) = 1$, and thus $\bfA_v$ will always be tracial.
\end{rem}

\begin{ex}
Following \cite[\S3.2]{MR3420332}, starting with $A\subseteq B$ a finite index ${\rm II}_1$ subfactor with $\delta^2=[B:A]$, the fair and balanced $\delta$-graph associated to the principal graph with $v=\star$ corresponding to $L^2(A)$ gives a connected W*-algebra object $\bfA_\star\in \Vec(\cT\cL\cJ(\delta))$.
This algebra object is $*$-isomorphic to the connected W*-algebra object in $\Vec(\cT\cL\cJ(\delta))$ constructed from the planar algebra $\cP_\bullet$ of $A\subset B$ as described in Section \ref{sec:PlanarAlgebras}.
(For example, we can think of $\cP_\bullet$ as embedded in its graph planar algebra by \cite{MR2812459,gpa}, and cutting down at the vertex $\star$ is an isomorphism of $*$-algebras by \cite[Prop.~4.10]{MR3402358} which commutes with partial traces.)
It is important to note that the associated extremal irreducible discrete inclusion $N\subseteq |\bfA_\star|_\bfH$ of ${\rm II}_1$ factors is distinct from the original subfactor.
Indeed, $N\subseteq |\bfA_\star|_\bfH$ has infinite index for $\delta \geq 2$.
\end{ex}

\section{Applications}

In our final section, we give some applications of our equivalence of categories, including 
standard invariants and subfactor reconstruction for extremal irreducible discrete subfactors,
analytic properties,
and a Galois correspondence in the spirit of \cite{MR1622812,MR2561199}.

\subsection{Standard invariants and subfactor reconstruction}
\label{sec:StandardInvariants}

As a first application of our main Theorem \ref{thm:Main}, we get a well-behaved definition of the standard invariant of an extremal irreducible discrete inclusion, together with a subfactor reconstruction theorem.
We rapidly recall the definition of the standard invariant for finite index ${\rm II}_1$ subfactors \cite{MR1334479,math.QA/9909027}, and the translation into unitary Frobenius algebra objects in rigid C*-tensor categories as in \cite{MR1966524}.

Starting with a finite index ${\rm II}_1$ subfactor $M_0=N\subseteq M = M_1$, we iterate the Jones basic construction to obtain a tower of ${\rm II}_1$ factors $(M_n)_{n\geq 0}$.
The \emph{standard invariant} is the \emph{standard} $\lambda$-\emph{lattice} \cite{MR1334479} of higher relative commutants 
$$
\cP_{n,+}:= M_0'\cap M_n 
\qquad
\qquad
\cP_{n-} := M_1'\cap M_{n+1}
\qquad n\geq 0.
$$
We may also consider the standard invariant as a shaded subfactor planar algebra $\cP_\bullet$ \cite{math.QA/9909027}, or as the unitary Frobenius algebra object $A=L^2(M)$ in the rigid C*-tensor category $\cC$ of bifinite $N-N$ bimodules generated by $L^2(M)$ \cite{MR1966524}.
Equivalently, $\cC$ can be viewed as the category of \emph{even} projections $\cC(\cP_\bullet)$ similar to the discussion in Section \ref{sec:PlanarAlgebras}.
Here, the objects are projections in $\cP_{2n,+}=M_0'\cap M_{2n}$, which correspond to $N-N$ sub-bimodules of $L^2(M_n)$ by \cite{MR1424954}.
The morphisms, composition, and tensor product are defined analogously as in Section \ref{sec:PlanarAlgebras}.
The algebra  $A$ now corresponds to $\id_{\cP_{2,+}}$, and the multiplication and unit are given by
$$
m_A = 
\begin{tikzpicture}[baseline=-.1cm]
	\fill[shaded] (-.4,-.3) .. controls ++(90:.3cm) and ++(270:.3cm) .. (-.1,.3) -- (.1,.3) .. controls ++(270:.3cm) and ++(90:.3cm) .. (.4,-.3) -- (.2,-.3) arc (0:180:.2cm);
	\draw (-.2,-.3) arc (180:0:.2cm);
	\draw (-.4,-.3) .. controls ++(90:.3cm) and ++(270:.3cm) .. (-.1,.3);
	\draw (.4,-.3) .. controls ++(90:.3cm) and ++(270:.3cm) .. (.1,.3);
\end{tikzpicture}
\qquad
\qquad
i_A=
\begin{tikzpicture}[baseline=-.1cm]
	\fill[shaded] (-.2,.2) arc (-180:0:.2cm);
	\draw (-.2,.2) arc (-180:0:.2cm);
\end{tikzpicture}\,.
$$

For an extremal irreducible discrete inclusion $(N\subseteq M, E)$, the bimodule $L^2(M,\phi)$ need not be bifinite, and thus cannot be viewed as an object of the rigid C*-tensor category $\cC$ of bifinite $N-N$ bimodules generated by $L^2(M,\phi)$.
Thus we must necessarily consider a larger W*-algebra object in a larger category.
Our connected W*-algebra objects $\bfA\in \Vec(\cC)$ were designed precisely for this purpose.

\begin{defn}
The \emph{standard invariant} of an extremal irreducible discrete inclusion $(N\subseteq M, E)$ consists of:
\begin{itemize}
\item
the rigid C*-tensor category $\cC_{(N\subseteq M, E)}$ of bifinite $N-N$ bimodules generated by $L^2(M,\phi)$, (the full 
subcategory of bifinite $N-N$ bimodules which appear as summands of $\boxtimes_N^k L^2(M,\phi)$ for some $k\geq 0$), together with
\item
the connected W*-algebra object $\bfA_{(N\subseteq M, E)}\in \Vec(\cC)$ defined by 
$\bfA_{(N\subseteq M, E)}(K) = \Hom_{N-N}(K, L^2(M,\phi)) \cong \Hom_{N-N}(K\boxtimes_N L^2(M,\phi), L^2(M,\phi))$,
which corresponds to the cyclic $\cC$-module W*-subcategory of $N-M$ bimodules $\Bim(N,M)$ generated by the basepoint $L^2(M,\phi)$.
\end{itemize}
We also define an \emph{abstract standard invariant} as a pair $(\cC,\bfA)$ where $\cC$ is a rigid C*-tensor category, and $\bfA\in \Vec(\cC)$ is a connected W*-algebra object such that $\bfA$ \emph{generates} $\cC$, i.e., for every $c\in \cC$, there is an $a\in \cC$ and an $n\in \bbN$ such that $c\preceq a^{\otimes n}$ and $\bfA(a)\neq (0)$.
\end{defn}

\begin{thm}
Given an abstract standard invariant $(\cC, \bfA)$ with $\cC$ a rigid \emph{C*}-tensor category and $\bfA\in \Vec(\cC)$ a connected \emph{W*}-algebra object, there is an extremal irreducible discrete inclusion $(N\subseteq M, E)$ whose standard invariant $(\cC_{(N\subseteq M, E)}, \bfA_{(N\subseteq M, E)})$ is equivalent to $(\cC, \bfA)$, i.e., there is an equivalence of categories $\bfH : \cC \to \cC_{(N\subseteq M, E)}$ and a $*$-algebra natural isomorphism $\kappa^\bfA : \bfA_{(N\subseteq M, E)}\circ \bfH^{\op} \Rightarrow \bfA$.
\end{thm}
\begin{proof}
Let $\bfH: \cC \to \spbfBim(N)$ be a fully faithful representation with $N$ a ${\rm II}_1$ factor.
(Such a representation always exists for $N=L\bbF_\infty$ by \cite{MR2051399,MR3405915}.)
Then $\bfH$ is automatically an equivalence of categories onto its essential image $\bfH(\cC)$.

The inclusion $N\subseteq M:= |\bfA|_\bfH = N\rtimes_\bfH \bfA$ with its canonical conditional expectation $E=E_N$ constructed in Section \ref{sec:vonNeumannRealization} is an extremal irreducible discrete inclusion.
Moreover, we see that $\cC_{(N\subseteq M,E)} = \bfH(\cC)$ by construction.
Notice now that for $c\in \cC$, we have 
$
\bfA_{(N\subseteq M, E)}(\bfH(c))
=
\Hom_{N-N}(\bfH(c), L^2|\bfA|_\bfH)
\cong 
\bfA(c)
$
via the $*$-algebra natural isomorphism $\kappa^\bfA$ from Section \ref{sec:EquivalenceOfCategories}.
Thus $\kappa^A$ may be viewed as a $*$-algebra natural isomorphism $\bfA_{(N\subseteq M, E)}\circ \bfH^{\op}\cong \bfA$.
\end{proof}

We also get a converse result.

\begin{thm}
Suppose $(N\subseteq M, E)$ is an extremal irreducible discrete inclusion with standard invariant $(\cC_{(N\subseteq M, E)}, \bfA_{(N\subseteq M, E)})$.
There is a fully faithful representation $\bfH : \cC_{(N\subseteq M, E)} \to \spbfBim(N)$ such that the extremal irreducible discrete inclusion $N\subseteq P = |\bfA_{(N\subseteq M, E)}|_\bfH$ with expectation $E^P_N$ is isomorphic to $(N\subseteq M, E)$, i.e., there is a $*$-algebra isomorphism $\delta:P \to M$ which fixes $N$ such that $E_N^P = \delta\circ E_N^P = E \circ \delta$.
\end{thm}
\begin{proof}
We define $\bfH$ to be the identity functor on $\cC_{(N\subseteq M, E)}$ which is already a subcategory of $\spbfBim(N)$.
The $*$-algebra isomorphism $\delta: |\bfA_{(N\subseteq M, E)}|_\bfH \to M$ is the isomorphism $\delta_M$ from Section \ref{sec:EquivalenceOfCategories}.
Recall that $\delta$ was defined on $\bfH(K)^\circ \otimes \bfA_{(N\subseteq M, E)}(K)$ for $K\in \cC_{(N\subseteq M, E)}$ by
$\xi\otimes f \mapsto f(\xi) \in M\Omega$.
Since $N\subseteq \bfA_{(N\subseteq M, E)}$ is included as $n\Omega \otimes i^{\bfA_{(N\subseteq M, E)}}$, we see that for all $n\in N$, $\delta(n\Omega\otimes i^{\bfA_{(N\subseteq M, E)}}) = i^{\bfA_{(N\subseteq M, E)}}(n\Omega) = n\Omega \in M\Omega$, and thus $\delta$ is the identity on $N$.
Finally, for every $\xi\otimes f \in \bfH(K)^\circ \otimes \bfA_{(N\subseteq M, E)}(K)$, $E^P_N(\xi\otimes f) = \delta_{K=L^2N} f(\xi)$, which is equal to $(E\circ \delta)(\xi\otimes f)$.
This is enough to conclude $E^P_N =\delta\circ E$.
\end{proof}

As a corollary, we get another axiomatization of the standard invariant of a finite index ${\rm II}_1$ subfactor as a \emph{compact} connected W*-algebra objects in $\cC^{\natural}$ instead of a unitary Frobenius algebra in $\cC$.
Our followup article \cite{1707.02155} provides a simple translation between these types of algebra objects in this setting.

\subsection{Analytic Properties}
\label{sec:AnalyticProperties}

Let $\bfA\in \Vec(\cC)$ be a connected W*-algebra object, and as before, let $L^2(\bfA)\in \Hilb(\cC)$ be the GNS Hilbert space object with respect to the right inner products $\langle\cdot|\cdot\rangle_a$.

\begin{defn}
The \emph{multiplier algebra} of $\bfA$ is the von Neumann algebra
$$
\ell^\infty(\bfA) := \End_{\Hilb(\cC)}(L^2(\bfA)) \cong \bigoplus_{a\in\Irr(\cC)} B(L^{2}(\bfA)(a)).
$$
For an extremal irreducible discrete inclusion $(N\subseteq M, E)$, we define the multiplier algebra of the inclusion to be the von Neumann algebra
$$
\End_{N-N}(L^2(M,\phi)) \cong N'\cap JNJ' = N'\cap \langle M, N\rangle.
$$
\end{defn}

For every ucp morphism $\theta: \bfA \Rightarrow \bfA$, $m_\theta = \bigoplus_{a\in \Irr(\cC)} \theta_a$ extends to a bounded norm one map in $\ell^\infty(\bfA)$ which we call a ucp-multiplier.
Conversely, given a ucp-multiplier $m\in \ell^\infty(\bfA)$, we can recover a canonical ucp morphism $\theta: \bfA \Rightarrow \bfA$ such that $m = m_\theta$.

For an irreducible discrete inclusion $(N\subseteq M, E)$, we define a ucp-multiplier as an $N-N$ bilinear ucp-map $M\to M$ which preserves the canonical state $\phi=\tau\circ E$.

Our equivalence of categories in Theorem \ref{thm:Equivalence} immediately gives us the following corollary.

\begin{cor}
\label{cor:IsomorphismOfMultiplierAlgebras}
Let $\bfA\in \Vec(\cC)$ be a connected \emph{W*}-algebra object.
For any fully faithful bi-involutive  $\bfH: \cC\to \spbfBim(N)$, 
$\ell^\infty(\bfA) \cong  \End_{N-N}(L^2|\bfA|_\bfH)$,
and the equivalence gives a bijective correspondence of ucp-multipliers.
\end{cor}

\begin{defn}
We say a net of ucp-multipliers $(m_\lambda)\subset \ell^\infty(\bfA)$ converges \emph{pointwise} to $m\in \ell^\infty(\bfA)$ if for every $a\in\Irr(\cC)$, the $(m_\lambda)_a\to m_a$ in the finite dimensional von Neumann algebra $B(L^2(\bfA)(a))$.
Notice that $m_\lambda \to m$ pointwise if and only if $m_\lambda \to m$ SOT when we consider $\ell^\infty(\bfA) =\bigoplus_{a\in\Irr(\cC)} B(L^{2}(\bfA)(a)) \subset B(\bigoplus_{a\in\Irr(\cC)} L^2(\bfA)(a))$.
This identification also tells us the meaning of \emph{finite rank} and \emph{compactness} for ucp-multipliers.
\end{defn}

We gave the following definitions of analytic properties for connected W*-algebra objects in \cite[\S5.5]{MR3687214}.

\begin{defn}
\label{def:analytic} 
Let $\bfA\in \Vec(\cC)$ be a connected W*-algebra object.
Then $\bfA$
\begin{enumerate}[(1)]
\item
is \textit{amenable} if there is a net of finite rank ucp-multipliers in $\ell^{\infty}(\bfA)$ converging to the identity pointwise,
\item
has the \textit{Haagerup property} if there is a net of compact ucp-multipliers in $\ell^{\infty}(\bfA)$ converging to the identity pointwise, and
\item
has \textit{property} (T) if every net of ucp multipliers in $\ell^{\infty}(\bfA)$ which converges to the identity pointwise converges in the operator norm.
\end{enumerate}
Similarly, when $(N\subseteq M, E)$ is an irreducible discrete inclusion, we get analogous definitions of analytic properties for $M$ relative to $N$ using ucp-multipliers.
In the case that $M$ is a finite von Neumann algebra, our definitions are equivalent to the standard definitions from \cite{correspondences,MR2215135} by \cite[Lem.~9.19 and 9.20]{1511.07329}.
\end{defn}

Corollary \ref{cor:IsomorphismOfMultiplierAlgebras} immediately implies the following.

\begin{cor}
An extremal irreducible discrete inclusion $(N\subseteq M, E)$ has one of the three analytic properties in Definition \ref{def:analytic} for $M$ relative to $N$ if and only if the underlying connected \emph{W*}-algebra object $\alg{M}$ has the corresponding analytic property.
\end{cor}

We now give a concrete application of the above correspondence.

\begin{defn}
In \cite{MR3406647}, the authors defined a ucp-multiplier of $\cC$ to be an $m\in \ell^\infty(\Irr(\cC))$ such that for every $a,b\in \cC$, the map $m_{a,b}:\cC(a\otimes b, a\otimes b) \to \cC(a\otimes b, a\otimes b)$ defined below is a ucp map.
First, we have an isomorphism of vector spaces
\begin{align*}
\bigoplus_{c\in\Irr(\cC)}
\cC(a, a\otimes c ) \otimes \cC(c\otimes b,b)
&\to
\cC(a\otimes b, a\otimes b)
\\
\alpha\otimes \beta
&\mapsto
(\id_a \otimes \beta)
\circ
(\alpha \otimes \id_b).
\end{align*}
We decompose $\psi \in \cC(a\otimes b, a\otimes b)$ as a sum $\sum_{c\in \Irr(\cC)} \psi_c$, where $\psi_c \in \cC(a, a\otimes c ) \otimes \cC(c\otimes b,b)$ for each $c\in \Irr(\cC)$.
We define $m_{a,b}(\psi) = \sum_{c\in \Irr(\cC)} m(c) \psi_c$.
\end{defn}

By \cite[Lem.~3.7]{MR3406647} and \cite[Ex.~5.33]{MR3687214}, the set of ucp multipliers $m$ on $\cC$ is in bijective correspondence with ucp morphisms $\theta: \bfS \Rightarrow \bfS$ where $\bfS\in \Vec(\cC\boxtimes \cC^{\op})$ is the symmetric enveloping algebra object, which is a connected W*-algebra object. 

\begin{defn}
Let $\bfA\in \Vec(\cC)$ be an arbitrary C*-algebra object, and let $m\in \ell^{\infty}(\Irr(\cC))$ be a ucp-multipler. 
We define a natural transformation $\theta^m:\bfA\Rightarrow \bfA$ given by $\theta^m_{c}:\bfA(c)\rightarrow \bfA(c)$ by $f\mapsto m(c) f$ for $c\in \Irr(\cC)$.
\end{defn}

\begin{prop}
If $m\in \ell^{\infty}(\Irr(\cC))$ is a ucp-multiplier, then $\theta^m:\bfA\Rightarrow \bfA$ is a ucp morphism.
\end{prop}
\begin{proof} 
Let $c\in\cC$. 
We must show that $\theta^m_{\overline{c}\otimes c}$ is positive.
To do so, we work in the cyclic $\cC$-module W*-category $\cM_\bfA$.
For $f\in \bfA(\overline{c}\otimes c)$, we have
$$
\widehat{\theta^m}_{\overline{c}\otimes c}(f^*\circ f)
=
\widehat{\theta^m}_{\overline{c}\otimes c}
\left(
\,\,
\begin{tikzpicture}[baseline = .4cm]
    \draw (.3,1) arc (90:-90:.5cm);
    \draw (.8,.5) -- (1.2,.5);
    \draw (0,0) -- (0,1);
    \draw (0,1.3) arc (0:180:.3cm) -- (-.6,.9) arc (0:-180:.3cm) -- (-1.2,1.7);
    \draw (0,-.3) arc (0:-180:.3cm) -- (-.6,.1) arc (0:180:.3cm) -- (-1.2,-.7);
    \filldraw (.8,.5) circle (.05cm);
    \roundNbox{dashed}{(-.9,.5)}{.4}{.1}{.1}{}
    \roundNbox{unshaded}{(0,0)}{.3}{0}{0}{$f$}
    \roundNbox{unshaded}{(0,1)}{.3}{0}{0}{$f^*$}
    \node at (-1.4,1.5) {\scriptsize{$\mathbf{c}$}};
    \node at (-.8,1.2) {\scriptsize{$\overline{\mathbf{c}}$}};
    \node at (.2,1.5) {\scriptsize{$\mathbf{c}$}};
    \node at (.2,.5) {\scriptsize{$\mathbf{c}$}};
    \node at (1,.7) {\scriptsize{$\bfA$}};
    \node at (.2,-.5) {\scriptsize{$\mathbf{c}$}};
    \node at (-.8,-.2) {\scriptsize{$\overline{\mathbf{c}}$}};
    \node at (-1.4,-.5) {\scriptsize{$\mathbf{c}$}};
\end{tikzpicture}
\right)
=
\begin{tikzpicture}[baseline = .4cm]
    \draw (.3,1) arc (90:-90:.5cm);
    \draw (.8,.5) -- (1.2,.5);
    \draw (0,0) -- (0,1);
    \draw (0,1.3) arc (0:180:.3cm) -- (-.6,.5);
    \draw (0,-.3) arc (0:-180:.3cm) -- (-.6,.5);
    \draw (-1.2,-.7) -- (-1.2,1.7);
    \filldraw (.8,.5) circle (.05cm);
    \roundNbox{unshaded}{(-.9,.5)}{.3}{.2}{.2}{$m(\psi)$}
    \roundNbox{unshaded}{(0,0)}{.3}{0}{0}{$f$}
    \roundNbox{unshaded}{(0,1)}{.3}{0}{0}{$f^*$}
    \node at (-1.4,1.5) {\scriptsize{$\mathbf{c}$}};
    \node at (-.8,1.2) {\scriptsize{$\overline{\mathbf{c}}$}};
    \node at (.2,1.5) {\scriptsize{$\mathbf{c}$}};
    \node at (.2,.5) {\scriptsize{$\mathbf{c}$}};
    \node at (1,.7) {\scriptsize{$\bfA$}};
    \node at (.2,-.5) {\scriptsize{$\mathbf{c}$}};
    \node at (-.8,-.2) {\scriptsize{$\overline{\mathbf{c}}$}};
    \node at (-1.4,-.5) {\scriptsize{$\mathbf{c}$}};
\end{tikzpicture}
$$
where $\psi = \coev_c \circ \coev_c^* \in \cC(c\otimes \overline{c}, c\otimes \overline{c})$.
Since $\psi\geq 0$ and $m\in \ell^\infty(\cC)$ is a ucp-multiplier, $m(\psi) \geq 0$, which immediately implies the right hand side is positive.
\end{proof}

\begin{cor} 
If $\cC$ is amenable or has the Haagerup property, then every connected \emph{W*}-algebra object in $\Vec(\cC)$ is amenable or has the Haagerup property respectively. 
\end{cor}
\begin{proof}
Let $(m_{\lambda})\in \ell^\infty(\cC)$ be a net of ucp-multipliers converging to $\id$ pointwise, and let $\bfA\in \Vec(\cC)$ be a connected W*-algebra object.  
Notice the $m_\lambda$ act by scalar multiplication on each component of $\ell^\infty(\bfA)$.
If $m_{\lambda}$ is finitely supported or $c_{0}$, then the corresponding map $\theta^{m_\lambda} \in \ell^\infty(\bfA)$ is as well, since the spaces $\bfA(c)$ are finite dimensional for all $c\in\Irr(\cC)$.
Finally, $m_\lambda \to m$ pointwise implies $\theta^{m_\lambda} \to \id$ pointwise.
\end{proof}

We can thus determine many analytic properties of an extremal irreducible discrete subfactor by looking at its support category.

\begin{cor} 
Let $(N\subseteq M, E)$ be an extremal irreducible discrete inclusion, and let $\cC$ be the rigid \emph{C*}-tensor category generated by $L^2(M,\phi)$.
If $\cC$ is amenable or has the Haagerup property, then $M$ has the corresponding property relative to $N$.  
\end{cor}
In particular, by \cite{MR1278111}, $\cT\cL\cJ_\bullet$ is amenable for $\delta= 2$, and by \cite{MR3406647}, $\cT\cL\cJ_\bullet$ has the Haagerup property for $\delta\geq 2$.
Thus for any planar algebra $\cP_\bullet$ with parameter $\delta\geq 2$ (respectively $=2$), $M(\cP_\bullet)$ has the Haagerup property (respectively is amenable) relative to $M(\cT\cL\cJ_\bullet)$.

\subsection{Galois correspondence}
\label{sec:Galois}

In the spirit of \cite{MR1622812,MR2561199}, our main Theorem \ref{thm:Equivalence} also gives us a Galois correspondence between intermediate subfactors $N\subseteq P\subseteq M$ for an extremal irreducible discrete inclusion $(N\subseteq M, E)$ and intermediate connected W*-algebra objects $\mathbf{1}\subseteq \bfB \subseteq \bfA$ in $\Vec(\cC)$.
One should compare with \cite[Lem.~3.8]{MR1622812} in the setting of inclusions of infinite factors.

\begin{cor}
The equivalence of categories in Theorem \ref{thm:Equivalence} restricts to an equivalence between the following categories:
\begin{itemize}
\item
The subcategory of $\ConAlg$ whose morphisms are unital $*$-algebra natural transformations, and
\item
the subcategory of $\DisInc_\bfH$ whose morphisms are normal unital $N-N$ biinear $*$-homomorphisms which preserve the canonical state.
\end{itemize}
\end{cor}
\begin{proof}
Given a $*$-algebra natural transformation $\theta: \bfA \Rightarrow \bfB$, $|\theta|_\bfH^\circ$ is a unital $*$-algebra map $N-N$ biinear $*$-homomorphism which preserves the canonical state.
Thus by Remarks \ref{rems:ExtendUCPMaps}, $|\theta|_\bfH : |\bfA|_\bfH \to |\bfB|_\bfH$ is a normal $*$-homomorphism.
Conversely, if $\psi : M \to P$ is such a $*$-homomorphism, it is straightforward to show that $\alg{\psi}: \alg{M} \Rightarrow \alg{P}$ is a $*$-algebra natural transformation. 
\end{proof}

\bibliographystyle{amsalpha}
{\footnotesize{
\bibliography{../../../Documents/research/penneys/bibliography}
}}
\end{document}